\documentclass[12pt,a4paper]{article}
\usepackage{hyperref}
\usepackage{mathtools}
\usepackage[nottoc,numbib]{tocbibind}
\usepackage[english]{babel}
\usepackage[utf8x]{inputenc}
\usepackage{amsfonts,longtable,amssymb,amsmath,latexsym,dsfont,mathabx}
\usepackage{wrapfig}
\usepackage[title]{appendix}
\usepackage{amsthm, amsmath, amssymb, amsfonts, amscd, amscd, geometry}
\newsavebox{\ssa}
\usepackage{bbm}

\newtheorem{thm}{Theorem}[subsection]
\newtheorem{prop}[thm]{Proposition}
\newtheorem{lemma}[thm]{Lemma}
\newtheorem{cor}[thm]{Corollary}

\theoremstyle{definition}
\newtheorem{def0}[thm]{Definition}

\newtheoremstyle{boldremark}
    {\dimexpr\topsep/2\relax} 
    {\dimexpr\topsep/2\relax} 
    {}          
    {}          
    {\bfseries} 
    {.}         
    {.5em}      
    {}          

\theoremstyle{boldremark}
\newtheorem{rem}[thm]{Remark}
\newtheorem{ex}[thm]{Example}
\newtheorem{cons}[thm]{Construction}
\usepackage{authblk}

\newcommand{\lthm}{Łoś's theorem}
\newcommand{\Rep}{\textrm{Rep}}
\newcommand{\Repb}{\textbf{Rep}}
\newcommand{\Hom}{\textrm{Hom}}
\newcommand{\End}{\textrm{End}}

\newcommand{\Ind}{\textrm{Ind}}
\newcommand{\Tr}{\textrm{Tr}}

\newcommand{\IND}{\textrm{IND}}

\newcommand{\po}{\mathfrak{po}}
\newcommand{\fan}{for almost all $n$}
\newcommand\prodF{{\prod}_{\mathcal F}}
\newcommand{\Fp}{\overline{\mathbb F}_p}
\newcommand{\Fpn}{\overline{\mathbb F}_{p_n}}
\newcommand{\FQ}{\overline{\mathbb Q}}
\newcommand{\Cext}{\overline{\mathbb C(\nu)}}
\newcommand{\chct}{\textrm{char}}

\newcommand{\ct}{\textrm{ct}}
\newcommand{\be}{\bold{e}}
\newcommand{\gr}{{\rm{gr}}}

\usepackage{titlesec}

\titleformat*{\section}{\normalfont\fontfamily{cmr}\fontsize{16}{19}\bfseries}
\titleformat*{\subsection}{\normalfont\fontfamily{cmr}\fontsize{15}{17}\selectfont}
\titleformat*{\subsubsection}{\normalfont\fontfamily{cmr}\fontsize{14}{17}\selectfont}

\long\def\/*#1*/{}

\title{New realizations of deformed double current algebras and Deligne categories}

\author[1]{Pavel Etingof}
\author[1]{Daniil Kalinov}
\author[2]{Eric Rains}

\affil[1]{Department of Mathematics, Massachusetts Institute of Technology, Cambridge, MA 02139, USA}

\affil[2]{Department of Mathematics, California Institute of Technology, Pasadena, CA 91125, USA}
\date{}

\begin{document}

\maketitle

\centerline{\bf To the memory of Ernest Borisovich Vinberg}  

\begin{abstract}
In this paper we propose an alternative construction of a certain class of Deformed Double Current Algebras. We construct them as spherical subalgebras of symplectic reflection algebras in the Deligne category. They can also be thought of as ultraproducts of the corresponding spherical subalgebras in finite rank. We also provide new presentations of DDCA of types A and B by generators and relations.
\end{abstract}

\tableofcontents

\section{Introduction}

Deformed double current algebras (DDCA) of ${\mathfrak{gl}}_m$ are interpolations with respect to the
rank $n$ of Schur algebras associated to symplectic reflection
algebras for wreath products $S_n\ltimes \Gamma^n$, where
$\Gamma$ is a finite subgroup of $SL(2,\mathbb C)$ (\cite{Gu3}). 
They can also be viewed as deformations of enveloping algebras of
(generalized) matrix $W_{1+\infty}$-algebras, and (in some cases) as rational limits of
affine Yangians and toroidal quantum groups. DDCA appeared
first (in a special case) in the physical paper \cite{bernard1995yangian} in 1994. 
However, the systematic theory of DDCA, including
their full definition, was developed only in the last 15 years,  
in a series of papers by N. Guay and his collaborators \cite{Gu1,Gu2,Gu3,guay2016deformed,guay2009double}. In these
papers, presentations of DDCA by generators
and relations are given, the Schur-Weyl functor is defined and
shown to be an equivalence of categories, and the degeneration
of toroidal quantum groups and affine Yangians to DDCA is considered. 

The goal of this paper is to give two alternative definitions of DDCA for ${\mathfrak{gl}}_1$ (i.e., of the interpolations of spherical symplectic reflection algebras). The first definition applies to any finite subgroup $\Gamma\subset SL_2(\mathbb C)$ and is based on {\it Deligne categories}. 
Namely, we consider the Deligne category ${\rm Rep}(S_\nu)$, $\nu\in \mathbb C$, which 
is obtained by interpolating the representation categories of the symmetric group $S_n$ with respect to $n$, \cite{deligne2007categorie}. Using this category, we can define the interpolation $\mathcal{C}_\nu$ of the representation category of the symplectic reflection algebra $H_{t,k}(S_n\ltimes \Gamma^n)$ attached to $\Gamma$ in which the integer $n$ is replaced by a complex parameter $\nu$ (\cite{etingof2014representation}, Subsection 5.3). In the category $\mathcal{C}_\nu$, we have an object $M$ obtained by interpolating the $H_{t,k}(S_n\ltimes \Gamma^n)$-modules $H_{t,k}(S_n\ltimes \Gamma^n)\be$, where ${\be}\in {\mathbb C}[S_n\ltimes \Gamma^n]$ is the projector to the trivial representation, and the DDCA for ${\mathfrak{gl}}_1$ attached to $\Gamma$ may be defined as $\mathcal D_{t,k,c,\nu}(\Gamma)=\End (M)$. This definition opens the door for studying the representation theory of $\mathcal D_{t,k,c,\nu}(\Gamma)$; indeed, if $N$ is another object of $\mathcal{C}_\nu$ then the space $\Hom(M,N)$ 
is naturally a (right) module over $\mathcal D_{t,k,c,\nu}(\Gamma)$. At the same time, it is easy to construct objects of 
$\mathcal{C}_\nu$ because it is given ``by generators and relations"; for instance, if $\Gamma$ is cyclic then $\mathcal{C}_\nu$ contains the category $\mathcal{O}$ which can be studied by methods of the theory of highest weight categories. In fact, in the case $\Gamma=1$ this has already been started in \cite{entova2014representations}. 

In future publications we plan to apply this approach to the DDCA of ${\mathfrak{gl}}_m$ for $m>1$. 
Note that one of its advantages is that it easily applies to the case of $m=1$ (discussed in this paper), while this is a difficult case for the approach of \cite{Gu1,Gu2,Gu3,guay2016deformed,guay2009double} which uses Steinberg-type presentations of ${\mathfrak{gl}}_m$. 
 
The second definition of the DDCA (which we show to be equivalent to the first one) is by explicit generators and relations (but different from \cite{Gu1,Gu2,Gu3,guay2016deformed,guay2009double}), and we give it only for $\Gamma=1$ and $\Gamma = \mathbb Z / 2$. This definition is based on deforming the presentation of the Lie algebra $\mathfrak{po}$ of Hamiltonians on ${\mathbb C}^2$ and of its even part $\mathfrak{po}^+$ by generators and relations. Namely, we show (in part using a computer) that $\mathcal D_{t,k,c,\nu}(\Gamma)$ in the case $\Gamma=1$ is the unique filtered deformation of the enveloping algebra $U(\mathfrak{po})$ for an appropriate filtration. We outline a similar approach for $\Gamma={\mathbb Z}/2$, although for larger $\Gamma$ the relations get too complicated.

The organization of the paper is as follows. 

Section 2 contains preliminaries. 

Section 3 describes generalities on ultraproducts, Deligne categories, and symplectic reflection algebras in complex rank, for simplicity concentrating mostly on the case of the rational Cherednik algebra of type A.

Section 4  explains two definitions of $\mathcal D_{t,k,\nu}$ -- the DDCA of type A, both the usual one (as the ultraproduct of spherical rational Cherednik algebras of type A) and the one via Deligne categories, and shows that they are equivalent. In this section we also state and prove the presentation of this algebra by generators and relations, showing that this is the unique filtered deformation of $U(\mathfrak{po})$. 

Finally, in Section 5 we generalize some of our results to DDCA for arbitrary $\Gamma$ and also state the result about the presentation of the DDCA of type B by generators and relations.

{\bf Acknowledgments.} This paper owes its existence to Victor Ginzburg, who proposed to study deformed double current algebras in the spring of 2001 and suggested, around the same time, some of the important ideas explored below. We are very grateful to Victor for sharing these ideas and initiating this research. We are also grateful to N. Guay, V. Ostrik, and T. Schedler for useful discussions. The work of P.E. and D.K. was  partially supported by the NSF grant DMS-1502244. The computer calculations for this paper were done using MAGMA, \cite{BCP}. 

\section{Preliminaries and notation}

\subsection{General notation} \label{sectnot}

In what follows we will use a lot of different categories of representations. We will always denote the usual (``finite rank") categories of representations using the boldface font, and use the regular font for the interpolation categories (e.g. $\Rep(S_\nu)$).

For example we will use the following notation for the categories of representations of symmetric groups. For convenience set ${\mathbb F}_0 = \Bbb Q$.

\begin{def0}
By $\textbf{Rep}(S_n; \Bbbk)$ denote the category of (possibly infinite dimensional) representations of the symmetric group $S_n$ over $\Bbbk$. By $\textbf{Rep}^{f}(S_n;\Bbbk)$ denote the full subcategory of finite dimensional representations. 
Also for $p\ge 0$ set  $\textbf{Rep}_p(S_n) := \textbf{Rep}(S_n; \Fp)$ and  $\textbf{Rep}^{f}_p(S_n) := \Repb^f(S_n; \Fp)$. 
\end{def0}

We will also fix the notation for the irreducible representations of the symmetric group.
\begin{def0}
 For a Young diagram $\lambda$, by $l(\lambda)$ denote the number of rows of the diagram (the length), by $|\lambda|$ the number of boxes (the weight) and by $\ct(\lambda)$ the content of $\lambda$, i.e., $\ct(\lambda)=\sum_{(i,j)\in \lambda}(j-i)$, where $(i,j)$ denotes the box of $\lambda$ in row $i$ and column $j$.
\end{def0}
\begin{def0}
For $p=0$ or $p>n$ and a Young diagram $\lambda$ such that $|\lambda|=n$ denote by $X_p(\lambda)$ the unique simple object of $\Repb_p(S_n)$ corresponding to $\lambda$.

For $n>0$ and $p\ge 0$ denote by $\mathfrak{h}_n^p \in \Repb_p(S_n)$, or shortly by $\mathfrak{h}_n$ (if there is no ambiguity about the characteristic) the standard permutation representation of $S_n$. 
\end{def0}

There is an important central element in $\Bbbk[S_n]$:
\begin{def0}\label{centraldef}
Denote the central element $\sum_{1 \le i<j \le n}s_{ij} \in \Bbbk[S_n]$ by $\Omega_n$. 
\end{def0}
\begin{rem}
Note that $\Omega_n$ acts on $X_p(\lambda)$ by $\ct(\lambda)$.
\end{rem}

As another piece of notation, below we will frequently use the following operation on Young diagrams:
\begin{def0} \label{defyoung}
 For a Young diagram $\lambda$ and an integer $n \ge \lambda_1+|\lambda|$ denote by $\lambda|_n$ the Young diagram $(n - |\lambda|, \lambda_1, \dots, \lambda_{l(\lambda)})$, where $\lambda_i$ is the length of the $i$-th row of $\lambda$.
\end{def0}

In what follows we will often use the language of tensor categories. Here's what we mean by a tensor category (see Definition 4.1.1 in \cite{etingof2016tensor}):
\begin{def0}
A tensor category $\mathcal C$ is a $\Bbbk$-linear locally finite abelian rigid symmetric monoidal category, such that $\End_{\mathcal C}(\mathbbm{1}) \simeq \Bbbk$.
\end{def0}

We will also fix a notation for the symmetric structure:
\begin{def0}
 For an object $X$ of a tensor category $\mathcal C$, we will denote by $\sigma_X$ the map from $X \otimes X$ to itself, given by the symmetric structure, i.e., the map permuting the two copies of $X$. Oftentimes, when the object we are referring to is obvious from the context, we will denote it simply by $\sigma$.
\end{def0}

We will also use the notion of the ind-completion of a category. For a general category ind-objects are given by diagrams in the category, with morphisms being morphisms between diagrams. However, in the case of a semisimple category there is a more concrete description. 

\begin{def0} \label{inddef}
For a semisimple category $\mathcal C$ with the set of simple objects $\{V_{\alpha}\}$ for $\alpha \in A$ the category\footnote{We use all uppercase letters to denote $\IND$, so as not to confuse it with the induction functors.}
 $\IND(\mathcal C)$  is the category $\mathcal D$ with objects $\bigoplus_{\alpha \in A} M_{\alpha}\otimes V_{\alpha}$, where $M_{\alpha}$ are (possibly infinite dimensional) vector spaces. The morphism spaces are given by:
$$
\Hom_{\mathcal D}(\bigoplus_{\alpha \in A} M_{\alpha} \otimes V_{\alpha},\bigoplus_{\beta \in A} N_{\beta} \otimes V_{\beta}) = \prod_{\alpha \in A} \Hom_{\textrm{Vect}}(M_{\alpha},N_{\alpha})  . 
$$
\end{def0}

Thus, in this case, we can think of ind-objects as infinite direct sums of objects of $\mathcal C$.

Next we would like to explain a way to define an ind-object of $\mathcal C$.
\begin{cons} \label{indcons}
Suppose $0 = X_0 \subset X_1 \subset X_2 \subset \dots \subset X_i \subset \dots$ is a nested sequence of objects of $\mathcal C$. Then their formal colimit, which we denote by $X$, is an object of $\IND(\mathcal C)$. We can write it down explicitly in terms of Definition \ref{inddef}. 

Indeed, suppose we have $X_i = \bigoplus_{\alpha \in A}M_{i,\alpha} \otimes V_{\alpha}$. Then it follows that:
$$
\bigcup_{i \in \mathbb N} X_i  = X = \bigoplus_{\alpha \in A} \left(\bigcup_{i \in \mathbb N} M_{i,\alpha}\right) \otimes V_{\alpha}  ,
$$
where $\bigcup_{i \in \mathbb N}X_i=\varinjlim X_i$ stands for the colimit along the diagram consisting of points numbered by $\mathbb N$ and arrows from $i$ to $i+1$ for all $i$.
\end{cons}

\begin{rem}\label{remindmor}
Suppose that $X$ and $Y$ are two objects constructed via Construction \ref{indcons}. Then:
$$
\Hom_{\IND(\mathcal C)}(X,Y) =\varprojlim_{i \in \mathbb N} \bigcup_{j \in \mathbb N}\Hom_{\mathcal C}(X_i,Y_j) .
$$

In case when $X$ is actually an object of $\mathcal C$, this simplifies to:
$$
\Hom_{\IND(\mathcal C)}(X,Y) = \bigcup_{j \in \mathbb N}\Hom_{\mathcal C}(X,Y_j).
$$
In other words, $X$ is a compact object of $\IND(\mathcal C)$.
\end{rem}
\begin{ex}
We have  $\textbf{Rep}_p(S_n) = \IND(\textbf{Rep}^{f}_p(S_n))$. Indeed, this holds for the representation category of any finite dimensional algebra.
\end{ex}

\subsection{Wreath products $S_n \ltimes \Gamma^n$}

To deal with DDCA with non-trivial $\Gamma$ we will need to use a certain interpolation of categories of representations of wreath products. Below we will state basic facts about representations of wreath products in finite rank.

\begin{def0} \label{wreathdef}
For a finite group $\Gamma$, consider the action of $S_n$ on $\Gamma^n$ by permutations. The semiderect product $S_n \ltimes \Gamma^n$ is called the wreath product.
\end{def0}

\begin{rem}
Outside of the present section we will be interested only in $\Gamma \subset \mathrm{SL}(2,\Bbbk)$. However the results stated in the present section hold for any $\Gamma$.
\end{rem}

We have the following classification of irreducible representations of $S_n \ltimes \Gamma^n$.
\begin{prop}\label{wreathprop1}
Suppose $\Bbbk$ is an algebraically closed field of characteristic \linebreak $\chct(\Bbbk)= p > n,|\Gamma|$ or $p=0$. Suppose $A$ is the set of indices which goes over all of the irreducible representations of $\Gamma$ over $\Bbbk$, i.e., $\{W_\alpha\}_{\alpha \in A}$ is the set of irreducible representations of $\Gamma$. Then the set of all irreducible representations of $S_n \ltimes \Gamma^n$  over $\Bbbk$ is in 1-1 correspondence with functions:
$$
\lambda: A \to {\rm Partitions},
$$
such that $\sum_{\alpha \in A}| \lambda (\alpha)| = n$. The representation corresponding to fixed $\lambda$ is given by:
$$
X_p(\lambda) = {\rm Ind}^{S_n \ltimes \Gamma^n}_{(\prod_{\alpha \in A}S_{\lambda(\alpha)})\ltimes \Gamma^n}(\bigotimes_{\alpha \in A}X_{p}(\lambda(\alpha))\otimes W_{\alpha}^{\otimes |\lambda(\alpha)|}).
$$

\end{prop}

We will use the notations for the representation categories similar to the case of the symmetric group:

\begin{def0}
By $\textbf{Rep}(S_n\ltimes \Gamma^n; \Bbbk)$ denote the category of representations of the wreath product $S_n\ltimes \Gamma^n$ over $\Bbbk$. By $\textbf{Rep}^{f}(S_n\ltimes \Gamma^n;\Bbbk)$ denote the full subcategory of finite dimensional representations. 

Also for $p\ge 0$ set  $$\textbf{Rep}_p(S_n\ltimes \Gamma^n) := \textbf{Rep}(S_n\ltimes \Gamma^n; \Fp),\ \textbf{Rep}^{f}_p(S_n\ltimes \Gamma^n) := \Repb^f(S_n\ltimes \Gamma^n; \Fp).$$ 
\end{def0}

\subsection{The Cherednik algebra}

In this paper we will be mainly interested in rational Cherednik algebras of type A. Thus we will only give definitions of this algebra below. For the definition and theory of general rational Cherednik algebras, see \cite{etingof2010lecture}.

\begin{def0} \label{cherAdef}
 The rational Cherednik algebra of type $A$ and rank $n$  over a field $\Bbbk$, denoted by $H_{t,k}(n,\Bbbk)=H_{t,k}(n)$, where $t,k \in \Bbbk$, is defined as follows. Consider the standard representation of $S_n$ acting by permutations on $\mathfrak h = \Bbbk^n$ with the basis given by $y_i \in \mathfrak h$, and the dual basis $x_i \in \mathfrak h^*$. Then $H_{t,k}(n)$ is the quotient of $\Bbbk[S_n] \ltimes T(\mathfrak h \oplus \mathfrak h^*) $ by the following relations:
 $$
 [x_i,x_j] = 0  ,\ [y_i,y_j] = 0  ,\ [y_i,x_j] = \delta_{ij}(t-k\sum_{m\ne i}s_{im}) + (1-\delta_{ij})ks_{ij}  ,
 $$
 where $s_{ij}$ denotes the transposition of $i$ and $j$.
\end{def0}

In other words, this is the rational Cherednik algebra corresponding to the root system $A_{n-1}$.

This algebra has a filtration determined by $\deg(x_i) = \deg(y_i)=1$ and $\deg(g) = 0$ for any group element $g$. The associated graded algebra is:
$$
\gr(H_{t,k}(n)) = \Bbbk[S_n]\ltimes S(\mathfrak h \oplus \mathfrak h^*). 
$$
This follows from the fact that the  analog of the PBW theorem holds for this algebra:
\begin{prop}
The natural map $H_{0,0}(n) \to \gr(H_{t,k}(n))$ is a vector space isomorphism.
\end{prop}

Another important object is the spherical subalgebra of the rational Cherednik algebra.
\begin{def0}
If $\chct(\Bbbk)=p > n$ or $p=0$, denote by $\mathcal B_{t,k}(n)$ the subalgebra $\bold{e}H_{t,k}(n)\be$ of $H_{t,k}(n)$, where $\be \in \Bbbk[S_n]$ is the averaging idempotemt.
\end{def0}

Note that:
$$
\gr(\mathcal B_{t,k}(n)) = S(\mathfrak h \oplus \mathfrak h^*)^{S_n} = \Bbbk[x_1,\dots,x_n,y_1,\dots, y_n]^{S_n}  .
$$

\begin{rem} \label{sphrem}
One can construct the spherical subalgebra in another way. Indeed, regard $\Bbbk$ as the trivial representation of $S_n$  and apply to it the induction functor  $\Ind_{S_n}^{H_{t,k}(n)}(\Bbbk)$. It's easy to see that this representation is in fact $H_{t,k}(n)\be$. Now the spherical subalgebra is given as follows:
$$
\mathcal B_{t,k}(n) = \be H_{t,k}(n)\be = \Hom_{S_n}(\Bbbk, H_{t,k}(n)\be) = \End_{H_{t,k}(n)}(\Ind_{S_n}^{H_{t,k}(n)}(\Bbbk))  .
$$
\end{rem}

Now we can introduce the corresponding categories of representations.

\begin{def0}
 By $\textbf{Rep}(H_{t,k}(n); \Bbbk)$ denote the category of (possibly infinite dimensional) representations of the rational Cherednik algebra $H_{t,k}(n)=H_{t,k}(n,\Bbbk)$. Also set  $\textbf{Rep}_p(H_{t,k}(n)) = \textbf{Rep}(H_{t,k}(n),\Fp)$. 
\end{def0}

\subsection{Symplectic reflection algebras}

Another entity we are going to use to construct DDCA with non-trivial $\Gamma$ is symplectic reflection algebras. Below we will give some basic definitions, needed for our purposes. For more on this topic see \cite{etingof2002symplectic}.

The symplectic reflection algebra is defined as follows:
\begin{def0}\label{sympfindef}
Fix a finite subgroup $\Gamma \subset \textrm{SL}(2; \Bbbk)$. Fix numbers $t,k \in \Bbbk$. Fix numbers $c_{C} \in \Bbbk$ for every conjugacy class $C \subset \Gamma$; we will denote the collection of these numbers by $c$. For every conjugacy class $C$, set $T_C:= \frac{1}{2}\Tr|_{\Bbbk^2}\gamma$, where $\gamma \in C$ is an element of the conjugacy class and we take the trace over the tautological representation.
Consider $V = (\Bbbk^2)^n$, the tautological representation of the wreath product $S_n \ltimes \Gamma^n$. Note that this space has a natural symplectic structure, which we will denote by $\omega$. Let $\Sigma$ stand for the set of elements of $S_n \ltimes \Gamma^n$ conjugate to a transposition. For a conjugacy class $C \subset \Gamma$, let $\Sigma_C$ be the set of all elements conjugate to $(1,1,\dots,1,\gamma)$ for $\gamma \in C$.

The symplectic reflection algebra $H_{t,k,c}(n,\Gamma)$ is the quotient of $\Bbbk[ S_n\ltimes \Gamma^n]\ltimes T(V)$ by the relations:
$$
[y,x] = t\omega(y,x) - k\sum_{s \in \Sigma}\omega(y,(1-s)x)s - \sum_{C}\frac{c_C}{1-T_C}\sum_{s \in \Sigma_C}\omega((1-s)y,(1-s)x)s,\ x,y\in V.
$$
\end{def0}

We can also define the spherical subalgebra of this algebra:
\begin{def0}
The spherical subalgebra of the symplectic reflection algebra $H_{t,k,c}(n,\Gamma)$ is denoted by $\mathcal B_{t,k,c}(n,\Gamma)$ and is given by:
$$
\mathcal B_{t,k,c}(n,\Gamma) = \be H_{t,k,c}(n,\Gamma)\be,
$$
where $\be$ is the symmetrizer for $S_n\ltimes \Gamma^n$.
\end{def0}

\begin{rem}
As before we have:
$$
\mathcal B_{t,k,c}(n,\Gamma) = {\rm Hom}_{S_n\ltimes \Gamma^n}(\Bbbk, \Ind_{S_n \ltimes \Gamma^n}^{H_{t,k,c}(n,\Gamma)}(\Bbbk))  .
$$
\end{rem}

We will use the same notation for the categories of representations:
\begin{def0}
 By $\textbf{Rep}(H_{t,k,c}(\Gamma,n); \Bbbk)$ denote the category of  representations of the symplectic reflection algebra $H_{t,k,c}(\Gamma,n)$ over $\Bbbk$. Also for $p\ge 0$ denote 
 $$
 \textbf{Rep}_p(H_{t,k,c}(\Gamma,n)) = \textbf{Rep}(H_{t,k,c}(\Gamma,n); \Fp) \ .
 $$
\end{def0}

\begin{rem}
Notice that when $\Gamma =1$ we get back the case of rational Cherednik algebra of type A, i.e., $H_{t,k,\emptyset}(n,1) = H_{t,k}(n)$. Also, in the case $\Gamma = \mathbb Z/2\mathbb Z$ we get the rational Cherednik algebra of type B.
\end{rem}

\subsection{Ultrafilters}

Below we will discuss some basic facts about ultrafilters and ultraproducts. Ultrafilters provide us with a notion of the limit of algebraic structures, which works really well for describing Deligne categories. Thus, we will use this framework extensively in the present paper.

We will define what ultrafilters and ultraproducts are, state their main properties and give some important examples, which will be used later in the paper. The following discussion is an updated version of the corresponding discussion from \cite{kalinov2018finite}. For more details on  this topic in the algebraic context, see \cite{schoutens2010use}.

\subsubsection{Ultrafilters and ultraproducts: basic definitions}

\begin{def0}
An ultrafilter $\mathcal F$ on a set $X$ is a subset of $2^{X}$ satisfying the following properties:

$\bullet$ $X \in \mathcal F$ ;

$\bullet$ If $A \in \mathcal F$ and $A \subset B$, then $B \in \mathcal F$ ;

$\bullet$ If $A,B \in \mathcal F$, then $A\cap B \in \mathcal F$ ;

$\bullet$ For any $A\subset X$ either $A$ or $X \backslash A$ belongs to $A$, but not both.
\end{def0}

For any $X$, there is an obvious family of examples of ultrafilters. Indeed, taking $\mathcal F_x = \{ A \in 2^{X}| x \in A \}$ for any $x \in X$ gives us an ultrafilter. Such ultrafilters are called principal. Using Zorn's lemma one can show that non-principal ultrafilters $\mathcal F$ exist iff the cardinality of $X$ is infinite. However the proof is non-constructive.

From now on we will only work with  non-principal ultrafilters on $X = \mathbb N$. 
\begin{def0}
For the rest of the paper we will denote by $\mathcal F$ a fixed non-principal ultrafilter on $\mathbb N$.
\end{def0}

Note that it doesn't matter which non-principal ultrafilter to take, and all our results do not depend on this choice. Also note  that all cofinite sets belong to $\mathcal F$. Indeed, if some cofinite set wouldn't belong to $\mathcal F$, it would follow that a finite set belongs to $\mathcal F$. But from this one can conclude that $\mathcal F$ is a principal ultrafilter for one of the elements of this set.

Throughout the paper we will use the following shorthand phrase.
\begin{def0}
By the statement  ``$A$ holds for almost all $n$'', where $A$  is a logical statement depending on $n$, we will mean that $A$ is true for some subset of natural numbers $U$, such that $U \in \mathcal F$.
\end{def0}

The following is an important lemma describing what happens with the conjuction and disjunction of statements which ``hold for almost all $n$". 

\begin{lemma}\label{ultrlemma}
1) If for two logical statements $A$ and $B$ we know that $A$ holds for almost all $n$ and $B$ holds for almost all $n$, then $A \wedge B$ holds \fan.

2) If for a finite number of logical statements $A_i$, for $i \in I$, we know that $\bigvee_{i\in I} A_i$ holds \fan, then there is $j\in I$ such that $A_j$ holds \fan. 
\end{lemma}
\begin{proof}
1) Indeed, we know that there is a set $U_A\in\mathcal F$ such that $A$ holds for all $n \in U_A$, and the corresponding set for $B$. Now by definition of the ultrafilter $U_A\cap U_B \in \mathcal F$, and $A\wedge B$ holds for all $n \in U_A \cap U_B$.

2) Suppose that none of the statements $A_i$ hold \fan. This means that the sets on which $A_i$ hold do not belong to $\mathcal F$. Thus by definition of the ultrafilter, the sets $V_i = \{ n \in \mathbb N|  \ A_i \text{ does not hold} \}$ are in $\mathcal F$. Thus $V = \bigcap_{i\in I} V_i \in \mathcal F$. But for any $n\in V$ we know that all of the statements $A_i$ do not hold. Hence for any $n \in V$ we know that $\bigvee_{i\in I} A_i$ does not hold. But the set $U = \{n \in N| \bigvee_{i\in I} A_i\}$ belongs to $\mathcal F$ by assumption. So we have $V$ and $\mathbb N \backslash V$ belonging to $\mathcal F$. A contradiction.
\end{proof}

We will use these elementary observations quite frequently, sometimes without even mentioning it.  

Now, define the notion of an ultraproduct. 
\begin{def0}
Suppose we have a sequence of sets $E_n$ labeled by natural numbers. Consider the set $\prod'_{\mathcal F}E_n$ consisting of the sequences $\{e_n\}_{n \in A}$ for a set $A \in \mathcal F$ and $e_n \in E_n$. i.e., $\prod'_{\mathcal F}E_n$ consists of sequences of elements of $E_n$ which are defined \fan. Then $\prod_{\mathcal F}E_n$ is the quotient of $\prod'_{\mathcal F}E_n$ by the following relation: $\{e_n\}_{n \in A} \sim \{e_n'\}_{n \in A'}$ iff $e_n = e_n'$ for almost all $n$ (i.e.,  on $B \subset A' \cap A$, such that $B \in \mathcal F$).
The set $\prod_{\mathcal F}E_n$ is called the ultraproduct of the sequence $\{E_n\}_{n \in \mathbb N}$. 
\end{def0}

\begin{rem}
Thus in a nutshell the ultraproduct consists of ``germs" of sequences of elements which are defined \fan. Because of this in what follows we will sometimes use ``sequence'' to mean ``sequence defined \fan''.
\end{rem}

\begin{rem}
Note that for any finite set $C$, the ultraproduct of its copies $\prod_{\mathcal F}C_i$ with $C_i = C$ is equal to $C$. Indeed, for any sequence $\{c_n\}_{n \in  A}$, for some $A \in \mathcal F$, we can define $U_{d} = \{ n \in A| d = c_n\} $ for any $d \in C$. Then we have $\bigcup_{d \in C}U_d = A$, thus one of the $U_d$'s must belong to $\mathcal F$. So it follows that $\{c_n\}_{n \in A} \sim \{d\}_{n\in A}$ for this particular $d$.
\end{rem}

Oftentimes we use the following notation:
\begin{def0}
For a sequence $\{E_n\}_{n \in \mathbb N}$, denote an element $\{ e_n \}_{n \in \mathbb N} \in \prod_{\mathcal F}E_n$ by $\prod_\mathcal F e_n$.
\end{def0}

This construction is interesting for us, because it, in a certain sense, preserves a lot of algebraic structures. We will explore this dimension of ultraproducts below.

\begin{ex} \label{firstultex}
First,  note that the ultraproduct inherits any operation or any relation which is defined on a sequence of sets $E_n$ for almost all $n$. For example, suppose we are given a sequence of $k$-ary operations $\circ_n$ defined for almost all $n$. Let $E:= \prod_{\mathcal F} E_n$ and consider the $k$-ary operation $\circ: E\times E\times \dots \times E\to E$ defined as 
$$
\circ(e^1,e^2,\dots,e^k) = \circ(\prodF e^1_n, \dots, \prodF e^k_n) = \prodF \circ_n(e^1_n,\dots,e^k_n)  .
$$
Note that this is the same as taking $\circ = \prodF \circ_n \in \prodF \Hom_{\rm Sets}(E^{\times k},E)$, so we can call $\circ$ an ultraproduct of $\circ_n$.
Now if we have any sequence of relations $r_n$ given \fan, they can be written as a sequence of $k$-ary maps with Boolean values. And one can define $r$ to be a relation on $E$ in a similar way 
$$
r(e^1,e^2,\dots,e^k) = r(\prodF e^1_n, \dots, \prodF e^k_n) = \prodF r_n(e^1_n,\dots,e^k_n) \in \prodF \textbf{2} = \textbf{2}.\footnote{Here $\textbf{2}$ stands for the Boolean set $\{0,1\}$.}
$$ For the same reason we can call the relation $r$ the ultraproduct of the relations $r_n$. Note that this means that if the relation $r_n$ was true {\fan} (i.e., ${\rm Im}(r_n) = \{1\}$ \fan), it follows that $r$ is also true.
\end{ex}

One can easily check for oneself that the above examples (\ref{firstultex}) can be extended to any collections of sequences of sets, maps between them and relations between maps. That means that if we have a collection of sequences of sets with a certain algebraic structure defined by maps between them, we can form the ultraproducts of these sets and these maps. Moreover if the sequences of maps satisfy a certain collection of relations, the ultraproduct will satisfy them too. 

These observations may be formulated in the following way:
\begin{thm}\textbf{\lthm} (Theorem 2.3.2 in \cite{schoutens2010use})

Suppose we have a collection of sequences of sets $E^{(k)}_i$ for $k = 1,\dots,m$, a collection of sequences of elements $f^{(r)}_i$ for $r = 1,\dots, l$, and a  formula of a first order language $\phi(x_1,\dots,x_l, Y_1, \dots, Y_m)$ depending on some parameters $x_i$ and sets $Y_j$. Denote by \linebreak $E^{(k)} = \prod_{\mathcal F}E^{(k)}_{n}$ and $f^{(r)} = \prod_{\mathcal F} f^{(r)}_n$.  Then 
$\phi(f^{(1)}_n, \dots, f^{(l)}_n, E^{(1)}_n, \dots, E^{(m)}_n)$ is true for almost all $n$ iff  $\phi(f^{(1)}, \dots, f^{(l)}, E^{(1)}, \dots E^{(m)})$ is true.
\end{thm}

In the next subsection we will provide a few examples of application of this theorem. One can easily see how the theorem works by working out what happens in these examples on one's own. Many of these examples will be used in the rest of the paper.

\subsubsection{Examples of ultraproducts}

\begin{ex}
 If $E_n$ is a sequence of monoids/groups/rings/fields then $\prod_{\mathcal F} E_n$ with operations given by taking the ultraproduct of the operations as elements of the corresponding sets of set-theoretical maps gives us a structure of a monoid/group/ring/field by Łoś's theorem.
\end{ex}

\begin{ex}
 If $V_i$ are finite dimensional vector spaces over a field $\Bbbk$, then $\prod_{\mathcal F} V_n$ is a vector space over $\prodF \Bbbk$, which is not necessarily finite dimensional, since the property of being finite dimensional cannot be written in a first-order language. But if the dimensions of $V_n$ are bounded, then they are the same for almost all $n$ and hence $V$ has the same dimension (for example, because the ultraproduct of bases is a basis).
\end{ex}

\begin{ex} \label{fieldextrans}
 Take the ultraproduct of a countably infinite number of copies of $\overline{\mathbb Q}$. By \lthm $\prod_{\mathcal F} \overline{\mathbb Q}$ is a field, which is algebraically closed. It has characteristic zero since $\forall k  \in \mathbb Z$ such that $k\ne 0$ it follows that $ k = \prod_{\mathcal F} k\ne 0$.  Also it is easy to see that its cardinality is continuum. Hence by Steinitz's theorem\footnote{This theorem tells us that two uncountable algebraically closed fields are isomorphic iff their characteristic and cardinality are the same. It is proven in \cite{steinitz1910algebraische}.} $\prod_{\mathcal F} \overline{\mathbb Q} \simeq \mathbb C$. Note that there is no canonical isomorphism.
 
 Consider the ultraproduct of integers $\prodF n$. Via the isomorphism constructed in the previous paragraph this is an element of $\mathbb C$. Notice that this element cannot satisfy any nontrivial polynomial equation over $\mathbb Q$ (indeed, the corresponding polynomial must have infinitely many roots), hence $\prodF n$ is a transcendental element of $\mathbb C$. By an automorphism of $\mathbb C$ we can send this element into any transcendental element of $\mathbb C$. 
 
 Thus we conclude that for any transcendental element $\nu \in \mathbb C$ there is an isomorphism $\prodF \FQ \simeq \mathbb C$, such that $\prodF n = \nu$.
 
 Also notice that by Steinitz's theorem it follows that $\overline{\mathbb C(x)}\simeq \mathbb C$, since they have the same cardinality. Thus we can also conclude that there is an isomorphism $\prodF \FQ \simeq \overline{\mathbb C(x)}$ such that $\prodF n = x$.
\end{ex}
 
 \begin{ex} \label{fieldexalg}
 Take the ultraproduct of $\overline{\mathbb F}_{p_n}$ for some sequence of distinct prime numbers $p_n$. As before, by \lthm${}$ $\prod_{\mathcal F} \overline{\mathbb F}_{p_n}$ is a field, which is algebraically closed. Also it has cardinality continuum. Now for any natural number $k$, we have $k = \prod_{\mathcal F} k \ne 0$, since it is equal to zero for at most a finite number of $n$.  Hence $\prod_{\mathcal F} \overline{\mathbb F}_{p_n} \simeq \mathbb C$ by Steinitz's theorem, again not in a canonical way.
 
 Suppose we are given an algebraic number $\nu \in \mathbb C$. Let us show that there exists a sequence of integers $\nu_n$ and prime numbers $p_n$ such that $\nu_n < p_n$ and $\prodF \nu_n = \nu$ inside $\prodF \Fpn \simeq \mathbb C$; this will be needed in what follows.
 
Let $q(x) \in \mathbb Z[x]$ be the minimal polynomial for $\nu$. We would like to find an infinite number of pairs $\nu_n$, $p_n$ such that $q(\nu_n) = 0 \mod p_n$. Let us show that there is an infinite number of primes dividing the collection of numbers $q(l)$ for $l \in \mathbb N$, from this it would follow that there is an infinite number of pairs since only a finite number of primes divide each $q(l)$. Suppose it is not so, and there are only $k$ such primes. Fix $C$ such that  we have $q(l) < C \cdot l^{\deg(q)}$ for all positive  integer values of $l$. Denote by $Q$ the number of integers of the form $q(l)$ for $l \in \mathbb Z_{\ge 0}$ such that $q(l)<L$. By the above inequality (that is $q(l) < C \cdot l^{\deg(q)}$ ) $Q$ is at least $\frac{1}{C}\cdot L^{\frac{1}{\deg(q)}}$. On the other hand the number $P$ of  numbers less than $L$ divisible only by $k$ fixed primes  is less or equal to $\log_2(L)^k$, since each prime number is at least $2$. Hence for big enough $L$ we have $P<Q$, which contradicts the hypothesis\footnote{This proof is also written by Nate Harman in the proof of Prop. 2.2 in \cite{harman2016deligne}.}. 
 
Hence we can take a sequence of distinct primes $p_n$ and a sequence of integers $\nu_n$ tending to infinity such that $q(\nu_n) = 0$ in $\mathbb F_{p_n}$ and $\nu_n < p_n$. It follows that $\prodF \nu_n$ in $\prodF \Fpn$ is a root of $q(x)$. Hence by an automorphism of $\mathbb C$ we can send $\prodF \nu_n$ into $\nu$.
\end{ex}

\begin{ex}\label{catex}
 Suppose $\mathcal C_n$ is a sequence of  (locally small) categories. We can define the ultraproduct category $\widehat{\mathcal C} = \prod_{\mathcal F} \mathcal C_n$ as the category 
whose objects are sequences of objects in $\mathcal C_n$.  
For clarity we will denote the ultraproduct of objects by\footnote{The superscript $C$ stands for "category".} $\prodF^C$. The morphisms in $\widehat{\mathcal C}$ are given by
 $$
 \Hom_{\widehat{\mathcal C}}(\prodF^C X_n,\prodF^C Y_n) = \prodF \Hom_{\mathcal C_n}(X_n,Y_n), 
 $$
 and the composition maps are given by the ultraproducts of the composition maps, i.e., \linebreak $(\prod_{\mathcal F}f_n) \circ (\prod_{\mathcal F}g_n) = \prod_{\mathcal F} (f_n \circ g_n)$. By \lthm${}$ this data satisfies the axioms of a category. If the categories $\mathcal C_n$ have some structures, for example the structures of an abelian or monoidal category, then $\widehat{\mathcal C}$ also has these structures\footnote{But the finite-length property, for example, does not survive, as it cannot be formulated as a first-order logical statement.}.
 
 Usually $\widehat{\mathcal C}$ is too big and it is interesting to consider a certain full subcategory $\mathcal C$ in there, for example by only  considering the ultraproducts of  sequences of objects of $\mathcal C_i$ bounded in some sense. This will be discussed in more detail in the next subsection. 
\end{ex}

\begin{rem}
Note that taking the ultraproduct of a sequence of algebraic objects as such is different from considering their ultraproduct as a sequence of objects in certain categories.

For example, consider a sequence of countably-dimensional vector spaces $V_n$ over $\Bbbk$. By \lthm${}$ $\prodF V_n$ is a vector space (although its dimension is more than countable). However, we can also regard $V_n$ as objects of the categories $\mathcal C_n = \textrm{Vect}_\Bbbk$ and construct $\prodF^C V_n \in \prodF \textrm{Vect}_{\Bbbk}$. The category $\prodF \textrm{Vect}_{\Bbbk}$ is not equivalent to the category of vector spaces (for example, it is rigid and can have objects of non-integer dimension), so $\prodF^C V_n$ is not a vector space in any sense. 
\end{rem}

 Also frequently it is useful to think about an ultraproduct as a certain kind of a limit as $n \mapsto \infty$, where $n$ becomes a ``free" parameter. 

\begin{ex} \label{finalgex}
Consider a sequence of finite dimensional algebras $A_n$ over $\overline{\mathbb Q}$ with a sequence of fixed vector space isomorphisms $A_n \simeq V$. Equivalently, this means that we have a sequence of binary operations $\mu_n:V \otimes V \to V$ which satisfy all the axioms of an algebra. Suppose in some basis (and hence in any basis) the matrices of $\mu_n$ have entries which depend polynomially on $n$.

Consider $A = \prod_{\mathcal F}A_n$. By Example \ref{fieldextrans} this is an algebra over $\overline{\mathbb C(x)}$. Since $A_i$ are finite dimensional and all isomorphic to $V$ via a fixed  isomorphism, we can also conclude that the binary operation on $A$, which we denote by $\mu$, is given by $\prodF \mu_n$.  Since $\mu_n$ depended polynomially on $n$ and $x = \prodF n$, it follows that $\mu$ is given by the same formulas as the sequence $\mu_n$ with $n$ substituted by $x$. In other words, if  $c^\gamma_{\alpha,\beta}(n)$ are the structure constants of $\mu_n$ in a certain basis then $c^\gamma_{\alpha,\beta}(x)$ are the structure constants of $\mu$. I.e., $n$ becomes a formal parameter in $A$.
\end{ex}

\subsubsection{Restricted ultraproducts}

When one works with a sequence of objects which are in some sense infinite dimensional, it's sometimes useful to consider a subobject in the ultraproduct consisting of the sequences of elements which are in a some way bounded. This can be called a {\it restricted ultraproduct.} We have already mentioned this in the case of categories in Example \ref{catex}. For example, the Deligne category $\Rep(S_\nu)$ will be constructed as a full subcategory in a certain ultraproduct category. 

In this section we will outline the definitions of the restricted ultraproduct which makes sense in the case of filtered or graded vector spaces and categories.

\begin{def0}\label{restrdef}
For a sequence of vector spaces $E_n$  with an increasing filtration $F^0E_n \subset F^1E_n \subset \dots \subset F^kE_n \subset \dots$, define the restricted ultraproduct $\prodF^r E_n$ to be equal to $\bigcup_{k=0}^\infty \prodF F^kE_n  \subset \prodF E_n$.
\end{def0}
\begin{def0}
For a sequence of vector spaces $E_n$  with a grading $E_n = \bigoplus_{k=0}^{\infty} {\rm gr}^kE_n$, define the restricted ultraproduct $\prodF^r E_n$ to be equal to $\bigoplus_{k=0}^\infty \prodF {\rm gr}^kE_n  \subset \prodF E_n$. Note that by taking $F^kE_n = \bigoplus_{i=0}^k {\rm gr}^iE_n$, this construction matches the construction of Definition \ref{restrdef}.
\end{def0}

We will use this notion in the case when the dimensions of the space $F^kE_n$ are finite and stabilize as $n\to \infty$ for fixed $k$. Let us give a few examples.

\begin{ex}
Consider a countable-dimensional vector space $V$ over $\Bbbk$. Consider a sequence of copies of $V$, i.e., $V_n = V$. Also consider an increasing filtration $F^jV$ by finite dimensional subspaces  and the same filtration on all $V_n$. We can calculate the restricted ultraproduct of this sequence:
$$
\prodF^rV_n = \bigcup_{k=0}^\infty \prodF F^kV_n = \bigcup_{k=0}^\infty F^kV = V  .
$$
Whereas the usual ultraproduct $\prodF V_n$ is more than countable-dimensional.
\end{ex}

\begin{ex} \label{infalgex}
This is an extension of Example \ref{finalgex} to an infinite dimensional setting. Consider $A_n$, a sequence algebras over $\FQ$  with an increasing filtration by finite dimensional subspaces, such that for every $k \in \mathbb N$ there is $N_k$ such that for $n > N_k$ all $F^kA_n$ are isomorphic as vector spaces to a fixed vector space $F^kA_{\infty}$ via fixed isomorphisms. I.e., every filtered component stabilizes after a certain point.

This means that we have a collection of sequences of coherent multiplication maps $\mu_n^{k,l}:F^kA_{\infty}\times F^lA_{\infty}\to F^{k+l}A_{\infty}$ defined \fan .  Let's also suppose that this sequence depends polynomially on $n$.

Consider $A = \prodF^rA_n$. Note that as a vector space the restricted ultraproduct equals to:
$$
\prodF^r A_n = \bigcup_{k=0}^{\infty} \prodF F^kA_n = \bigcup_{k=0}^{\infty}F^kA_{\infty},
$$
since $F^kA_n = F^kA_{\infty}$ \fan.

Now as in Example \ref{finalgex} the ultraproducts $\mu^{k,l} = \prodF \mu_{n}^{k,l}$ define a coherent collection of multiplication maps, the union of which defines a map $\mu:A \times A \to A$. The structure constants of this multiplication can also be obtained by taking the structure constants of $A_n$ and plugging in $x$ instead of $n$.

Note that the same construction works if the structure constants depend on $n$ as rational functions.

This example  shows better why it makes sense to think about the ultraproduct as a limit.
\end{ex}

We also would like to introduce a related construction, which we will also call a restricted ultraproduct. This will take place in the setting of the ultraproducts of categories. Suppose $\{\mathcal D_i\}$ is a sequence of artinian abelian categories and $\mathcal D = \prodF \mathcal D_i$ is their ultraproduct (an abelian category which is, in general, not artinian). Suppose $\mathcal C$ is a full artinian subcategory of $\mathcal D$. Using Construction \ref{indcons} we can obtain ind-objects of $\mathcal C$ in the following way.
\begin{cons} \label{indconsult}
Suppose we have a sequence of ind-objects $X_n \in \IND(\mathcal D_n)$ such that each $X_n$ is equipped with a filtration by objects of $\mathcal D_n$. I.e., we have $F^0X_n \subset F^1X_n \subset \dots \subset F^iX_n \subset \dots$, where all $F^iX_n \in \mathcal D_n$ and $X_n = \bigcup_{i \in \mathbb N}F^iX_n$. Also suppose that for each $i\ge 0$, $\prodF^CF^iX_n  \in \mathcal C$. Denote $\prodF^CF^iX_n$ by $F^iX_\infty$. It is clear that 
we have injections $F^iX_\infty\hookrightarrow F^{i+1}X_\infty$. 

It follows that the sequence $F^iX_\infty$ defines an object $X_\infty \in \IND(\mathcal C)$ as:
$$
X_\infty = \bigcup_{i \in \mathbb N}F^iX_\infty = \bigcup_{i \in \mathbb N}\prodF^CF^iX_n  .
$$
\end{cons}

We will use a special notation for this construction:
\begin{def0} \label{restrcatdef}
In the setting of Construction \ref{indconsult}, call $X_\infty$ the restricted ultraproduct of $X_n$ with respect to the fixed filtration. We will write
$$
X_\infty = \prodF^{C,r}X_n  .
$$
\end{def0}

\begin{rem} Let $\widetilde F^\bullet$ be another filtration on the sequence $\lbrace{X_n\rbrace}$ such that $\prodF^CF^iX_n  \in \mathcal C$, and let $\widetilde X_\infty$ be the corresponding restricted ultraproduct. Let us say that 
$F,\widetilde{F}$ are equivalent if for any $i$ there exist $r(i),s(i)$ such that 
$F^iX_n\subset \widetilde F^{r(i)}X_n$ and $\widetilde F^iX_n\subset F^{s(i)}X_n$ for almost all $n$. If $F,\widetilde F$ are equivalent, then we have maps 
$F^iX_\infty\to \widetilde F^{r(i)}X_\infty$ and $\widetilde F^{i}X_n\to F^{s(i)}X_\infty$, which give rise to maps $X_\infty\to \widetilde X_\infty$ and $\widetilde X_\infty\to X_\infty$ which are clearly inverse to each other; thus $X_\infty$ and $\widetilde X_\infty$ are naturally isomorphic. This shows that $X_\infty$ depends only on the equivalence class of the filtration $F$. 

However, not all filtrations are equivalent. E.g., if $X_n=\Bbbk^n$, $F^iX_n$ is spanned by the first $i+1$ standard basis vectors for $i\le n-1$, $g_n\in GL(n,\Bbbk)$ and 
$\widetilde F=g_n(F)$ on $X_n$ then in general $F,\widetilde F$ are not equivalent. 
Thus, without specifying a filtration (at least up to equivalence), we cannot define 
the restricted ultraproduct of $X_n$. 
\end{rem}

\section{Deligne Categories}

\subsection{Constructions of the category $\Rep(S_\nu)$}

In this section we will discuss a well known construction of the interpolation category for the symmetric group due to Deligne \cite{deligne2007categorie} and its basic properties. For more on this topic see \cite{comes2009blocks,comes2012ideals,comes2014deligne,etingof2014representation,etingof2016representation}. We assume that $\Bbbk$ has characteristic $0$. 

We will start by introducing the system of vector spaces which is going to play a role of the homomorphism spaces in the corresponding skeletal category. Although these spaces are best understood using diagrams, we will omit this for the sake of space. We advise anyone seeing Deligne categories for the first time to see \cite{comes2009blocks} for a much clearer diagrammatic construction of $\Rep(S_\nu)$.

\begin{def0}
 Denote by $\Bbbk P_{n,m}$ a vector space over a field $\Bbbk$ with the basis given by all possible partitions of an $n+m$-element set. Diagrammatically an element of the basis is represented by two rows of $\bullet$'s, the first of length $n$ and the second of length $m$, where all $\bullet$'s belonging to the same part of the partition are connected by edges. So, in other words, it is a graph on $n+m$ vertices, the set of connected components of which corresponds to a partition of $n+m$ (The graphs with the same set of connected components represent the same basis element).
 
 Define a map $\phi_\nu^{n,m,k}: \Bbbk P_{m,k} \times  \Bbbk P_{n,m} \to \Bbbk P_{n,k}$ for $\nu \in \Bbbk$ as follows.
 Consider two basis elements $\lambda \in \Bbbk P_{n,m}$ and $\mu \in \Bbbk P_{m,k}$. Take a vertical concatenation of the graphical representations of the corresponding partitions (the last one on top) and identify the rows of length $m$. After this we are left with a partition of three rows of $\bullet$'s of length $n,m$ and $k$. Now let's denote by $l(\mu,\lambda)$ the number of connected components consisting purely of $\bullet$'s lying in the second row. Also regard a partition of rows $n,k$ consisting of the same connected components as the partition of rows $n,m,k$ but with elements of the second row deleted, and denote it by $\mu \cdot \lambda$. Then $\phi_\nu^{n,m,k}(\mu, \lambda) = \nu^{l(\mu,\lambda)}\mu \cdot \lambda$.
 
 Define $\Bbbk P_{n}(\nu)$ to be $\Bbbk P_{n,n}$ with a structure of an algebra given by the map $\phi_\nu^{n,n,n}$. This algebra is called the {\it partition algebra} and it was introduced by Purdon in \cite{purdon1991potts}.
\end{def0}

The spaces $\Bbbk P_{n,m}$ can be seen as limits of the homomorphism spaces $\Hom_{S_N}(\mathfrak h_N^{\otimes n}, \mathfrak h_N^{\otimes m})$, where $\mathfrak h_N$ is the permutation representation of $S_N$.

Using this we can define a preliminary skeletal\footnote{Here "skeletal" means that all isomorphism classes of objects consist of exactly one object} category $\Rep^0(S_\nu; \Bbbk)$:
\begin{def0}
  For $\nu \in \Bbbk$ we denote by   $\Rep^0(S_\nu;\Bbbk)$ a skeletal rigid symmetric monoidal $\Bbbk$-linear category with objects given by elements of $\mathbb Z_{\ge 0}$, which can be graphically represented by rows of $\bullet$'s, and denoted by $[n]$.
  
  The set of morphisms $\Hom_{\Rep^0(S_\nu; \Bbbk)}([n],[m])$ is equal to $\Bbbk P_{n,m}$ and the composition maps are given by $\phi_\nu^{n,m,k}$.
  
  Tensor product on objects is defined by the horizontal concatenation of rows and on morphisms by the horizontal concatenation of diagrams. All objects $[n]$ are self-dual.
\end{def0}

Using this we can define the Deligne category $\Rep(S_\nu; \Bbbk)$ itself:
\begin{def0}
 For $\nu \in \Bbbk$, the Deligne category $\Rep(S_\nu; \Bbbk)$ is the Karoubian envelope of the additive envelope of $\Rep^0(S_\nu; \Bbbk)$. 
 \end{def0}
 
 This means that we add all possible direct sums and direct summands into our category.

Below we will list a few pieces of notation and results concerning Deligne categories. They are well known and can be found for example in \cite{comes2009blocks,etingof2014representation}.

\begin{def0}
The object $[1]$ is called the permutation representation and is denoted by $\mathfrak h$.
  The object $[0]$ is called the trivial representation and is denoted by $\Bbbk$ (by a slight abuse of notation).
\end{def0}

The important properties of $\Rep(S_\nu; \Bbbk)$ are listed below:
\begin{prop}
\textbf{a)} For $\nu \notin \mathbb Z_{\ge 0}$ ${\rm Rep}(S_\nu ; \Bbbk)$ is a semisimple tensor category. \\
\textbf{b)} For $\nu \notin \mathbb Z_{\ge 0}$  simple objects of ${\rm Rep}(S_\nu; \Bbbk)$ are in 1-1 correspondence with Young diagrams of arbitrary size. They are denoted by $\mathcal X(\lambda)$. Moreover $\mathcal X(\lambda)$ is a direct summand in $[|\lambda|]$. \\ 
\textbf{c)} The categorical dimension of $\mathfrak h$ is $\nu$ and of $\Bbbk$ is $1$. \\
\textbf{d)} All $\mathcal X(\lambda)$ are self-dual.
\end{prop}

The Deligne category enjoys a certain universal property:
\begin{prop}\label{unprop}
(8.3 in \cite{deligne2007categorie})
For any $\Bbbk$-linear Karoubian symmetric monoidal category $\mathcal T$, the category of $\Bbbk$-linear symmetric monoidal functors from ${\rm Rep}(S_\nu; \Bbbk)$ to $\mathcal T$ is equivalent to the category $\mathcal T^f_\nu$ of commutative Frobenius algebras in $\mathcal T$ of dimension $\nu$. The equivalence sends a functor $F$ to the object $F(\mathfrak h)$.
\end{prop}

The important consequence of this result is that for every commutative Frobenius algebra $A$ in a Karoubian symmetric category $\mathcal T$ of dimension $\nu$, we have a symmetric monoidal functor from $\Rep(S_\nu; \Bbbk)$ to $\mathcal T$ which sends $\mathfrak h$ to $A$.

\begin{rem} 
Here by a commutative Frobenius algebra in $\mathcal T$ we mean an object $A$ with the  following structure. It is an associative commutative algebra with the corresponding algebraic structure given by $\mu_A,1_A$, and if we define a map:
$$
\Tr: A \xrightarrow{1 \otimes {\rm coev}_A} A \otimes A \otimes A^* \xrightarrow{\mu_A \otimes 1} A \otimes A^* \xrightarrow{{\rm ev}_A} \mathbbm{1},
$$
then the pairing $A \otimes A \xrightarrow{\mu_A} A \xrightarrow{\Tr} \mathbbm{1}$ is required to be non-degenerate, i.e., it corresponds to an isomorphism between $A$ and $A^*$ under the identification of $\Hom_{\mathcal T}(A\otimes A,\mathbbm{1})$ with $\Hom_{\mathcal T}(A, A^*)$.
\end{rem}

In the rest of the paper we will use Deligne categories over the following fields:
\begin{def0} \label{extdef}
For $\nu \in \mathbb C$ set $\Rep(S_\nu) := \Rep(S_\nu;\mathbb C)$. For $\nu \in \overline{\mathbb C(\nu)}$ set $\Rep^{\rm ext}(S_\nu) := \Rep(S_\nu; \overline{\mathbb C(\nu)})$.
\end{def0}

\subsection{Deligne categories $\Rep(S_\nu)$ and $\Rep(S_\nu \ltimes \Gamma^\nu)$ as ultraproducts}

\subsubsection{The category $\Rep(S_\nu)$ as an ultraproduct}

In this section we will  show how to construct $\Rep(S_\nu)$ using ultraproducts, and discuss some important consequences of this construction. This method is very useful, because it allows one to transfer all kinds of constructions and their properties from the case of finite rank categories almost automatically. The main ideas of this approach were contained in \cite{deligne2007categorie},\cite{harman2016deligne}\footnote{For the similar discussion about $\Rep(GL_\nu)$ see \cite{deligne2007categorie}, \cite{harman2016deligne}, \cite{kalinov2018finite}.}.

The idea is to construct the category $\Rep(S_\nu)$ for non-integer $\nu$ as a full subcategory in the ultraproduct category following Example \ref{catex}. We have the following result (See the introduction of \cite{deligne2007categorie} or Theorem 1.1 in \cite{harman2016deligne}):

\begin{thm} \label{ultdelthm}

\textbf{a)} Suppose $\nu\in \mathbb C$ is transcendental. Consider $\widehat{\mathcal C} = \prod_{\mathcal{F}} \textbf{Rep}^f_0(S_n)$. Set $\mathfrak h_\nu:= \prod^C_{\mathcal{F}}\mathfrak h_n$. Fix an isomorphism $\prod_{\mathcal F}\overline{\mathbb Q}\simeq \mathbb C$ such that $\prod_{\mathcal F} i = \nu$. Then the full subcategory of the $\prod_{\mathcal F}\overline{\mathbb Q}$-linear category $\widehat{\mathcal C}$ generated by $\mathfrak h_\nu$ under taking tensor products, direct sums and direct summands is equivalent to the $\mathbb C$--linear category $\Rep(S_\nu)$, in a way consistent with the fixed isomorphism $\prod_{\mathcal F}\overline{\mathbb Q} \simeq \mathbb C$.

\textbf{b)} Suppose $\nu \in \mathbb C$ is algebraic but not a nonnegative integer. Fix a sequence of distinct primes $p_n$, a sequence of integers $\nu_n$, and an isomorphism $\prod_{\mathcal F}\overline{\mathbb F}_{p_n}\simeq \mathbb C$ such that \linebreak $\prod_{\mathcal F}\nu_n = \nu$. Set $\widehat{\mathcal C}: = \prod_{\mathcal{F}} \textbf{Rep}^f_{p_n}(S_{\nu_n})$. Set $\mathfrak h_\nu := \prod_{\mathcal{F}}^C\mathfrak h_{p_n}^{\nu_n}$. Then the  full subcategory of the $\prod_{\mathcal F}\overline{\mathbb F}_{p_n}$-linear category $\widehat{\mathcal C}$ generated by $\mathfrak h_\nu$ under taking  tensor products, direct sums and direct summands is equivalent to the $\mathbb C$-linear category $\Rep(S_\nu)$,  in a way consistent with the fixed isomorphism $\prod_{\mathcal F}\overline{\mathbb F}_{p_n}\simeq \mathbb C$.
\end{thm}
\begin{proof}
\textbf{a)} The required isomorphism of fields exists by Example \ref{fieldextrans}. So we have a Karoubian symmetric monoidal category $\widehat{\mathcal C}$ linear over $\mathbb C$, with an object $\prod_{\mathcal F}^C\mathfrak h_n$ of dimension $\nu$. Since every $\mathfrak h_n$ is a commutative Frobenius algebra, it follows by \lthm${}$ that $\mathfrak h_\nu$ is also a commutative Frobenius algebra. Hence by Proposition \ref{unprop} we obtain a symmetric monoidal functor $F:\Rep(S_\nu) \to \widehat{\mathcal C}$ which takes $\mathfrak h$ to $\mathfrak h_{\nu}$. Since $\Rep(S_\nu)$ is generated by $\mathfrak h$ under taking  tensor products, direct sums and direct summands, it follows that the image of $\Rep(S_\nu)$ under $F$ is  the full subcategory $\mathcal C$ in $\widehat{\mathcal C}$ generated by $\mathfrak h_{\nu}$ under taking  tensor products, direct sums and direct summands. So we know that $F:\Rep(S_\nu) \to \mathcal C$ is essentially surjective. Now it is enough to prove that it is fully faithful. 

Note that it is enough to prove that 
$$
\prodF \Hom_{S_n}(\mathfrak h_n^{\otimes r}, \mathfrak h_n^{\otimes s}) = \Hom_{\Rep(S_\nu)}([r], [s]),
$$
and that the composition maps are the same. Indeed, if this is true, both categories can be obtained as the Karoubian envelopes of the additive envelopes of the categories consisting of all $[r]$ or  $\mathfrak h_\nu^{\otimes r}$ respectively.

But this follows from Theorem 2.6 in \cite{comes2009blocks}. Indeed, there it is stated that there is an isomorphism between $\overline{\mathbb Q}P_{r,s}$ and $\Hom_{S_{n}}(\mathfrak h_n^{\otimes r},\mathfrak h_n^{\otimes s})$ for $n>r+s$. So for almost all $n$ we have $\Hom_{S_n}(\mathfrak h_n^{\otimes r}, \mathfrak h_n^{\otimes s}) = \overline{\mathbb Q}P_{n,m}$. Also Proposition 2.8 in the same article states that under this isomorphism the composition rule on $\Hom_{S_n}(\mathfrak h_n^{\otimes r}, \mathfrak h_n^{\otimes s})$ transforms into the composition rule on $\overline{\mathbb Q}P_{r,s}$ in the definition of $\Rep^0(S_\nu)$. So it follows that, indeed, $\prod_{\mathcal F} \Hom_{S_n}(\mathfrak h_n^{\otimes r}, \mathfrak h_n^{\otimes s}) = \Hom_{\Rep(S_t)}([r], [s])$, and the composition rule is the same.

\textbf{b)} Again the required isomorphism exists by Example \ref{fieldexalg}. The rest of the proof is the same since the representation theory of $S_{n}$ is the same in zero characteristic and in characteristic $p>n$, and $p_n>\nu_n$ for almost all $n$.
\end{proof}

\begin{rem}
Note that for the purposes of this theorem we could also have used the categories $\Repb_{p_n}(S_{\nu_n})$.
\end{rem}

We can also formulate a similar result for $\Rep^{\rm ext}(S_\nu)$:
\begin{cor} \label{ultdelcor}
Fix an isomorphism $\prodF \FQ \simeq \overline{\mathbb C(\nu)}$ such that $\prodF n = \nu$.  Set \linebreak $\widehat{\mathcal C} = \prod_{\mathcal{F}} \textbf{Rep}^f_{0}(S_{n})$. Set $\mathfrak h_\nu = \prod_{\mathcal{F}}^C\mathfrak h_{n}$. Then the  full subcategory of the $\prod_{\mathcal F}\FQ$-linear category $\widehat{\mathcal C}$ generated by $\mathfrak h_\nu$ under taking  tensor products, direct sums and direct summands is equivalent to the $\Cext$-linear category $\Rep(S_\nu)$,  in a way consistent with the fixed isomorphism $\prod_{\mathcal F}\FQ\simeq \Cext$.
\end{cor}
\begin{proof}
This follows from the above Theorem and the fact that $\mathbb C \simeq \overline{\mathbb C(\nu)}$ (see Example \ref{fieldextrans}).
\end{proof}

\begin{rem}
As mentioned in the beginning of Section \ref{sectnot}, to treat the algebraic and transcendental cases simultaneously, it's useful to agree on the convention that by $\overline{\mathbb F}_0$ we will mean $\overline{\mathbb Q}$, and so the case $\nu_n = n$, $p_n = 0$ in the setting of part $(b)$ of the Theorem \ref{ultdelthm} gives us transcendental $\nu$. Also below we will always assume that the sequences $p_n$ and $\nu_n$ are the sequences from Theorem \ref{ultdelthm} or Corollary \ref{ultdelcor} corresponding to the given $\nu$. Finally, we will work only with $\nu \in \mathbb C \backslash \mathbb Z_{\ge 0}$.
\end{rem}

Now we would like to explain why this construction of the Deligne categories is quite useful. To begin with, we would like to construct the simple objects $\mathcal X(\lambda)$ as ultraproducts. This is easy to do, using the notation from Definition \ref{defyoung}:
\begin{prop} \label{simpprop}
The irreducible object $\mathcal X(\lambda)$ of ${\rm Rep}(S_\nu)$ can be obtained as an ultraproduct of irreducible objects of ${\bf Rep}^f_{p_n}(S_{\nu_n})$ as $\mathcal X(\lambda) = \prod_{\mathcal F}^CX_{\nu_n}(\lambda|_{\nu_n})$.
\end{prop}
\begin{proof}
From Section 3.3 of \cite{comes2009blocks} we know that the algebras $\Bbbk P_r(\nu)$ for $\nu \ne 0,1, \dots, 2r$ have the same set of idempotents obtained by specialization from idempotents of $\Bbbk(x)P_r(x)$. Now by construction all simple objects of $\Rep(S_\nu)$ are given by the primitive idempotents of  $\End_{\Rep^0(S_\nu;\Bbbk)}([r]) = \Bbbk P_r(\nu)$. And by Theorem \ref{ultdelthm}, $\Bbbk P_r(\nu) \simeq \prodF \Fpn P_r(\nu_n)$ in such a way that basis elements are ultraproducts of basis elements. Thus it follows that idempotents in $\Bbbk P_r(\nu)$ are given by the ultraproducts of the same idempotents {\fan}. And so the claim follows.
\end{proof}

This result allows us to reformulate the definition of $\Rep(S_\nu)$ as an ultraproduct.
\begin{prop} \label{simplerultrthm}
In the notation of Theorem \ref{ultdelthm} the category ${\rm Rep}(S_\nu)$ can be described as the full subcategory of $\widehat{\mathcal C} ={\bf Rep}^f_{p_n}(S_{\nu_n})$ consisting of sequences of objects $Y_n= \bigoplus_{\alpha \in A_n} X_{p_n}(\lambda_{n,\alpha})$ for some indexing sets $A_n$ and Young diagrams $\lambda_{n,\alpha}$ such that both the sequence of $|A_n|$ and the sequence of $\max_{\alpha \in A_n}(|\lambda_{n,\alpha}| - (\lambda_{n,\alpha})_1)$, where $(\lambda_{n,\alpha})_1$ is the length of the first row, are bounded {\fan}.
\end{prop}
\begin{proof}
We know that $\Rep(S_\nu)$ is a full subcategory of $\widehat{\mathcal C}$ so we just need to match the objects.

On the one hand, suppose $Y \in \Rep(S_\nu)$. We know that for some set of Young diagrams $\mu_{\alpha}$ with $\alpha \in A$, a finite indexing set, we have $Y = \bigoplus_{\alpha \in A}\mathcal X(\mu_{\alpha})$, so from Proposition \ref{simpprop} it follows that $Y = \prodF^C \bigoplus_{\alpha \in A} X_{p_n}(\mu_{\alpha}|_{\nu_n})$. Thus we have a required sequence with $A_n = A$ and $\lambda_{\alpha,n} = \mu_{\alpha}|_{\nu_n}$. The sequence $|A_n|=A$ is constant, hence so is the sequence $\max_{\alpha \in A_n}(|\lambda_{n,\alpha}| - (\lambda_{n,\alpha})_1) = \max_{\alpha \in A}(|\mu_{\alpha}|)$.

On the other hand, suppose we have a sequence described in the statement of the Theorem. Since we know that $|A_n|$ is bounded {\fan}, there is a finite number of options for the cardinality of $|A_n|$ {\fan}, thus from part 2 of Lemma \ref{ultrlemma} it follows that {\fan}  the cardinality is the same. Fix $A$ to be a set of this cardinality. So, {\fan}  we have $Y_n = \bigoplus_{\alpha \in A} X_{p_n}(\lambda_{n,\alpha})$. Suppose $\max_{\alpha \in A_n}(|\lambda_{n,\alpha}| - (\lambda_{n,\alpha})_1)$ is bounded by $L$. Now each $\lambda_{n,\alpha}$ is a Young diagram of weight $\nu_n$ with at most $L$ boxes in the rows above the first one. I.e., for $n$ big enough (namely, $\nu_n > 2L$), it follows that each $\lambda_{n,\alpha} = \mu_{n,\alpha}|_{\nu_n}$ where $\mu_{n,\alpha}$ is a Young diagram of weight at most $L$. So {\fan}  each $Y_n$ is uniquely determined by a collection of $|A|$ Young diagrams of weight at most $L$. Notice that there is only a finite number of such collections. So by the same Lemma it follows that {\fan}  the collection is the same. Denote it by $\{\mu_{\alpha}\}_{\alpha \in A}$. Hence, {\fan}  up to a permutation we have $Y_n = \bigoplus_{\alpha \in A}X_{p_n}(\mu_{\alpha}|_{\nu_n})$. Hence we have $\prodF^C Y_n = \bigoplus_{\alpha \in A} \mathcal X(\mu_{\alpha})$ which is indeed an object of $\Rep(S_\nu)$.
\end{proof}

So, as promised in Example \ref{catex}, $\Rep(S_\nu)$ can indeed be described as given by ultraproducts bounded in a certain sense. 

We will also need to explain how to interpolate the central element $\Omega_n \in \Bbbk[S_n]$ to $\Rep(S_\nu)$. Recall that we can consider the central elements of $\Bbbk[S_{\nu_n}]$ as endomorphisms of the identity functor of $\Repb_{p_n}(S_{\nu_n})$.
\begin{def0} \label{delcentr} 
Denote by $\Omega$ the endomorphism of the identity functor of $\Rep(S_\nu)$ given by the restriction of the endomorphism $\prodF \Omega_{\nu_n}$.
\end{def0}
One can easily calculate the action of $\Omega$ on simple objects.
\begin{prop} 
\cite{etingof2014representation} The action of $\Omega$ on an object $\mathcal X(\lambda)$ is given by: $$
\Omega|_{\mathcal X(\lambda)} = \left(\ct(\lambda) - |\lambda| +\frac{(\nu - |\lambda|)(\nu - |\lambda|-1)}{2}\right) 1_{\mathcal X(\lambda)}  .
$$
\end{prop}
\begin{proof}
Since $\mathcal X(\lambda) = \prodF^C X_{p_n}(\lambda|_{\nu_n})$, one needs to calculate $\prodF \ct(\lambda|_{\nu_n})$. It's easy to see that  each box of $\lambda$ contributes an extra $-1$ to the content of $\lambda|_{\nu_n}$, also $\nu_n- |\lambda|$ new boxes in the first row contribute $0 + 1 + \dots + (\nu_n-|\lambda|-1)$ to the content of $\lambda|_{\nu_n}$, thus we have:
$$
\prodF \ct(\lambda|_{\nu_n}) = \prodF \left(\ct(\lambda) - |\lambda| + \frac{(\nu_n-|\lambda|)(\nu_n-|\lambda|-1)}{2}\right) = 
$$
$$
=\ct(\lambda) - |\lambda| + \frac{(\nu -|\lambda|)(\nu -|\lambda|-1)}{2},
$$
which is exactly the value in the statement of the proposition.
\end{proof}

\begin{rem}
Note that all of the results of this Section work mutatis mutandis for $\Rep^{\rm ext}(S_\nu)$ (see Definition \ref{extdef}).
\end{rem}

Now we would like to give the reader a general idea of how this can be used to transfer constructions and facts from  representation theory in  finite rank to the context of Deligne categories.

Suppose we have a representation-theoretic structure $\mathcal Y_n$ in each $\Repb_{p_n}(S_{\nu_n})$ which can be constructed uniformly in an element-free way for every $n$. Then we can define the same structure $\mathcal Y$ in $\Rep(S_\nu)$ using the analogs of the same objects and maps. Since the definitions are the same, it would follow that $\mathcal Y = \prodF \mathcal Y_n$. Now one can try to transfer the properties of $\mathcal Y_n$ to $\mathcal Y$. For some it can be as easy as a direct application of \lthm. Others require quite a bit of technical work before one can do that. For some interesting results of this type see \cite{kalinov2018finite,harman2020classification}.

Oftentimes the structure $\mathcal Y$ might include some ind-objects of $\Rep(S_\nu)$. This will happen, for example, when we will try to define the rational Cherednik algebra  in $\Rep(S_\nu)$. Thus we  will deal with ind-objects in the ultraproduct setting in the next subsection.

\subsubsection{Ind-objects of $\Rep(S_\nu)$ as restricted ultraproducts}

In this section we are going to explain how ind-objects of $\Rep(S_\nu)$ can be obtained as restricted ultraproducts, thus extending Theorem \ref{ultdelthm} in a certain way. 

To do that, we will use the result of Construction \ref{indcons}.

\begin{prop} \label{inddelcons}
Suppose we have a sequence of representations of $M_n \in {\bf Rep}_{p_n}(S_{\nu_n})$, with fixed filtration by subrepresentations of finite length. i.e., we have $F^iM_n \in {\bf Rep}^f_{p_n}(S_{\nu_n})$ such that $\bigcup_{i \in \mathbb N}F^iM_n = M_n$. Also suppose that $\prodF^CF^iM_n \in {\rm Rep}(S_\nu)$. Then it follows that $M = \prodF^{C,r}M_n = \bigcup_{i \in \mathbb N}\prodF^C F^iM_n$ is an object of ${\rm IND}({\rm Rep}(S_\nu))$.\footnote{One can also define, through a more involved construction, the category $\IND(\Rep(S_\nu))$ as a subcategory of $\prodF \Repb_{p_n}(S_{\nu_n})$. Note that this subcategory will not be full. In this way one would also be able to consider $\prodF^C\bigcup_{i \in \mathbb N}F^iM_n$, i.e., take the ultraproduct directly. It can be shown that this would define the same object $M$.}
\end{prop}
\begin{proof}
This follows from Construction \ref{indconsult}.
\end{proof}

\begin{rem}
Note that, using Remark \ref{remindmor}, we conclude that if $M \in \IND(\Rep(S_\nu))$ has finite length, then for any $N \in \IND(\Rep(S_\nu))$ constructed via Proposition \ref{inddelcons}, we have:
$$
\Hom_{\IND(\Rep(S_\nu))}(M,N) = \bigcup_{j \in \mathbb N}\Hom_{\Rep(S_\nu)}(M,F^jN)=\bigcup_{j \in \mathbb N}\prodF\Hom_{\Repb_{p_n}(S_{\nu_n})}(M_{n},F^jN_{n})=
$$
$$
= {\prodF^r}\Hom_{\Repb_{p_n}(S_{\nu_n})}(M_{n},N_{n})  ,
$$
with the filtration arising from the filtration on $N$. 
\end{rem}

\subsubsection{The category $\Rep(S_\nu \ltimes \Gamma^\nu)$}

In this section we will explain how the category of representations of the wreath product in complex rank can be constructed. 

There are several ways to approach this problem. One construction was developed by Knop in \cite{knop2007tensor}. Another approach can be found in \cite{mori2012representation}. However, in the present paper we will use a different approach, outlined in \cite{etingof2014representation}. For brevity we will only address the case of transcendental $\nu$ in this section, although with slight modifications the results can be extended to the algebraic case as well.

Below we will use the notion of a  unital vector space. For details see \cite{etingof2014representation}.

\begin{def0}
A unital vector space $V$ is a vector space together with a unit, i.e., a distinguished non-zero vector denoted by $1 \in V$.
\end{def0}

In \cite{etingof2014representation} it is shown that given a finite dimensional unital vector space $V$, one can define an ind-object $V^{\otimes \nu}\in \Rep(S_\nu)$. The idea behind this is that, although there is no way to algebraically define $x^t$, there is such a way to define $(1+x)^t:=\sum_{m\ge 0}\binom{t}{m}x^m$. 

We can also construct this object via an ultraproduct. Anyone not familiar with \cite{etingof2014representation} might regard this as definition for the purposes of this paper.

Note that the $S_n$-module $V^{\otimes n}$ has a natural filtration induced by the filtration on $V$ given by $F^0V=\Bbbk 1$, $F^1V=V$. 

\begin{prop}\label{compow}
For a finite dimensional unital vector space $V$, the ind-object $V^{\otimes \nu}$ is given by:
$$
V^{\otimes \nu} = \prodF^{C,r} V^{\otimes n}. 
$$
\end{prop}
\begin{proof}
Using the notation of \cite{etingof2014representation}, we have:
$$
V^{\otimes n} = \bigoplus_{\lambda}S^{\lambda|_n}V \otimes X(\lambda|_n) ,
$$
where $S^{\lambda|_{n}}$ are the corresponding Schur functors, and 
$$
F^iV^{\otimes n} = \bigoplus_{|\lambda|\le i} S^{\lambda|_n}V \otimes X(\lambda|_n).
$$
Thus, we obtain
$$
\prodF^{C,r} V^{\otimes n} = \bigcup_i \bigoplus_{|\lambda|\le i} \left( \prodF S^{\lambda|_n}V\right) \otimes \mathcal X(\lambda) = \bigoplus_{\lambda} S^{\lambda,\infty}V \otimes \mathcal X(\lambda)=V^{\otimes \nu},
$$
as needed. 
\end{proof}

Now consider a finite subgroup $\Gamma \subset \mathrm{SL}(2,\FQ)$. Proposition \ref{compow} allows us to define the following algebra:
\begin{def0}
An ind-object $\mathbb C[\Gamma]^{\otimes \nu} $ is constructed via Proposition \ref{compow} starting with $\FQ[\Gamma]$ as a unital vector space. It has the structure of the algebra given by the ultraproduct of the algebra structures on $\FQ[\Gamma]^{\otimes n}$.
\end{def0}

Using this, one can define the category $\Rep(S_\nu\ltimes \Gamma^{\nu})$ in the following way:
\begin{def0}
The category $\Rep(S_\nu\ltimes \Gamma^{\nu})$ is the category of $\mathbb C[\Gamma]^{\otimes \nu}$-modules in $\Rep(S_\nu)$. I.e., its objects are objects of $\Rep(S_\nu))$ with the structure of a $\mathbb C[\Gamma]^{\otimes \nu}$-module,  and its morphisms are morphisms in $\Rep(S_\nu)$ which commute with the module structure.
\end{def0}

It can be shown that $\Rep(S_\nu\ltimes \Gamma^{\nu})$ is equivalent to the wreath product category defined by Knop. 

We can construct some of the objects of $\Rep(S_\nu\ltimes \Gamma^{\nu})$ as ultraproducts.
\begin{prop} \label{wreathprop}
Consider a sequence of modules $M_n \in {\bf Rep}_0(S_n \ltimes \Gamma^n)$ whose ultraproduct as $S_n$-modules is a well-defined object of $\Rep(S_\nu)$. Then, this ultraproduct also lies in $\Rep(S_\nu\ltimes \Gamma^{\nu})$.
\end{prop}
\begin{proof}
Denote $M = \prodF^{C} M_n$. Indeed since $M_n$ has a structure of a $\FQ[\Gamma]^{\otimes n}$-module in $\Repb_0(S_n)$, it follows that $M$ has a structure of $\prodF^C \FQ[\Gamma]^{\otimes n} = \mathbb C[\Gamma]^{\otimes \nu}$-module. Hence it is an object of $\Rep(S_\nu\ltimes \Gamma^{\nu})$.
\end{proof}

In this way we can interpolate irreducible objects of $\Repb_0(S_n\ltimes \Gamma^n)$.
\begin{def0}
In the notation of Proposition \ref{wreathprop1}, consider $\lambda$ to be any function:
$$
\lambda: A \to \{\text{Partitions}\}.
$$
Denote by $\mathcal X(\lambda)$ the object of $\Rep(S_\nu \ltimes \Gamma^\nu)$ defined as:
$$
\mathcal X(\lambda) = \prodF^C X(\lambda_n)  ,
$$
where $\lambda_n({\rm triv}) = \lambda({\rm triv})|_n$ and $\lambda_n(\alpha) = \lambda(\alpha)$ for all other irreducibles $\alpha$ of $\Gamma$.

It follows that $\mathcal X(\lambda)$ is irreducible.
\end{def0}

\begin{rem}
We leave out the proof of the fact that these ultraproducts indeed define an object of $\Rep(S_\nu)$. This can be done using the results of \cite{knop2007tensor}, but we do not need this for this paper.
\end{rem}

\subsection{Cherednik algebras in complex rank}

\subsubsection{Cherednik algebra of type A in complex rank}

In this subsection we will explain how to construct the interpolation category for the representations of the rational Cherednik algebra of type A. After that we will construct an induction functor interpolating the functors $\Ind_{S_n}^{H_{t,k}(n)}$. This will allow us to define the DDC-algebra below. One can find more information about the rational Cherednik algebras in complex rank in \cite{entova2014representations}.

The definition of $\Rep(H_{t,k}(\nu))$ mimics the definition of  representations in the finite rank in an element-free way:
\begin{def0} \label{ultcherdef}
The category $\Rep(H_{t,k}(\nu))$ is defined as follows. The objects are given by  triples $(M,x,y)$, where $M$ is an ind-object of $\Rep(S_\nu)$, $x$ is a map  $x: \mathfrak h^* \otimes M \to M$ and $y$ a map $y: \mathfrak h \otimes M \to M$, both of which are morphisms in $\IND(\Rep(S_{\nu}))$. They also satisfy the following conditions:
$$
x \circ (1\otimes x) -x \circ (1\otimes x) \circ (\sigma \otimes 1) = 0  , 
$$
as a map from $\mathfrak h^* \otimes \mathfrak h^* \otimes M$ to $M$;
$$
y \circ (1\otimes y) -y \circ (1\otimes y) \circ (\sigma \otimes 1) = 0  , 
$$
as a map from $\mathfrak h \otimes \mathfrak h \otimes M$ to $M$;
$$
y \circ (1\otimes x) -x \circ (1\otimes y) \circ (\sigma \otimes 1) = t \cdot {\rm ev}_{\mathfrak h} \otimes 1 - k \cdot ({\rm ev}_{\mathfrak h}\otimes 1) \circ( \Omega^{3} - \Omega^{1,3}) , 
$$
as a map  from $\mathfrak h \otimes \mathfrak h^* \otimes M$ to $M$, 
where $\Omega$ is a central element from Definition \ref{delcentr}, and indices indicate the spaces on which $\Omega$ acts in the tensor product $\mathfrak h \otimes \mathfrak h^* \otimes M$.

The morphisms of $\Rep(H_{t,k}(\nu))$ are the morphisms of $\IND(\Rep(S_\nu))$ which commute with the action-maps $x$ and $y$.

Also by $\Rep^{\rm ext}(H_{t,k}(\nu))$ denote the similar category constructed over $\Rep^{\rm ext}(S_\nu)$.
\end{def0}

Note that, since our objects already are ``$S_\nu$-modules" we don't need to define any additional ``$S_\nu$-action". 

The last formula in the definition may need some explanation. To clarify it, let us apply it to $y_i \otimes x_j \otimes M$ in the finite rank case. We have:
$$
[y_i,x_j] = t\delta_{ij} - k\delta_{ij}\Omega + k\delta_{ij}\sum_{m<l,m\ne i, l\ne i}s_{ml} + k(1-\delta_{ij})s_{ij} = t\delta_{ij}-k\delta_{ij}\sum_{m\ne i}s_{im} + k(1-\delta_{ij})s_{ij}  ,
$$
which is precisely the formula from Definition \ref{cherAdef}. Hence we see that this is indeed the finite-rank definition rewritten in an element-free way.

Now we would like to show how we can construct some of the objects of  the category $\Rep(H_{t,k}(\nu))$ as ultraproducts.

\begin{rem}
Below we will denote by $t_n,k_n$ the elements of $\Fpn$  such that $\prodF t_n = t$ and $\prodF k_n = k$ under the fixed isomorphism of $\prodF \Fpn \simeq \mathbb C$. We will use the similar notation for all other parameters of algebras used in the paper.
\end{rem}

\begin{lemma} \label{thmcherault}
Suppose $M_n$ is a sequence of objects of ${\bf Rep}_{p_n}(H_{{t_n},{k_n}}(\nu_n))$ such that their (restricted) ultraproduct as objects of ${\bf Rep}_{p_n}(S_{\nu_n})$ lies in ${\rm IND}({\rm Rep}(S_\nu))$. Suppose $x_n$ and $y_n$ are the maps which define the action of generators of the corresponding Cherednik algebra on $M_n$. Then $(\prodF^{C,r}M_n, \prodF x_n, \prodF y_n)$ defines an object of ${\rm Rep}(H_{t,k}(\nu))$.
\end{lemma}
\begin{proof}
It's easy to see that the data $(\prodF^{C,r}M_n, \prodF x_n, \prodF y_n)$ is well defined. Since $x_n$ and $y_n$ satisfy the same conditions in finite rank and complex rank it follows that by \lthm${}$ this is indeed an object of $\Rep(H_{t,k}(\nu))$.
\end{proof}

Now we would like to construct an interpolation of the functors $\Ind_{S_{\nu_n}}^{  H_{t_n,k_n}(\nu_n)}$. It is possible to construct the full functor as ultraproduct directly, but this functor would a priori have $\prodF \Repb_{p_n}(H_{t_n,k_n}(\nu_n))$ as its target category, so we would need to explain why the functor really gives us objects of $\Rep(H_{t,k}(\nu))$. Instead we will construct this functor directly, which will also show that it agrees with the ultraproduct functor when applied to objects of $\Rep(S_\nu)$.

The idea is, following the PBW theorem, to think about ``$ H_{t,k}(\nu)$" as ``the direct sum $\bigoplus_{i,j \ge 0}S^i(\mathfrak h^*) \otimes S^j(\mathfrak h) \otimes \mathbb C[S_\nu]$" and take the tensor product with $V \in \Rep(S_\nu)$ ``over $\mathbb C[S_\nu]$".

\begin{cons} \label{inducdef}
 For an object $V \in \Rep(S_\nu)$, consider an ind-object $I_V= \oplus_{i,j\ge 0}I_{i,j}$, where $I_{i,j} = S^i(\mathfrak h^*) \otimes S^j(\mathfrak h ) \otimes V$, and maps $x_V: \mathfrak h^* \otimes I_V \to I_V$ and $y: \mathfrak h \otimes I_V \to I_V$, which are defined as follows.

First, note that $S^{i+1}(\mathfrak h)$ is isomorphic to a direct summand of $\mathfrak h \otimes S^{i}(\mathfrak h)$, let's denote the corresponding inclusion and projection as $\iota_{i+1, y}$ and $\pi_{i+1,y}$ respectively. The same is true for $\mathfrak h^*$, the corresponding morphisms are $\iota_{i+1,x}$ and $\pi_{i+1,x}$.

Now define $(x_V)|_{I_{i,j}}: \mathfrak h^*\otimes I_{i,j} \to I_{i+1,j}$ to be equal to $\pi_{i+1,x} \otimes 1$ for all $i,j$. Also define $(y_V)|_{I_{0,j}}: \mathfrak h\otimes I_{0,j} \to I_{0,j+1}$ as $\pi_{j+1,y}\otimes 1$. And lastly we define $(y_V)|_{I_{i,j}}: \mathfrak h \otimes I_{i,j} \to I_{i,j+1} \oplus I_{i-1,j}$ by induction in $i$ as:
$$
\left[(x \otimes 1) \circ(1 \otimes y \otimes 1) \circ (\sigma \otimes 1) + t \cdot {\rm ev}_{\mathfrak h} \otimes 1 - k\cdot ({\rm ev}_{\mathfrak h } \otimes 1) \circ (\Omega^{I_{i-1,j}} - \Omega^{\mathfrak h, I_{i-1,j}})\right]\circ(1 \otimes \iota_{i,x} \otimes 1)  .
$$ 
\end{cons}

Now we would like to show that this defines an object of $\Rep( H_{t,k}(\nu))$. Indeed:
\begin{lemma} \label{indlemma1}
In the notations of Construction \ref{inducdef}, the triple $(I_V,x_V,y_V)$ defines an object of $\Rep( H_{t,k}(\nu))$.
\end{lemma}
\begin{proof}
Indeed, the first two formulas of Definition \ref{ultcherdef} are satisfied by the properties of symmetric powers, and we defined the action of $y_V$ by induction in such a way that the third equation is also satisfied.

Another way to see that is to note that in the finite rank case this construction amounts to $ H_{{t_n},{k_n}}(\nu_n) \otimes_{S_{\nu_n}} V_n$, and so by \lthm, we do get a correct structure of an \linebreak ``$ H_{t,k}(\nu)$-module".
\end{proof}

Now we need to construct the action of the induction functor on  morphisms.
\begin{cons} \label{indchermordef}
In the notation of Construction \ref{inducdef}, given a morphism \linebreak $\phi: V \to U$, define a  morphism $I_\phi: I_V\to I_U$ in the following way: $(I_\phi)|_{S^i(\mathfrak h^*) \otimes S^j(\mathfrak h) \otimes V}:=1 \otimes 1 \otimes \phi$.
\end{cons}
\begin{lemma} \label{indlemma2}
In the notation of Constructions \ref{inducdef} and \ref{indchermordef}, $I_\phi$ is a morphism in ${\rm Rep}( H_{t,k}(\nu))$.
\end{lemma}
\begin{proof}
This is easy to see both straight from the definition, or by the ultraproduct argument, since in finite rank this defines an actual $ H_{t_n,k_n}(\nu_n)$-module morphism.
\end{proof}

Now we can define the actual functor:
\begin{def0} \label{indfunctor}
Define a functor $\Ind_{S_\nu}^{ H_{t,k}(\nu)}: {\rm Rep}(S_\nu) \to \Rep(H_{t,k}(\nu))$ in the following way. On objects it takes $V$ to the triple $(I_V,x_V,y_V)$ from Construction \ref{inducdef}. And on morphisms it takes $\phi: V \to U$ to $I_\phi$ from Construction \ref{indchermordef}. This is a well defined functor by Lemmas \ref{indlemma1} and \ref{indlemma2}.
\end{def0}
The next Corollary follows by construction and the above lemmas:
\begin{cor} \label{indcor}
For any object $V \in {\rm Rep}(S_\nu)$ such that $V = \prodF V_n$ we have:
$$
{\rm Ind}_{S_\nu}^{ H_{t,k}(\nu)}V = \prodF^{C,r} {\rm Ind}_{S_{\nu_n}}^{ H_{{t_n},{k_n}}(\nu_n)}V_n  ,
$$
where the filtration on ${\rm Ind}_{S_{\nu_n}}^{ H_{{t_n},{k_n}}(\nu_n)}V_n$ is obtained from the filtration of $H_{{t_n},{k_n}}(\nu_n)$ given by $\deg(x_i)=\deg(y_i)=1$ and $\deg(\sigma_{ij}) = 0$.
\end{cor}

\begin{rem}
All of the constructions of the present section work for $\Rep^{\rm ext}( H_{t,k}(\nu))$ in the same fashion.
\end{rem}

\subsubsection{Symplectic reflection algebras in complex rank}

In this section we will briefly generalize the results of the previous section to the context of symplectic reflection algebras. As in Section 3.2.3, we will work for transcendental $\nu$ for simplicity. Also as in that section, we fix a finite group $\Gamma \subset \textrm{SL}(2,\FQ)$.

Below we will define the category $\Rep(H_{t,k,c}(\nu,\Gamma))$ following the lines of Definition \ref{ultcherdef}. To do this, we need to find the  analog of $V$ in Definition \ref{sympfindef}.
\begin{prop}
The ultraproduct $\prodF^C (\FQ^2)^n$ as $S_n\ltimes \Gamma^n$-modules defines an object of ${\rm Rep}(S_\nu\ltimes \Gamma^{\nu})$. 
\end{prop}
\begin{proof}
Indeed as $S_n$-modules, each $(\FQ^2)^n =\mathfrak h_n \oplus \mathfrak h_n$, hence their ultraproduct is given by $\mathfrak h^2$ as an object of $\Rep(S_\nu)$. Thus by Proposition \ref{wreathprop} it follows that it is also an object of $\Rep(S_\nu\ltimes \Gamma^{\nu})$. The symplectic pairing is given by the ultraproduct of symplectic pairings.
\end{proof}
We will denote this object by $V$ and call the fundamental representation of ``$S_\nu\ltimes \Gamma^\nu$''. Also $V$ carries a natural symplectic pairing $\omega$.

Now we are ready to define the category itself.
\begin{def0} \label{sympldeldef}
Consider $t,k,c_C,T_C$ as in Definition \ref{sympfindef} with $\Bbbk = \mathbb C$. Let $\nu \in \mathbb C$ be a transcendental number. The objects of the category $\Rep(H_{t,k,c}(\nu,\Gamma))$ are given by pairs $(M,y)$, where $M$ is an object of $\Rep(S_\nu\ltimes \Gamma^{\nu})$ and $y$ is a map:
$$
y: V \otimes M \to M ,
$$
such that the following holds:
$$
y\circ (1\otimes y) \circ ((1-\sigma) \otimes 1) = (\omega \otimes 1) \circ \left(t - k(\Omega^3 - \Omega^{1,3}) - \sum_{C}\frac{c_C}{1-T_C}(\Omega_C^{3}-\Omega_C^{13}-\Omega_C^{23}+\Omega_C^{123})\right)  ,
$$
as a map from $V \otimes V \otimes M $ to $M$, where $\Omega$ is an endomorphism from Definition \ref{delcentr} and $\Omega_C$ is the endomorphism obtained in a similar way as the ultraproduct of endomorphisms of the identity functor arising from the sum of elements of the group belonging to the conjugacy class $C$. 

The morphisms are given by morphisms in $\Rep(S_\nu\ltimes \Gamma^{\nu})$ which commute with $y$.
\end{def0}

In a fashion similar to the discussion after Definition \ref{ultcherdef} one can see that this definition is the same as in finite rank, written in an element free way. Thus for the same reasons one obtains the following statement, which generalizes Proposition \ref{wreathprop} and Lemma \ref{thmcherault}.
\begin{prop}
Suppose $M_n$ are $H_{t_n,k_n,c_n}(n,\Gamma)$-modules whose ultraproduct $\prodF^{C,r} M_n$ is a well defined object of ${\rm IND}({\rm Rep}(S_\nu))$. Suppose $y_n$ denotes the corresponding map \linebreak $y_n: (\FQ^2)^n \otimes M_n \to M_n$. Then ($\prod^{C,r}_{\mathcal F} M_n$, $\prodF y_n$) is an object of ${\rm Rep}(H_{t,k,c}(\nu,\Gamma))$.
\end{prop}

Also repeating the steps of Section 3.3.1 we can construct the induction functor. Since the construction is almost literally the same, we just state the result.

\begin{prop}\label{indfunctsympl}
There is a functor 
$$
{\rm Ind}_{S_\nu\ltimes \Gamma^n}^{H_{t,k,c}(\nu,\Gamma)}: {\rm Rep}(S_\nu\ltimes \Gamma^\nu) \to {\rm Rep}(H_{t,k,c}(\nu,\Gamma))  \ ,
$$
such that, if $M \in {\rm Rep}(S_\nu\ltimes \Gamma^\nu)$ is an object given by ultraproduct of $S_n\ltimes \Gamma^n$-modules, i.e., $M =\prodF^C M_n$, then:
$$
{\rm Ind}_{S_\nu\ltimes \Gamma^n}^{H_{t,k,c}(\nu,\Gamma)}(M) = \prodF^{C,r} {\rm Ind}_{S_n \ltimes \Gamma^n}^{H_{t_n,k_n,c_n}(n,\Gamma)}(M_n)  .
$$
\end{prop}

\section{The deformed double current algebra of type A}

\subsection{Construction via the Deligne category $\Rep(S_\nu)$}

\subsubsection{The construction}

In this section we will construct the DDC-algebra of type A we are after, which we will call $\mathcal D_{t,k,\nu}$. We will do this by taking endomorphisms of a certain object of the Deligne category.

First define the following object:
\begin{def0}
Define an object $H_{t,k}(\nu)\be \in \Rep(H_{t,k}(\nu))$ to be equal to $\Ind_{S_\nu}^{H_{t,k}(\nu)}(\Bbbk)$.
\end{def0}

One can easily construct $H_{t,k}(\nu)\be$ as an ultraproduct.

\begin{lemma} \label{HElemma}
The object $H_{t,k}(\nu)\be$ is isomorphic to $\prod^{C,r}_{\mathcal F}H_{{t_n},{k_n}}(\nu_n)\be$.
\end{lemma}
\begin{proof}
Indeed, we know that $H_{{t_n},{k_n}}(\nu_n)\be = \Ind_{S_{\nu_n}}^{H_{t_n,k_n}(\nu_n)}(\Fpn)$, so the conclusion follows from Corollary \ref{indcor}. 
\end{proof}

Note that we can define a filtration on $H_{t,k}(\nu)\be$ by objects of $\Rep(S_\nu)$ using the construction in Definition \ref{inddef}. Indeed,  assign $\deg(x):= 1$ and $\deg(y):=1$, i.e., we take $F^mH_{t,k}(S_\nu)\be$ to be equal to $\sum_{i=0}^m S^i(\mathfrak h^*) \otimes S^{k-i}(\mathfrak h) \otimes \mathbb C$. This agrees with the filtration by $\textbf{Rep}_{p_n}(S_{\nu_n})$-modules of $H_{{t_n},c_n}(\nu_n)\be$ given by $\deg(x_i)=\deg(y_i)=1$.\footnote{This is the filtration used in taking the restricted ultraproduct in Lemma \ref{HElemma}.}

Notice that the same assignment of degrees defines a grading of $H_{t,k}(\nu)\be$ (and respectively $H_{{t_n},c_n}(\nu_n)\be$) by $S_\nu$-modules ($S_{\nu_n}$-modules). Hence we have a corollary:
\begin{cor}
The object $F^mH_{t,k}(\nu)\be$ is isomorphic to $\prod^C_{\mathcal F}F^mH_{t_n,k_n}(\nu_n)\be$.
\end{cor}

Now we are ready to define the DDC-algebra in question.
\begin{def0}
The algebra $\mathcal D_{t,k,\nu}$ is the endomorphism algebra $\End_{\Rep(H_{t,k}(\nu))}(H_{t,k}(\nu)\be)$.
\end{def0}

Note that this is an actual algebra over $\mathbb C$, since it is given by a vector space of morphisms.

Also note that we can rewrite this as:
$$
\End_{\Rep(H_{t,k}(\nu))}(H_{t,k}(\nu)\be) = \Hom_{\IND(\Rep(S_\nu))}(\mathbb C, H_{t,k}(\nu)\be)  .
$$
So this algebra is given by the direct sum of all trivial representations of $S_\nu$ in $H_{t,k}(\nu)\be$. Via this observation we can trivially restrict the grading of $H_{t,k}(\nu)\be$ to the grading on $\mathcal D_{t,k,\nu}$.

Note that by Remark \ref{sphrem} in finite rank this construction gives us the spherical subalgebra $\mathcal B_{t_n,k_n}(\nu_n)$. The spherical subalgebras inherit the gradings in a similar fashion.

To finish this section we would like to relate these algebras.

\begin{prop} \label{DDCult}
The algebra $\mathcal D_{t,k,\nu}$ is given by the restricted ultraproduct of the spherical subalgebras $\prodF^r \mathcal B_{t_n,k_n}(\nu_n)$ with respect to the  filtrations mentioned in the discussion after Lemma \ref{HElemma}.
\end{prop}
\begin{proof}
Indeed, by the definition of the DDC-algebra we have:
$$
\mathcal D_{t,k,\nu} = \Hom_{\IND(\Rep(S_\nu))}(\mathbb C, H_{t,k}(\nu)\be) = \prodF^r\Hom_{\Repb_{p_n}(S_\nu)}(\Fpn, H_{t_n,k_n}(\nu_n)\be)  ,
$$
where the restricted ultraproduct is taken with respect to the filtrations on $H_{t_n,k_n}(\nu_n)\be$. Hence we can conclude that:
$$
\mathcal D_{t,k,\nu} = \prodF^r\mathcal B_{t_n,k_n}(\nu_n)  ,
$$
as required.
\end{proof}

\begin{rem}
These results suggest that the family of algebras $\mathcal B_{t_n,k_n}(\nu_n)$ should fall into the class covered by Example \ref{infalgex}. This is indeed the case and will be proved in Appendix A. This shows that we could have constructed $\mathcal D_{t,k,\nu}$ via the restricted ultraproduct without using Deligne categories. However, the construction via Deligne categories is more conceptual and has a number of advantages. For example, it allows one to easily define a large family of representations of $\mathcal D_{t,k,\nu}$. Indeed, if $M$ is an $H_{t,k}(\nu)$-module (see \cite{entova2014representations} for a description of some of them), then $\Hom_{\Rep(S_\nu)}(\mathbb C,M)$ has a natural structure of a $\mathcal D_{t,k,\nu}$-module. Admittedly, these modules are also  constructible as ultraproducts (as, by definition, is everything obtained from Deligne categories), but their direct construction via Deligne categories is more transparent. 
\end{rem}

We will also need the same algebra defined in $\Rep^{\rm ext}(H_{t,k}(\nu))$, and will denote it also by $\widetilde {\mathcal D}_{t,k,\nu}$ (note that in  this case $\nu$ is not a  number, but a  variable). Clearly, the analog of Proposition \ref{DDCult} also holds for this algebra.

\subsubsection{A basis of the deformed double current algebra of type A}

In this section we will construct a basis of $\mathcal D_{t,k,\nu}$. Note that in this section $t,k$ are arbitrary elements of $\Bbbk$ and $n$ is any integer.

In order to do this we will start by working with the spherical subalgebras in finite rank. One question which is worthwhile to ask is: can we introduce a basis of filtered components of these algebras which stabilizes for large $n$? Indeed, this should be possible since their restricted ultraproduct lies in $\IND(\Rep(S_\nu))$.

We will construct such a basis in the following way.
\begin{def0}
Define elements $T_{r,q,n} \in \mathcal B_{t,k}(n)$ (over $\Bbbk$) for $r,q \in \mathbb Z_{\ge 0}$, $r+q = L$ using the formula
$$
\sum_{r,q \ge 0,\ r+q=L} T_{r,q,n}\frac{u^rv^q}{r!q!} := \sum_{i=1}^n \frac{(ux_i + vy_i)^L}{L!}\be  ,
$$
where $u,v$ are formal variables. 
\end{def0}These elements are well defined if ${\rm char}(\Bbbk)=0$ or $>L$.

Next we need to define certain combinations of these elements.
\begin{def0}
Denote by  ${{\bf m}}$ a collection of non-negative integers $m_{r,q}$ for all $r,q \in \mathbb Z_+$ such that  $r+q>0$, all but finitely many of them zero. Denote $|{{\bf m}}| := \sum_{r,q\ge 0,r+q>0}m_{r,q}$ and $w({{\bf m}}) := \sum_{r,q\ge 0,r+q>0}(r+q)m_{r,q}$. Define elements $T_n({{\bf m}})\in \mathcal B_{t,k}(n)$, with $|{\bf m}| = m$, by the formula
$$
\sum_{\bold m: |\bold m|=m} T_n({{\bf m}}) \prod_{r,q\ge 0}\frac{z_{r,q}^{m_{r,q}}}{m_{r,q}!} = \frac{\left(\sum_{r,q\ge 0, r+q > 0}z_{r,q}T_{r,q,n}\right)^m}{m!}  .
$$
Here $z_{r,q}$ are once again formal variables and if we work in positive characteristic, we assume that $w(\bold m)<\chct(\Bbbk)$.
\end{def0}

We clarify these definitions by writing these elements more explicitly.  Define $a(r,q,j)$ for $1 \le j \le r+q$ to be $a(r,q,j) = x$ for $1 \le j \le r$ and $a(r,q,j) = y$ for $r+1 \le j \le r+q$. Then 
$$
T_{r,q,n} = \frac{1}{(r+q)!} \sum_{i=1}^n \sum_{\sigma \in S_{r+q}} \left(\prod_{j=1}^{r+q} a(r,q,\sigma(j))_i\right) \bold e  ,
$$
where the product in $\prod_{j=1}^{r+q}$ is taken from left to right (i.e., $a(r,q,\sigma(1))_i a(r,q,\sigma(2))_i \dots$). In other words, this element consists of sums of all possible shuffles of $r$ copies of $x_i$ and $q$ copies of $y_i$. Similarly, $T_n({{\bf m}})$ is proportional to the sum of all possible shuffles of $m_{r,q}$ copies of $T_{r,q,n}$.

Let us see what happens with these elements under the ``leading term" map: 
$$
{\rm gr}^L: F^L\mathcal B_{t,k}(n) \to F^L\mathcal B_{t,k}(n)/F^{L-1}\mathcal B_{t,k}(n)\simeq \Bbbk[x_1,\dots,x_n,y_1,\dots,y_n]^{S_n}_L  .
$$
We calculate:
$$
{\rm gr}^{r+q}(T_{r,q,n}) = \sum_{i=1}^nx_i^ry_i^q  .
$$
\begin{def0}
Denote by $P_{r,q,n}$ the symmetric polynomial $\sum_{i=1}^nx_i^ry_i^q$.
\end{def0}

So we can further conclude that:
$$
{\rm gr}^{w({{\bf m}})}(T_n({{\bf m}})) = \prod_{r,q\ge 0, r+q > 0}{P_{r,q,n}^{m_{r,q}}} \ .
$$

From this we can conclude the following.
\begin{lemma} \label{lemmasymm}
For $L \le n$ and $\chct(\Bbbk)=0$ or large compared to $n$, the vector space $F^L\mathcal B_{t,k}(n)/F^{L-1}\mathcal B_{t,k}(n)$ has a basis $\{T_n({{\bf m}}) | w({{\bf m}})=L\}$.
\end{lemma}
\begin{proof}
Indeed, from invariant theory we know that
$$
\Bbbk[x_1,\dots,x_n, y_1,\dots,y_n]^{S_n}_L= (\Bbbk[P_{r,q,n}]_{r,q\ge 0, 0 < r+q} )_L \ .
$$
 So a basis in $\Bbbk[x_1,\dots,x_n, y_1,\dots,y_n]^{S_n}_L$ is given by products $\prod_{r,q, 0 < r+q}P_{r,q,n}^{m_{r,q}}$ for $m_{r,q}$ such that $\sum_{r,q\ge 0, r+q > 0}m_{r,q}(r+q)=L$. But this is exactly the basis in question up to multiplication by non-zero constants.
\end{proof}

Now it follows that $T_n({{\bf m}})$ form a basis of the corresponding filtered component.  So we have a corollary. 
\begin{cor} \label{basisfin}
For $L \le n$ and $\chct(\Bbbk)=0$ or large compared to $n$ the vector space $F^L\mathcal B_{t,k}(n)$ has a basis given by $\{T_n({{\bf m}}) | w({{\bf m}})\le L\}$.
\end{cor}

This tells us that $F^L\mathcal B_{t,k}(n)$ indeed stabilizes as $n \to \infty$. 

Now we would like to construct similar elements in $\mathcal D_{t,k,\nu}$.
Notice that we can think about $T_n({{\bf m}})$ as a map from $\Bbbk$ to $H_{t,k}(n)\be$. The image of this map lies within the filtered component of degree $w({{\bf m}})$. Thus the ultraproduct $\prod_{\mathcal F}T_{\nu_n}({{\bf m}})$ gives us a well-defined map from $\mathbb C$ to $H_{t,k}(\nu)\be$.

\begin{rem}
From this point on $t,c,\nu \in \mathbb C$ are the same elements as in the rest of the paper.
\end{rem}
\begin{def0}
By $T({{\bf m}})$ denote the element of $\mathcal D_{t,k,\nu}$ given by $\prod_{\mathcal F}T_{\nu_n}({{\bf m}})$.
\end{def0}

\begin{rem}
We can also  write down these maps  explicitly as follows.

First we send $\mathbb C$ to $\mathfrak h \otimes \mathfrak h^* \otimes \mathbb C$ via the co-evaluation map. After that using maps ${\rm pw}^q:\mathfrak h \to \mathfrak h^{\otimes q}$ and ${\rm pw}^r:\mathfrak h^* \to \mathfrak h^{*\otimes r}$ (ultraproducts of the standard maps $x_i \to x_i \otimes \dots \otimes x_i$), we send the target object of the previous map to $\mathfrak h^{\otimes q} \otimes \mathfrak h^{*\otimes r} \otimes \mathbb C$. Then we send this object to the  ${\rm Perm}_{r,q}(\mathfrak h)\otimes \mathbb C$, where ${\rm Perm}_{r,q}(\mathfrak h)$ is given by the direct sum of all possible permutations of tensor products of $q$ copies $\mathfrak h$ and $r$ copies of $\mathfrak h^*$. At last, we  act on this object via the map, which we denote ${\rm appl}$, sending any object $Y_1 \otimes \dots \otimes Y_{r+q}\otimes \mathbb C$ (where $Y_i = \mathfrak h$ or $\mathfrak h^*$) to $H_{t,k}(\nu)\be$ using the maps $x$ and $y$ applied starting from right. To sum up, we have:
$$
T_{r,q}: \mathbb C \xrightarrow{\rm coev} \mathfrak h \otimes \mathfrak h^* \otimes \mathbb C \xrightarrow{{\rm pw}^q \otimes {\rm pw}^r} \mathfrak h^{\otimes q} \otimes \mathfrak h^{*\otimes r} \otimes \mathbb C \xrightarrow{{\rm perm}_{r,q} \otimes 1} {\rm Perm}_{r,q}(\mathfrak h)\otimes \mathbb C \xrightarrow{\rm appl} H_{t,k}(\nu)\be  .
$$
Here ${\rm perm}_{r,q}$ is the average of all the permutaions.  
It's easy to see that this is the same as the ultraproduct of $T_{r,q,\nu_n}$. One can then obtain the maps $T({{\bf m}})$ by multiplication of these maps. 
\end{rem}

Using the last result we can conclude that the maps $T({{\bf m}})$ are a basis of $\mathcal D_{t,k,\nu}$.
\begin{prop}
The elements $T({{\bf m}})$ for all choices of ${{\bf m}}$ constitute a basis of $\mathcal D_{t,k,\nu}$.
\end{prop}
\begin{proof}
Indeed, from Proposition \ref{DDCult} we know that $F^L\mathcal D_{t,k,\nu} = \prodF F^L\mathcal B_{t_n,k_n}(\nu_n)$. But {\fan}  (i.e., $\nu_n > L$) by Corollary \ref{basisfin} we know that  the basis of $F^L\mathcal B_{t_n,k_n}(\nu_n)$ is given by  $T_{\nu_n}({{\bf m}})$ with $w({{\bf m}}) \le L$. Since $\mathcal D_{t,k,\nu} = \sum_{L\ge 0}F^L\mathcal D_{t,k,\nu}$ it follows that $T({{\bf m}})$ constitute the basis of the whole algebra.
\end{proof}

In a similar fashion we have a parallel proposition:
\begin{prop}
The elements $T({{\bf m}})$ for all choices of ${{\bf m}}$ constitute a basis of the $\Cext$-algebra $\widetilde{\mathcal D}_{t,k,\nu}$.
\end{prop}

\subsubsection{Deformed double current algebra of type A with central parameter}

For our convenience we would like to make $\nu$ into a central element and consider our DDC-algebra over $\mathbb C$. In order to do this, we will need the following result:
\begin{lemma} \label{lemmaappA}
The structure constants of the basis $T({{\bf m}}) \in \widetilde{\mathcal D}_{t,k,\nu}$ depend polynomially on $\nu$.
\end{lemma}
\begin{proof}
This follows from the fact that the only way $\nu$ can appear in the product of two basis elements, is if in the corresponding finite rank basis vectors we encounter an empty sum $\sum_{i=1}^{\nu_n}$, each of which contributes a multiple of $\nu_n$. For details see Appendix A.
\end{proof}

From this we can conclude that a $\mathbb C[\nu]$-lattice $\bigoplus_{{{\bf m}}} \mathbb C[\nu]T({{\bf m}})\subset \widetilde{\mathcal D}_{t,k,\nu}$ inherits the structure of algebra from $\widetilde{\mathcal D}_{t,k,\nu}$. Now we can define the following algebra:
\begin{def0}
By $\mathcal D^{\rm ext}_{t,k}$ denote the algebra $\bigoplus_{{{\bf m}}} \mathbb C[\nu]T({{\bf m}})$, regarded as an algebra over $\mathbb C$. In this context we will denote the central element $\nu$ by $K$.
\end{def0}
We can write down a basis of $\mathcal D^{\rm ext}_{t,k}$:
\begin{prop}
The elements $T({{\bf m}})K^j$ for all tuples ${{\bf m}}$ and $j\ge 0$ constitute a $\mathbb C$-basis of $\mathcal D^{\rm ext}_{t,k}$.
\end{prop}
\begin{proof}
This is evident from the definition.
\end{proof}

Note that trivially we also have the following result:
\begin{prop} \label{propquot}
For $\nu \in \mathbb C \backslash \mathbb Z$ we have, $\mathcal D^{\rm ext}_{t,k}/(K-\nu) = \mathcal D_{t,k,\nu}$. And for $\nu \in \Cext$ we have $(\mathcal D^{\rm ext}_{t,k} \otimes_{\mathbb C} \Cext)/(K-\nu) = \widetilde{\mathcal D}_{t,k,\nu}$.
\end{prop}

\subsection{Presentation by generators and relations}

\subsubsection{The Lie algebra $\po$} \label{posect}

To give a presentation of $\mathcal D^{\rm ext}_{1,k}$ by generators and relations, we will have to start with the Lie algebra $\mathfrak{po}$ of polynomials on the symplectic plane. Later it will turn out that the DDC-algebra is a flat filtered deformation of $U(\mathfrak{po})$.

\begin{def0}
By $\mathfrak{po}$ denote the Lie algebra over $\Bbbk$ which is $\Bbbk[p,q]$ as a vector space, with the bracket defined by:
$$
[q^kp^l,q^mp^n] = (lm-nk)q^{k+m-1}p^{l+n-1}.
$$
We will denote the element $1 \in \Bbbk[p,q]$ by $K$.
\end{def0}
In other words, this Lie algebra is given by the standard Poisson bracket on $\Bbbk[p,q]$ determined by $\{p,q\}=1$.

This algebra admits the following grading:
\begin{def0}
Endow the Lie algebra $\mathfrak{po}$ with a grading given by $\deg(q^kp^l) = k+l-2$.
In this grading the bracket has degree $0$.
\end{def0}

Note that $(-\frac{q^2}{2}, pq, \frac{p^2}{2})$ constitutes an $\mathfrak{sl}_2$-triple. Hence we conclude that $\mathfrak{po}_0 \simeq \mathfrak{sl}_2$. This endows $\mathfrak{po}$ with a structure of an $\mathfrak{sl}_2$-module. It is easy to see that $\mathfrak{po}_i$ is isomorphic to the simple highest weight module $V_{i+2}$ of highest weight $i+2$.

\begin{def0}
Denote by $\mathfrak{n}$ the Lie subalgebra of $\mathfrak{po}$ given by $\mathfrak {n}=\bigoplus_{i >0} \mathfrak{po}_i$.
\end{def0}
As an $\mathfrak{sl}_2$-module we have:
$$
\mathfrak{n} = V_3 \oplus V_4 \oplus V_5 \oplus \dots  .
$$

\subsubsection{A presentation of $\po$ by generators and relations.}

To find a presentation of $\mathfrak{po}$ by generators and relations, it is enough to find the corresponding presentation of $\mathfrak n$. The rest will follow easily. This was done in \cite{van1991defining} using a computer calculation of the cohomology spaces of $\mathfrak{n}$ to obtain a minimal set of generators and relations. We will reproduce this result below. We will also present a direct proof of this result in Appendix B.

First, it's easy to find the generators:
\begin{def0}
The Lie algbera $\mathfrak{n}$ is generated by $\mathfrak{n}_{1}$. 
\end{def0}
\begin{proof}
 Indeed, this easily follows by induction from the formulas $p^kq^l = [\frac{p^{k+1}q^{l-2}}{k+1},\frac{q^3}{3}]$ for $l\ge 2$, $p^kq = [\frac{p^k}{k},\frac{pq^2}{2}]$ and $p^k = [\frac{p^{k-1}}{k-1}, p^2q]$.
\end{proof}

So it follows that the algebra $\mathfrak{n}$ is a quotient of the free Lie algebra $L(\mathfrak{n}_1)$, where $\mathfrak{n}_1 \simeq V_3$. The Lie algebra $L(\mathfrak{n}_1)$ has a grading determined by $\deg(\mathfrak{n}_1)=1$.

To describe the relations in a language of $\mathfrak{sl}_2$-modules we will first have to introduce a few definitions.

\begin{def0} \label{defiso1}
Fix an isomorpism of $\mathfrak{n}_1$ with $V_3$ with the highest weight vector specified as $c_1 = \frac{q^3}{6}$.

Consider $\Lambda^2\mathfrak{n}_1 = L(\mathfrak{n}_1)_2$. As $\mathfrak{sl}_2$-modules we have $\Lambda^2\mathfrak{n}_1 \simeq V_4 \oplus V_0$. Denote the submodule of $\Lambda^2\mathfrak{n}_1$ isomorophic to $V_0$ by $\phi_1$ and the submodule isomorphic to $V_4$ by $\phi_2$. Fix an isomorphism of $\phi_1$ with $V_0$ with the highest weight vector specified as $c_1 \wedge c_4-c_2\wedge c_3$, where $c_i = f^{i-1}c_1$. Fix an isomorphism of $\phi_2$ with $V_4$ with the highest weight vector specified as $d_1 = c_2\wedge c_1$.

Consider $\phi_2 \otimes \mathfrak{n}_1 \subset L(\mathfrak{n}_1)_3$. We have $ \phi_2 \otimes \mathfrak{n}_1 \simeq V_7 \oplus V_5 \oplus V_3 \oplus V_1$. Denote the submodule isomorphic to $V_1$ by $\psi_1$, the submodule isomorphic to $V_3$ by $\psi_2$, the submodule isomorphic to $V_5$ by $\psi_3$ and submodule isomorphic to $V_7$ by $\psi_4$. Fix an isomorphism of $\psi_1$ with $V_1$ with the highest weight vector specified as $-4d_1 \otimes c_4 + 3d_2\otimes c_3 - 2d_3 \otimes c_2 +d_4 \otimes c_1$, where $d_i = f^{i-1}d_1$. 

Consider $ \wedge^2 \phi_2 \subset L(\mathfrak{n}_1)_4$. We have $\wedge^2 \phi_2 =  V_6 \oplus V_2$. Denote the submodule isomorphic to $V_2$ by $\chi_1$. Fix an isomorphism of $\chi_1$ with $V_2$ with the highest weight vector specified as $3d_3 \wedge d_2 - 2 d_4 \wedge d_1$.
\end{def0}

We have the following proposition.
\begin{prop} \label{1proprel}
The Lie algebra $\mathfrak{n}$ is isomorphic to the quotient of the free Lie algebra $L(\mathfrak{n}_1)$ by the ideal generated by the $\mathfrak{sl}_2$-modules $\phi_1$, $\psi_4$, $\psi_1$ and $\chi_1$. This is a minimal set of relations.
\end{prop}
\begin{proof}
 As stated in the beginning of this section, one can find a proof of this result by a computer computation in \cite{van1991defining}. See Appendix B for a more direct proof.
\end{proof}

Now we can move to the description of the whole algebra. First let us introduce the notation for the remaining part of $\po$:
\begin{def0}
Denote by $\mathfrak b$ the Lie subalgebra of $\po$ given by $\mathfrak{po}_{-2}\oplus \mathfrak{po}_{-1}\oplus \mathfrak{po}_0$. We have $\po = \mathfrak{b} \oplus \mathfrak{n}$.
\end{def0}

We will also need a little more notation:
\begin{def0} \label{defiso2}
Fix an isomorphism of $\mathfrak{b}_0$ with $\mathfrak{sl}_2$ given by $e \mapsto b_1=-\frac{q^2}{2}$ and \linebreak $f \mapsto b_3 =\frac{p^2}{2}$. Fix an isomorphism of  $\mathfrak{b}_{-1}$ with $V_1$ with the highest weight vector specified as $a_1 = q$. Fix an isomorphism of $\mathfrak{b}_{-2}$ with $V_0$ with the highest weight vector specified as $K$. 

Consider the free Lie algebra $L(\mathfrak{b} \oplus \mathfrak{n}_1)$. Consider $\Lambda^2 \mathfrak{b}_{-1} \subset L(\mathfrak{b} \oplus \mathfrak{n}_1)_2$, we have \linebreak $\Lambda^2 \mathfrak{b}_{-1}\simeq V_0$. Fix an isomorphism  of $\Lambda^2 \mathfrak{b}_{-1}$ with $V_0$ with the highest weight vector specified as $a_1\wedge a_2$.

Consider $\mathfrak{n}_1 \otimes \mathfrak{b}_{-1} \subset L(\mathfrak{b} \oplus \mathfrak{n}_1)_2$. We have $\mathfrak{n}_1 \otimes \po_{-1} \simeq V_4 \oplus V_2$. Denote the submodule isomorphic to $V_2$ by $\alpha_1$ and the submodule isomorphic to $V_4$
by $\alpha_2$. Fix an isomorphism of $\alpha_1$ with $V_2$ with the highest weight vector specified as $c_2\otimes a_1-2c_1 \otimes a_2$.
\end{def0}

\begin{prop} \label{propgen2}
The Lie algebra $\po$ is generated by $\mathfrak{b} \oplus \mathfrak{n}_1$ with the following set of relations:
\begin{gather} \mathfrak{b}_{-2} \simeq V_0 \text{ is central},\ 
    \mathfrak{b}_0 \simeq \mathfrak{sl}_2,\ \mathfrak{b}_{-1} \simeq V_1 \text{ as an $\mathfrak{sl}_2$-module}  ,\ \Lambda^2 \mathfrak{b}_{-1} = \mathfrak{b}_{-2}, \nonumber\\  \ \mathfrak{n}_1 \simeq V_3 \text{ as an $\mathfrak{sl}_2$-module} ,\  \alpha_2 = 0  , \ \alpha_1 = \mathfrak{b}_0  , \nonumber \\ \phi_1 = 0  ,\ \psi_4 = 0  ,\ \psi_1 = 0,\ \chi_1 = 0 \nonumber  ,
\end{gather}
where we use the isomorphisms from Definitions \ref{defiso1} and \ref{defiso2}. And by $\lambda X \simeq \mu Y$ for two $\mathfrak{sl}_2$-submodules of $L(\mathfrak{b} \oplus \mathfrak{n}_1)$  with two fixed isomorphisms with $V_j$ and two numbers $\lambda,\mu$ we mean that we take the quotient by the image of the map 
$$
V_j \xrightarrow{(\lambda,-\mu)}V_j \oplus V_j \simeq X\oplus Y \subset L(\mathfrak{b} \oplus \mathfrak{n}_1)  . 
$$
\end{prop}
\begin{proof}
 This easily follows from Proposition \ref{1proprel}. Indeed, the first line of relations ensures that the subalgebra generated by $\mathfrak{b}$ is indeed $\mathfrak{b}$, the third line ensures that the subalgebra generated by $\mathfrak{n}_1$ is isomorphic to $\mathfrak{n}$. The second line fixes the adjoint action of $\mathfrak{b}$ on $\mathfrak{n}_1$ making sure that nothing more is generated. 
\end{proof}

One can also give a more explicit presentation, without using the language of $\mathfrak{sl}_2$-modules.

\begin{prop} \label{propgen3}
The Lie algebra $\mathfrak{po}$ is generated by elements $K$ of degree $-2$,  $q=a_1$  and $p=a_2$ of degree $-1$, $e := b_1= -\frac{q^2}{2}$ and $f :=b_3 = \frac{p^2}{2}$ of degree $0$, and $r := c_1= \frac{q^3}{6}$ of degree $1$, with defining relations:
\begin{gather}\label{rell}
     [K, X] = 0 \text{ for any $X$},\ [p,q] = K,\ [f,q] = p,\ [p,f] = 0,\ [e,p] = q  , \nonumber\\ 
       [[f,e],f]=2f, \nonumber\\
   [r,p] = e,\ [e,r] = 0,\ {\rm ad}_{f}^4(r) = 0,\ [e, [f,r]] = 3r, \nonumber  \\
     [r,{\rm ad}_{f}^3(r)] - [{\rm ad}_{f}(r),{\rm ad}_{f}^2(r)] = 0,\\ {\rm ad}_{r}^3(f)= 0,\ \nonumber \\
    4[{\rm ad}^3_{f}(r), {\rm ad}^2_{r}(f)] - 3[{\rm ad}^2_{f}(r), {\rm ad}_{f}{\rm ad}^2_{r}(f)] + 2[{\rm ad}_{f}(r), {\rm ad}^2_{f}{\rm ad}^2_{r}(f)] - [r, {\rm ad}^3_{f}{\rm ad}^2_{r}(f)] = 0  , \nonumber \\
    3[{\rm ad}_{f}^2{\rm ad}_{r}^2(f), {\rm ad}_{f}{\rm ad}_{r}^2(f)] - 2[{\rm ad}_{f}^3{\rm ad}_{r}^2(f), {\rm ad}_{r}^2(f)]=0.\nonumber
\end{gather}
\end{prop}
\begin{proof}
 In order to get this presentation from the one given in Proposition \ref{propgen2}, to start with, we need to throw out some of the generators. Indeed, in the formulation we threw out the generator corresponding to $h$ in the $\mathfrak{sl}_2$-triple of $\mathfrak{b}_0$ and we have only taken one generator from the whole of $\mathfrak{n}_1$ -- the highest-weight vector $r$. This is obviously enough, since we can generate the whole of $\mathfrak{sl}_2$ using $e$ and $f$, and then generate the rest of $\mathfrak{n}_1$ by the action of $\mathfrak{b}_0$ on $r$.
 
 Now, it's easy to see that the first line of the relations in Proposition \ref{propgen2} transforms into the first two lines of relations \eqref{rell} and the second line of the relations in Proposition \ref{propgen2} transforms into the third line of the relations \eqref{rell}. We only need to keep the highest-weight vectors of the third line of the relations in Proposition \ref{propgen2}, since the rest of the relations can be generated by the action of $\mathfrak{b}_0$. These four highest-weight vectors are given in the last lines of relations \eqref{rell} in the same order as the corresponding $\mathfrak{sl}_2$-modules in Proposition $\ref{propgen2}$.
 
 For the details of these calculations see Appendix B.
 \end{proof}
 
\begin{rem} \label{altgenpo}
Using this we can also write down a presentation of $\po$ with just three generators. Indeed, the Lie algebra $\mathfrak{po}$ is generated by elements $p$, $f$ and $r$ of degrees $-1,0,1$ respectively, with defining relations:
\begin{gather}\label{rell1}
     [{\rm ad}_p^3(r), X] = 0 \text{ for any $X$},\nonumber\\ [[[p,r],f],p]= p,\ [p,f] = 0  ,\ (\text{degree -1}) \nonumber\\ 
     [[[p,r],f],f]=2f,\ (\text{degree 0})\nonumber\\
    {\rm ad}_r^2(p) = 0  ,\ {\rm ad}_{f}^4(r) = 0  ,\ [[[p,r],f], r] = 3r,\ (\text{degree 1})\nonumber  \\
     [r,{\rm ad}_{f}^3(r)] - [{\rm ad}_{f}(r),{\rm ad}_{f}^2(r)] = 0,  \ (\text{degree 2})\nonumber \\
     {\rm ad}_{r}^3(f)= 0  ,\ (\text{degree 3}) \nonumber \\
    4[{\rm ad}^3_{f}(r), {\rm ad}^2_{r}(f)] - 3[{\rm ad}^2_{f}(r), {\rm ad}_{f}{\rm ad}^2_{r}(f)] + 2[{\rm ad}_{f}(r), {\rm ad}^2_{f}{\rm ad}^2_{r}(f)] - [r, {\rm ad}^3_{f}{\rm ad}^2_{r}(f)] = 0, \nonumber \\
    3[{\rm ad}_{f}^2{\rm ad}_{r}^2(f), {\rm ad}_{f}{\rm ad}_{r}^2(f)] - 2[{\rm ad}_{f}^3{\rm ad}_{r}^2(f), {\rm ad}_{r}^2(f)]=0\  (\text{degree 4}) .\nonumber
\end{gather}
\end{rem}

\subsubsection{Flat filtered deformations of $U(\po)$}

In the beginning of Section \ref{posect} we've mentioned that $\mathcal D^{\rm ext}_{1,k}$ is going to be isomorphic to a flat filtered deformation of $U(\mathfrak{po})$. For this reason in this section we will formulate a result on flat filtered deformations of $U(\mathfrak{po})$ obtained via computer calculations and then present a known flat filtered deformation of $U(\po)$.

Using computer calculation one can arrive at the following proposition about the deformations of $U(\po)$. Again, before we can formulate the relations in terms of $\mathfrak{sl}_2$-modules we need to introduce some notations:

\begin{def0}
Consider a free associative algebra $T(\mathfrak{b}\oplus \mathfrak{n}_1)$. Denote the subspace \linebreak $S^2\mathfrak{b}_{-1} \subset T(\mathfrak{b}\oplus\mathfrak{n}_1)_2$ isomorphic to $V_2$ as $\mathfrak{sl}_2$-module by $\beta_1$. Fix an isomorphism of $S^2\mathfrak{b}_{-1}$ with $V_2$ with the highest weight vector specified by $a_1^2$.

Also for any $\mathfrak{sl}_2$-submodule $\gamma \subset T(\mathfrak{b}\oplus\mathfrak{n}_1)$, denote by $K^i\gamma$ the submodule $\gamma \otimes \mathfrak{b}_{-2}^{\otimes i}$. If $\gamma$ had a fixed isomorphism with $V_j$ with the highest weight vector specified by $v_\gamma$, fix an isomorphism of $\gamma \otimes \mathfrak{b}_{-2}^{\otimes i}$ with $V_j$ with the highest weight vector specified by $v_\gamma \otimes K^{\otimes i}$.
\end{def0}

We are ready to state the main result of the section.

\begin{prop} \label{propdef1}
Suppose $U$ is a flat filtered deformation of $U(\mathfrak{po})$ as an associative algebra (up to an automorphism), such that $U(\mathfrak{b})$ is still a subalgebra of  $U$, and the action of $U(\mathfrak{b})$ on $\mathfrak{b}\oplus\mathfrak{n}_1$ is not deformed. Then $U$ is isomorphic to $A_{s_1,s_2}$ defined below for some values of $s_1$ and $s_2$. The algebra $A_{s_1,s_2}$ is generated by $\mathfrak{b}\oplus \mathfrak{n}_1$ with the set of relations given by the first two lines of Proposition \ref{propgen2} and the following relations, which substitute the last line in Proposition \ref{propgen2}:
\begin{gather}
\phi_1 = -\frac{s_1K}{2}  ,\ \psi_4 = 0  ,\ \psi_1 \simeq 15s_1\mathfrak{b}_{-1}  ,\ \chi_1 \simeq 3((30s_1 + 14s_2K)\mathfrak{b}_0 + 7s_2\beta_1)  , \nonumber  \end{gather}
where $s_1,s_2 \in \mathbb C[K]$, ''$\simeq$" means the same thing as in Proposition \ref{propgen2}, and all the submodules of $L(\mathfrak{b}\oplus \mathfrak{n}_{-1})$ are interpreted as submodules of $T(\mathfrak{b}\oplus \mathfrak{n}_{-1})$ via the map \linebreak $L(\mathfrak{b}\oplus \mathfrak{n}_{-1}) \to T(\mathfrak{b}\oplus \mathfrak{n}_{-1})$ which sends the elements of the free Lie algebra into the corresponding commutators in the free associative algebra. 
\end{prop}
\begin{proof}
First of all note that our requirement on the type of deformation effectively means that we consider such deformations of relations in Proposition \ref{propgen2} which change only the last four relations, augmenting them by some lower order terms. 

The outline of the computer calculation used is as follows.

Given a family of putative flat filtered deformations of a finitely graded algebra, the subscheme over which it is flat is cut out by the condition that for any linear combination of the deformed relations, the leading degree term is in the undeformed ideal. Just as in the commutative setting, there is a notion of Gr\"obner bases for noncommutative algebras, and one could in principle check flatness by computing the Gr\"obner bases of both the original and the deformed ideal and verifying that the leading terms agree. Unfortunately (since basic questions about noncommutative algebras are undecidable), the Gr\"obner basis is in general infinite, so the algorithm that produces such a basis will not terminate. However, we can still produce a {\em subset} of the equations satisfied on the flat locus via this approach, by simply stopping the calculation at some arbitrary point. In the case of interest, we do this by computing all S-polynomials of pairs of the deformed relations (noting that in the noncommutative case two relations may have more than one S-polynomial) and reducing them modulo the deformed relations. This gives us out a new collection of relations, and any such relation that vanishes in $U({\po})$ must vanish on the flat deformation, so gives an equation for each of its coefficients. After using these equations to eliminate parameters, we find that some of the relations become independent of the parameters, and thus we may reduce mod those relations. The resulting set contains 12 relations of degree 15 that span a 10-dimensional space of relations on $U({\po})$, and thus gives two new relations vanishing on $U({\po})$, allowing us to eliminate all but two parameters as required.
\end{proof}

\begin{rem} Note that we can specialize the central element $K$ to a number, which will give a $3$-parameter flat family of algebras $A_{s_1,s_2,K}$, with 
$s_1,s_2,K\in \Bbb C$. These parameters have degrees $4,6,-2$, respectively; alternatively, we may view this deformation as one with four deformation parameters 
$s_1,s_2,s_1'=s_1K,s_2'=s_2K$ of degrees $4,6,2,4$, respectively, which are constrained by the relation $s_1s_2'=s_2s_1'$; i.e., deformations are parametrized by a quadratic cone in $\Bbb C^4$. Also, we see that up to rescaling there are only two essential parameters, $s_1^*=s_1K^2$ and $s_2^*=s_2K^3$.  
\end{rem} 

As before, this presentation can be formulated more explicitly as follows:
\begin{prop} \label{propdef2}
The algebra $A_{s_1,s_2}$ is generated by the same generators as $\mathfrak{po}$ and the same set of relations as in Proposition \ref{propgen3}, with the last four relations deformed as follows:
\begin{gather}
[r,{\rm ad}_{f}^3(r)] - [{\rm ad}_{f}(r),{\rm ad}_{f}^2(r)] = -\frac{s_1K}{2}  ,\nonumber \\ {\rm ad}^3_{r}(f)=0  ,  \\
    4[{\rm ad}^3_{f}(r), {\rm ad}^2_{r}(f)] - 3[{\rm ad}^2_{f}(r), {\rm ad}_{f}{\rm ad}^2_{r}(f)] + 2[{\rm ad}_{f}(r), {\rm ad}^2_{f}{\rm ad}^2_{r}(f)] 
    - [r, {\rm ad}^3_{f}{\rm ad}^2_{r}(f)] = 15s_1q  , \nonumber\\
    3[{\rm ad}_{f}^2{\rm ad}_{r}^2(f), {\rm ad}_{f}{\rm ad}_{r}^2(f)] - 2[{\rm ad}_{f}^3{\rm ad}_{r}^2(f), {\rm ad}_{r}^2(f)]= 3((30s_1 + 14s_2K)e + 7s_2q^2)   ,\nonumber
\end{gather}
where $s_1,s_2 \in \mathbb C[K]$.
\end{prop}
\begin{proof}
This is easy to see following the proof of Proposition \ref{propgen3}.
\end{proof}

\begin{rem}
We can also rewrite the above relations (Proposition \ref{propdef2}) using the set of generators of Remark \ref{altgenpo}. Indeed, the algebra $A_{s_1,s_2}$ is generated by the same set of generators as $\mathfrak{po}$ in Remark \ref{altgenpo} (i.e., $p, f,r$) and the same set of relations as in Remark \ref{altgenpo}, with the last four (degrees $2,3,4$) deformed as follows:
\begin{gather}
[r,{\rm ad}_{f}^3(r)] - [{\rm ad}_{f}(r),{\rm ad}_{f}^2(r)] = -\frac{s_1K}{2}  , \\ {\rm ad}^3_{r}(f)=0  , \nonumber \\
    4[{\rm ad}^3_{f}(r), {\rm ad}^2_{r}(f)] - 3[{\rm ad}^2_{f}(r), {\rm ad}_{f}{\rm ad}^2_{r}(f)] + 2[{\rm ad}_{f}(r), {\rm ad}^2_{f}{\rm ad}^2_{r}(f)] 
    - [r, {\rm ad}^3_{f}{\rm ad}^2_{r}(f)] = 15s_1{\rm ad}^2_{p}(r)  , \nonumber\\
    3[{\rm ad}_{f}^2{\rm ad}_{r}^2(f), {\rm ad}_{f}{\rm ad}_{r}^2(f)] - 2[{\rm ad}_{f}^3{\rm ad}_{r}^2(f), {\rm ad}_{r}^2(f)]
    =3( 7s_2{\rm ad}^2_{p}(r)^2-(30s_1 + 14s_2K){\rm ad}_{p}(r))   ,\nonumber
\end{gather}
where $K={\rm ad}^3_p(r)$ and 
$s_1,s_2 \in \mathbb C[K]$.
\end{rem}

Below we will show that the universal enveloping algebra of the Lie algebra $\mathbb C[x,\partial]$ gives us an example of such a deformation. This result is well-known, see \cite{feigin1980homology}.
\begin{def0}
Denote by $\mathbb C[x,\partial]$ the Lie algebra of polynomial differential operators, with a Lie bracket given by the commutator.
\end{def0}
Consider a grading on $\mathbb C[x,\partial]$ given by $\deg(x^k\partial^l) =k+l-2$. We have a decomposition $\mathbb C[x,\partial] = \bigoplus_{i=-2} \mathbb C[x,\partial]_i$. It's easy to see that with this grading the Lie bracket decreases filtration degree at least by $2$ and preserves degree modulo $2$:
$$
[,]: \mathbb C[x,\partial]_i \otimes \mathbb C[x,\partial]_j \to \mathbb C[x,\partial]_{i+j} \oplus \mathbb C[x,\partial]_{i+j-2} \oplus \dots. 
$$
Indeed, when we compute the commutator we use the identity $[\partial,x]=1$ at least once, and each time it decreases the grading by $2$. 

\begin{lemma} \label{lemmaassdif}
The associated graded Lie algebra of $\mathbb C[x,\partial]$ is isomorphic to $\po$.
\end{lemma}
\begin{proof}
Writing down the commutator of basis elements,  we have:
$$
[x^k\partial^l, x^m \partial^n] = (lm-nk)x^{k+m-1}\partial^{l+n-1} + \dots  .
$$
So by taking the associated graded of $\mathbb C[x,\partial]$  and denoting the image of $x$ by $q$ and the image of $\partial$ by $p$, we end up with $\mathfrak{po}$.
\end{proof}

And we have the following corollary:
\begin{cor}
$\mathbb C[x,\partial]$ is a non-trivial flat filtered deformation of $\po$ as a Lie algebra.
\end{cor}
\begin{proof}
The flatness follows from Lemma \ref{lemmaassdif} and the fact that the graded dimensions of the two Lie algebras are the same. 

The fact that this deformation is non-trivial (which is not hard to check directly) is known as the van Hove-Groenewold's theorem in quantum mechanics, which says that classical infinitesimal symmetries deform nontrivially under quantization. See Theorem 13.13 in \cite{hall2013quantum}.
\end{proof}

Now from Proposition \ref{propdef1} it follows that $U(\mathbb C[x,\partial])$ must be isomorphic to $A_{s_1,s_2}$ for some choice of $s_1$ and $s_2$. Let us now compute these parameters.
\begin{prop} \label{propdiff}
The algebra $U(\mathbb C[x,\partial])$ is isomorphic to $A_{1,0}$.
\end{prop}
\begin{proof}
From Proposition \ref{propdef1} we know that $U(\mathbb C[x,\partial]) \simeq A_{s_1,s_2}$. Since this deformation actually comes from the Lie algebra deformation, we can conclude that $s_2$ must be equal to zero. Now we can consider the Lie algebra $\mathfrak{a}_{s_1}$ given by the generators and relations of Proposition \ref{propdef2} with $s_2=0$. So we know that $\mathbb C[x,\partial] \simeq \mathfrak{a}_{s_1}$. Let's denote this isomorphism by $\varepsilon: \mathfrak{a}_{s_1} \to \mathbb C[x,\partial]$. Since $\varepsilon$ is determined up to a constant, we can set the image of $K$ under  $\varepsilon$ to be $\varepsilon(K) = 1$. Now since $\mathfrak a_{s_1}$ is a deformation of $\gr(\mathbb C[x,\partial])$, we know that $\varepsilon(q) = x + \dots$, $\varepsilon(p) = \partial + \dots$, $\varepsilon(e) = -\frac{x^2}{2}+ \dots$, $\varepsilon(f) = \frac{\partial^2}{2}+ \dots$ and $\varepsilon(r) = \frac{x^3}{6} + \dots$, where ``$\dots$" stand for the lower order terms. Also note that since the commutator is deformed in degrees starting with $-2$, it follows that the lower order terms also can appear only starting with degrees $-2$. Hence $\varepsilon(q) = x$ and $\varepsilon(p) = \partial$. Suppose $\varepsilon(e) = -\frac{x^2}{2}+ c_1$ and $\varepsilon(f) = \frac{\partial^2}{2}+ c_2$, it follows that  $[\varepsilon(e),\varepsilon(f)] = x\partial + \frac{1}{2}$. Now by calculating $[[\varepsilon(e),\varepsilon(f)],\varepsilon(e)] = [x\partial, -\frac{x^2}{2}]  =  -x^2$, we conclude that $c_1$ must be equal to $0$. The same holds true for $c_2$. Now suppose $\varepsilon(r) = \frac{x^3}{6} + d_1x+ d_2\partial$. Now $[\varepsilon(q),\varepsilon(r)] = -d_2$, hence $d_2 = 0$. And $[\varepsilon(p),\varepsilon(r)] = \frac{x^2}{2} + d_1$, hence $d_1 = 0$. So we know the images of the commutators. Now it's enough to calculate one of the relations. 

We compute ${\rm ad}_{f}(r) = [\frac{\partial^2}{2}, \frac{x^3}{6}] = \frac{x^2\partial + x}{2}$, ${\rm ad}_{f}^2(r) = [\frac{\partial^2}{2},\frac{x^2\partial+x}{2}] = x\partial^2 + \partial$ and ${\rm ad}_{f}^3(r) = [\frac{\partial^2}{2},x\partial^2] = \partial^3$. So it follows that:
$$
[\frac{x^3}{6},\partial^3] - [\frac{x^2\partial + x}{2},x\partial^2 + \partial] = -\frac{3}{2}x^2\partial^2 -3x\partial - 1 + \frac{3}{2}x^2\partial^2 + 3x\partial + \frac{1}{2} = -\frac{1}{2}  .
$$
Thus we conclude that $s_1 = 1$.
\end{proof}

We also have a corollary:
\begin{cor}
The deformation $A_{1,0}$ is flat.
\end{cor}

\begin{rem}
Of course we could have proved that $\mathbb C[x,\partial]$ is isomorphic to $\mathfrak{a}_{1}$ without using computer computation and Proposition \ref{propdef1}. Indeed, one just needs to check that $1,\  x, \ \partial, \ -\frac{x^2}{2}, \ \frac{\partial^2}{2}$ and $\frac{x^3}{6}$ satisfy the required relations, which is easy to do.
\end{rem}

\subsubsection{The deformed double current algebra of type A as a flat filtered deformation of $U(\po)$}

Here we would like to show that the generic choice of parameters $s_1$ and $s_2$ can give us the algebra $\mathcal D^{\rm ext}_{1,k}$. 

Below we will need to compute things in $\mathcal D^{\rm ext}_{1,k}$. Since this algebra is defined as a certain lattice in the ultraproduct, we need to understand how we can do this. The following definition provides us with a method.
\begin{def0} \label{defshort}
Suppose $ Y \in \mathcal D^{\rm ext}_{1,k}$ is given by $Y = f(\lbrace{T({{\bf m}})\rbrace})K^i$, where $f$ is a non-commutative polynomial with coefficients in $\mathbb C$. By construction we know that $Y|_{K=\nu} = \prodF f_n(\lbrace{T_{\nu_n}({{\bf m}})\rbrace})\nu_n^i$, where $f_n$ are non-commutative polynomials with coefficients in $\FQ$, such that $\prodF f_n = f$. As a shorthand notation we will write $Y \backsim f_n(\lbrace T_{\nu_n}({{\bf m}})\rbrace\nu_n^i)$, where we will consider the r.h.s. for large enough $n$.
\end{def0}

With this tool we are ready to continue:

\begin{prop}\label{DDCflatdef}
The algebra $\mathcal D^{\rm ext}_{1,k}$ is a flat filtered deformation of $U(\mathfrak{po})$.
\end{prop}
\begin{proof}
Indeed,  we know that the basis in this algebra is given by $T({{\bf m}})K^i$. Also recall the natural filtration we considered in the previous section (so that $T({{\bf m}})K^i \in (\mathcal D^{\rm ext}_{1,k})_{w({{\bf m}})}$). Since by Lemma \ref{lemmasymm} we know that $\gr\mathcal B_{1,k_n}(\nu_n) = \FQ[P_{r,q,\nu_n}]_{r,q\ge 0, 0 < r+q}$ in sufficiently low degrees, where the associated graded is taken with respect to the filtration discussed after Lemma \ref{HElemma}, it follows that
$$
\gr\mathcal D^{\rm ext}_{1,k} = (\prodF \FQ[P_{r,q,\nu_n}]_{r,q\ge 0, 0 < r+q  \le \nu_n})|_{\nu = K} = \mathbb C[P_{r,q}]_{r,q\ge 0}  ,
$$
where $P_{r,q} = \gr^{r+q}(T_{r,q})$ and $P_{0,0} = \gr^0(K)$.

Now the bracket $[,]$ acts as follows:
$$
[,]: (\mathcal D^{\rm ext}_{1,k})_n \otimes (\mathcal D^{\rm ext}_{1,k})_m \to (\mathcal D^{\rm ext}_{1,k})_{m+n-2} \oplus (\mathcal D^{\rm ext}_{1,k})_{m+n-4} \oplus \dots  ,
$$
where we consider the grading of the algebra as a vector space. Indeed,  this follows from the fact that $[T({{\bf m}}),T({{\bf n}})] \backsim [T_{\nu_n}({{\bf m}}),T_{\nu_n}({{\bf n}})]$, and to calculate the latter expression we need to use the commutator $[x_i,y_j]$ at least once, which, each time we use it, lowers the degree by $2$. We would like to calculate the leading term of the commutator. To calculate ${\rm gr}^{w({{\bf m}})+w({{\bf n}})-2}([T({{\bf m}}),T({{\bf n}})])$  it is enough to compute it via $\backsim$,  commuting freely elements within $T_{\nu_n}({{\bf m}})$ and leaving only the highest term in the commutator of $[x_i,y_j] = \delta_{ij} + \dots$. So:
$$
{\rm gr}^{w({{\bf m}})+w({{\bf n}})-2}([T({{\bf m}}),T({{\bf n}})]) \backsim \left[\prod_{r,q\ge0, r+q > 0}P_{r,q,\nu_n}^{m_{r,q}}, \prod_{r,q\ge0, r+q > 0}P_{r,q,\nu_n}^{n_{r,q}}\right] =
$$
$$
=\prod_{r,q\ge0, r+q > 0} P_{r,q,\nu_n}^{m_{r,q}+n_{r,q}}\sum_{r_1,r_2,q_1,q_2} \frac{m_{r_1,q_1}n_{r_2,q_2}}{P_{r_1,q_1,\nu_n}P_{r_2,q_2,\nu_n}}[P_{r_1,q_1,\nu_n},P_{r_2,q_2,\nu_n}]  .
$$

But now:
$$
[P_{r_1,q_1,\nu_n},P_{r_2,q_2,\nu_n}] =  \sum_{i,j=1}^{\nu_n} [x_i^{r_1}y_i^{q_1},x_j^{r_2}y_j^{q_2}] =
$$
$$
=(q_1r_2 - q_2r_1) P_{r_1+r_2-1,q_1+q_2-1,\nu_n}  ,
$$
where we use $P_{0,0,\nu_n}$ to denote $\nu_n$.

These formulas show us that $\gr\mathcal D^{\rm ext}_{1,k}$ is isomorphic to a deformation of $U(\mathfrak{po})$ after identification of $T_{i,j}$  with $p^iq^j$. So it follows that $\mathcal D^{\rm ext}_{1,k}$ is a deformation of $U(\mathfrak{po})$. Moreover it is a flat filtered deformation, by virtue of the fact that $T({{\bf m}})K^i$ constitute a basis of $\mathcal D^{\rm ext}_{1,k}$.
\end{proof}

Since we know all possible flat filtered deformations of $U(\po)$, it follows that $\mathcal D^{\rm ext}_{1,k}$ is isomorphic to $A_{s_1,s_2}$ for some choice of constants. We would also like to calculate the exact correspondence. 
\begin{prop}\label{defpropDDC}
The DDC-algebra $\mathcal D^{\rm ext}_{1,k}$ is isomorphic to $A_{s_1,s_2}$ with  
$$
s_1=1 + k(k+1)(1-K) \ \text{and} \ s_2 =k(k+1). 
$$
\end{prop}
\begin{proof}
We know that $\mathcal D^{\rm ext}_{1,k} \simeq A_{s_1,s_2}$ for some $s_1,s_2 \in \mathbb C[K]$. Denote this isomorphism by $\beta:A_{s_1,s_2} \to \mathcal D^{\rm ext}_{1,k}$. It is enough to calculate $s_1,s_2$ via evaluating one of the commutators. We will largely follow the steps of the proof of Proposition \ref{propdiff}. 

First, we will define another shorthand notation for our calculations:
\begin{def0}
Recall the notation of Definition \ref{defshort}. Suppose $Y = \beta(X)$ for some element $X \in A_{s_1,s_2}$. Consider the faithful polynomial representation: 
$$
\pi_n:\mathcal B_{1,k_n}(\nu_n) \to \FQ(x_1,\dots,x_{\nu_n})[\partial_1,\dots,\partial_{\nu_n}]^{S_{\nu_n}}.
$$
 Oftentimes it will be easier for us to work with $f_n(\lbrace T_{\nu_n}({{\bf m}})\rbrace )$ as $\pi_n^{-1}(f_n(\lbrace\pi_n( T_{\nu_n}({{\bf m}}))\rbrace ))$. In this case we will use another shorthand notation $X \backsim_\beta f_n(\lbrace\pi_n( T_{\nu_n}({{\bf m}}))\rbrace )\nu_n^i$, where we will consider the r.h.s for a large enough $n$.
\end{def0}

Note that the highest orders of generators are as follows. We have\footnote{Here we slightly abuse the notation and denoting by $p,q$ both the elements $q,p$ in $\po$ and the generators of $A_{s_1,s_2}$. Since these elements lie in different spaces this shouldn't cause any confusion.} ${\rm gr}^0(K) = K$, ${\rm gr}^1(T_{1,0}) = {q}$, ${\rm gr}^1(T_{0,1}) = {p}$, ${\rm gr}^2(T_{2,0}) = {q}^2$, ${\rm gr}^2(T_{1,1}) =  {pq}$, ${\rm gr}^2(T_{0,2}) =  {p}^2$ and ${\rm gr}^3(T_{3,0}) =  {q}^3$. Thus it follows that $\beta(K)=K + \dots$, $\beta(q) = T_{1,0} + \dots$, $\beta(p) = T_{0,1} + \dots$, $\beta(e) = -\frac{T_{2,0}}{2} + \dots$, $\beta(f)= \frac{T_{0,2}}{2} + \dots$ and $\beta(r) = \frac{T_{3,0}}{6} + \dots$, where ``$\dots$" stand for lower order terms. Note that since commutator has additional terms only 2 degrees lower, it follows that additional terms in $\beta$ also can only be $2k$ degrees lower for a positive integer $k$.

Thus it follows that there are no additional terms in the action of $\beta$ on $K,q$ and $p$. Suppose $\beta(e) = - \frac{T_{2,0}}{2} + \gamma_1$ and $\beta(f) = \frac{T_{0,2}}{2} + \gamma_2 $ for some $\gamma_i \in \mathbb C[K]$. First let's calculate $[\beta(e),\beta(f)]$:
$$
[e,f] \backsim_{\beta} \frac{1}{4}[\sum_i D_i^2, \sum_j x_j^2] = \frac{1}{4}[\sum_i \partial_i^2 - c(c+1)\sum_{i\ne k}\frac{1}{(x_i-x_k)^2}, \sum_j x_j^2] =
$$
$$
=\frac{1}{4}\sum_{i,j}[ \partial_i^2,x_j^2] = \frac{1}{4}\sum_i (4x_i\partial_i + 2) \backsim_\beta \beta^{-1}(T_{1,1})  ,
$$
so we conclude that $[\beta(e),\beta(f)] = T_{1,1}$. Now we want $[[\beta(e),\beta(f)],\beta(e)] = -\frac{1}{2}[T_{1,1},T_{2,0}]$ to be equal to $2\beta(e)$. We calculate:
$$
[[e,f],e]\backsim_{\beta} \frac{1}{2}[\sum_i x_i\partial_i, \sum_{j}x_j^2] \backsim_{\beta} \beta^{-1}\left(\frac{1}{2}T_{2,0}\right).
$$
Thus we conclude $\gamma_1 = 0$. A similar calculation results in $\gamma_2 = 0$. 

Now we can write $\beta(r) = \frac{T_{3,0}}{6} + \delta_1T_{1,0} + \delta_2 T_{0,1}$ for $\delta_i \in \mathbb C[K]$ (we only need to add elements of the lower degrees which have the same parity). Let's calculate $[\beta(r),\beta(q)]$ and $[\beta(r),\beta(p)]$. To do this, we need to calculate $[T_{3,0},T_{1,0}]$ and $[T_{0,3},T_{0,1}]$. The first one is obviously zero. So we have $[\beta(r),\beta(a_1)] = \delta_2[T_{0,1},T_{1,0}] = \delta_2K$. But this commutator should be zero. Hence $\delta_2 = 0$. Now for the other one:
$$
[T_{0,1,n},T_{3,0,n}] = \sum_{i,j}[\partial_i,x_j^3] = 3T_{2,0,n}  ,
$$

Thus $[\beta(r),\beta(p)] = [\frac{T_{3,0}}{6} + \delta_1T_{1,0}, T_{0,1}] = \frac{T_{2,0}}{2} - \delta_1K = -\beta(e) -\delta_1K$. Hence $\delta_1 = 0$. Thus we have successfully calculated the images of all the generators. 

Now we need to calculate the image of $ 3[{\rm ad}_{f}^2{\rm ad}_{r}^2(f), {\rm ad}_{f}{\rm ad}_{r}^2(f)] - 2[{\rm ad}_{f}^3{\rm ad}_{r}^2(f), {\rm ad}_{r}^2(f)]$. Indeed,  this is the only relation where both $s_1$ and $s_2$ are present. We calculate:
$$
{\rm ad}_{r}(f) = -[f,r] \backsim_{\beta} -\frac{1}{12}[\sum_i \partial_i^2, \sum_jx_j^3]= -\frac{1}{2}\sum_i(x_i^2\partial + x_i)  ,
$$
and 
$$
{\rm ad}_{r}^2(f) \backsim_{\beta} \frac{1}{12}[\sum_i x_i^2\partial_i, \sum_j x_j^3] = \frac{1}{4}\sum_i x_i^4  .
$$

Similarly we can compute the results of the action of powers of ${\rm ad}_{f}$. The differential operator part is quite straightforward, but we will write down the part depending on $c$ in more detail. Denoting $X = {\rm ad}_{r}^2(f)$ and $\kappa = k(k+1)$, we have:
$$
{\rm ad}_{f}(X) \backsim_{\beta} \frac{1}{8}[\sum_i \partial_i^2, \sum_j x_j^4] = \sum_{i}\left(x_i^3\partial_i + \frac{3}{2}x_i^2\right),
$$
$$
{\rm ad}^{2}_{f}(X) \backsim_{\beta} \frac{1}{4}[\sum_i \partial_i^2, \sum_j(2x_j^3 \partial_j + 3x_j^2)] -\frac{\kappa}{2}[\sum_{i\ne j}\frac{1}{(x_i-x_j)^2}, \sum_m x_m^3\partial_m] =
$$
$$
=\sum_i\left(3x_i^2\partial_i^2 + 6x_i\partial_i +\frac{3}{2}\right) -\frac{\kappa}{2}\left(\sum_{i\ne j}\frac{2x_i^3-2x_j^3}{(x_i-x_j)^3}\right).
$$
Now, transforming the last sum, we have:
$$
\sum_{i\ne j}\frac{2x_i^3-2x_j^3}{(x_i-x_j)^3} = \sum_{i\ne j}\frac{2x_ix_j -x_i^2-x_j^2 + 3x_i^2 +3x_j^2}{(x_i-x_j)^2} = 3\sum_{i\ne j}\frac{x_i^2 + x_j^2}{(x_i-x_j)^2} - n(n-1)  .
$$
So in total we have:
$$
{\rm ad}^{2}_{f}(X) \backsim_{\beta} \sum_i(3x_i^2\partial_i^2 + 6x_i\partial_i) -\frac{3}{2}\kappa\sum_{i\ne j}\frac{x_i^2+x_j^2}{(x_i-x_j)^2} + \frac{3}{2}n + \frac{\kappa n(n-1)}{2}.
$$
The next one is
$$
{\rm ad}^{3}_{f}(X) \backsim_{\beta} \frac{1}{2}[\sum_i \partial_i^2, \sum_j(3x_j^2\partial_j^2 + 6x_j\partial_j)] - \frac{3}{4}\kappa[\sum_m\partial_m^2, \sum_{i\ne j} \frac{x_i^2+x_j^2}{(x_i-x_j)^2}] - 
$$
$$
-\frac{\kappa}{2}[\sum_{i\ne j}\frac{1}{(x_i-x_j)^2}, \sum_m(3x_m^2\partial_m^2 + 6x_m\partial_m)]. 
$$
The second commutator in this formula amounts to:
$$
\sum_{k,i\ne j}[\partial_k^2, \frac{x_i^2+x_j^2}{(x_i-x_j)^2}] = \sum_{i\ne j} \left( 2\frac{2x_i\partial_i +2x_j\partial_j}{(x_i-x_j)^2} + 2\frac{(x_i^2+x_j^2)(-2\partial_i+2\partial_j)}{(x_i-x_j)^3} +\frac{2+2}{(x_i-x_j)^2} +\right. 
$$
$$
\left.+2\frac{-4x_i+4x_j}{(x_i-x_j)^3} + \frac{(x_i^2+x_j^2)(6+6)}{(x_i-x_j)^4}\right) = 
$$
$$
=4\sum_{i\ne j}\left(\frac{(x_i+x_j)(x_i\partial_j - x_j \partial_i)}{(x_i-x_j)^3} - \frac{1}{(x_i-x_j)^2} + 3\frac{(x_i^2+x_j^2)}{(x_i-x_j)^4}\right)  ,
$$
and the third one:
$$
[\sum_{i\ne j}\frac{1}{(x_i-x_j)^2}, \sum_m(3x_m^2\partial_m^2 + 6x_m\partial_m)] = -3\sum_{i\ne j}\left(2\frac{-2x_i^2\partial_i + 2x_j^2\partial_j}{(x_i-x_j)^3} + \frac{6x_i^2+6x_j^2}{(x_i-x_j)^4} +2\frac{-2x_i+2x_j}{(x_i-x_j)^3}   \right).
$$
So the original expression amounts to:
$$
{\rm ad}^{3}_{f}(X) \backsim_{\beta} 3\sum_i(2x_i\partial_i^3 + 3\partial_i^2) + 3\kappa\sum_{i\ne j}\left( \frac{2x_j^2\partial_j -2x_i^2\partial_i -(x_i+x_j)(x_i\partial_j-x_j\partial_i)}{(x_i-x_j)^3}-3\frac{1}{(x_i-x_j)^2} \right) = 
$$
$$
=3\sum_i(2x_i\partial_i^3 + 3\partial_i^2) - 3\kappa\sum_{i\ne j}\left(\frac{x_i\partial_j + x_j\partial_i + 2x_j\partial_j + 2x_i\partial_i}{(x_i-x_j)^2} +3\frac{1}{(x_i-x_j)^2}\right). 
$$

So, now we can finally compute the image of the relation:
$$
3[{\rm ad}^2_{f}(X),{\rm ad}_{f}(X)] - 2[{\rm ad}^3_{f}(X),X] \backsim_{\beta} $$
$$
\backsim_{\beta} 3[3\sum_i(x_i^2\partial_i^2 + 2x_i\partial_i), \sum_j(x_j^3\partial_j + \frac{3}{2}x_j^2)] - 2[3\sum_i(2x_i\partial_i^3 + 3\partial_i^2), \frac{1}{4}\sum_j x_j^4] -
$$
$$
-\frac{9}{2}\kappa[\sum_{i\ne j}\frac{x_i^2+x_j^2}{(x_i-x_j)^2},\sum_mx_m^3\partial_m] + 6\kappa[\sum_{i\ne j}\frac{x_i\partial_j + x_j\partial_i + 2x_i\partial_i + 2x_j\partial_j}{(x_i-x_j)^2}, \frac{1}{4}\sum_m x^4_m].
$$
The part coming from the first two commutators is just the r.h.s. of the relation when $k=0$. It is equal to:
$$
-15 \cdot 3 \sum_i x_i^2 \backsim 3\cdot 2\cdot 15 b_1  ,
$$
as we would expect since in this case $s_1 = 1, s_2 = 0$.

The third commutator gives:
$$
[\sum_{i\ne j}\frac{x_i^2 + x_j^2}{(x_i-x_j)^2}, \sum_m x_m^3 \partial_m] = -\sum_{i\ne j}\frac{2x_i^4 + 2x_j^4}{(x_i-x_j)^2} + \sum_{i \ne j} \frac{2(x_i^2+x_j^2)(x^3_i-x^3_j)}{(x_i-x_j)^3} =
$$
$$
=2\sum_{i\ne j} \frac{(x_i^2 +x_j^2)(x_i^2+x_ix_j+x_j^2) - x_i^4-x_j^4 }{(x_i-x_j)^2},
$$
and the forth one:
$$
[\sum_{i\ne j}\frac{x_i\partial_j + x_j\partial_i + 2x_i\partial_i +2x_j\partial_j}{(x_i-x_j)^2},\sum_m x_m^4] = 4\sum_{i\ne j}\frac{x_ix_j^3 + x_jx_i^3 + 2x_i^4 + 2x_j^4}{(x_i-x_j)^2}.
$$
If we put together the formulas for the third and forth commutators in the original expression, we get:
$$
3[{\rm ad}^2_{f}(X),{\rm ad}_{f}(X)] - 2[{\rm ad}^3_{f}(X),X] - 3\cdot 15\cdot 2b_1 \backsim_{\beta} 
$$
$$
\backsim_{\beta} 3\kappa\sum_{i\ne j}\frac{-3(2x_i^2x_j^2 + x_i^3x_j + x_j^3x_i) + 2(x_ix_j^3 + x_jx_i^3 + 2x_i^4 +2x_j^4)}{(x_i-x_j)^2} =
$$
$$
=3\kappa \sum_{i\ne j}\frac{4x_i^4 + 4x_j^4-6x_i^2x_j^2 -x_ix_j^3-x_jx_i^3}{(x_i-x_j)^2} = 3\kappa \sum_{i\ne j}\frac{4(x_i^2+x_j^2)(x_i-x_j)^2 + 7x_ix_j(x_i-x_j)^2}{(x_i-x_j)^2}=
$$
$$
= 3\kappa\sum_{i\ne j}(8x_i^2 + 7x_ix_j) 
=3(8\kappa(n-1)\sum_i x_i^2 + 7(\sum_i x_i)^2 - 7\sum_{i}x_i^2) \backsim_{\beta}
$$
$$
\backsim_{\beta} 3\kappa(-16b_1(K-1) + 7(a_1^2 + 2b_1))  .
$$
Thus we see that:
$$
3[{\rm ad}^2_{f}(X),{\rm ad}_{f}(X)] - 2[{\rm ad}^3_{f}(X),X] \backsim_{\beta} 3(2(15-\kappa(8K-15))b_1 +7\kappa a_1^2) = 
$$
$$
=3((30(1+\kappa(1-K)) +14\kappa K)b_1 +7\kappa a_1^2) .
$$
And we can conclude that $s_1 = 1+k(k+1)(1-K)$ and $s_2 = k(k+1)$.
\end{proof}

\begin{rem}
Note that instead of using the computer calculation from Proposition \ref{propdef1}, we could have defined the map on generators by the same formula as $\beta$ and checked that it satisfies the remaining relations. This is easy to do, in fact the relation we have checked is the most complicated one. 
\end{rem}

\begin{rem}
One can think about the isomorphism of Proposition \ref{defpropDDC} in the following way. For the Lie algebra $\mathbb C[x,\partial]$ there exists a standard map:
$$
U(\mathbb C[x,\partial]) \to S^n\mathbb C[x,\partial] = \textrm{Diff}(\mathbb C^n)^{S_n}  .
$$

One can deform this map to arrive at the map:
$$
A_{s_1,s_2} \to  \left(\textrm{Diff}(\mathbb C^n)\left[\frac{1}{\prod_{1 \le i<j \le n} (x_i-x_j)}\right]\right)^{S_n}  ,
$$
with $s_1 = 1+k(k+1)(1-n)$ and $s_2 = k(k+1)$. These maps are given by the formulas in the polynomial representation of the Cherednik algebra, which we used in the proof of Proposition \ref{defpropDDC}. The isomorphism $\beta$ can be thought of as a certain ultraproduct of these maps.
\end{rem}

We have another corollary:
\begin{cor}
The algebra $A_{1 +k(k+1)(1-K),k(k+1)}$ is a flat filtered deformation of $U(\mathfrak{po})$.
\end{cor}

\begin{rem}
Note that via Proposition \ref{propquot} we can also easily obtain the presentation by generators and relations of DDC-algebras $\mathcal D_{1,k,\nu}$.
\end{rem}

\subsection{The Galois symmetry} 

Recall that the algebra $\mathcal D_{1,k,\nu}$ is the quotient of $\mathcal D^{\rm ext}_{1,k}$. Indeed, by Proposition \ref{propquot} we have $\mathcal D_{1,k,\nu} = \mathcal D^{\rm ext}_{1,k,\nu}/(K - \nu)$. I.e., it is an algebra where the central parameter $K$ became a scalar.

Now we can see that the equations for $s_1$ and $s_2$ in Proposition \ref{defpropDDC} can be written in terms of the essential parameters $s_1^*=s_1\nu^2$ and $s_2^*=s_2\nu^3$ as follows: 
\begin{equation}\label{equati}
s_1^*=(k^2+k+1)\nu^2-k(k+1)\nu^3,\ s_2^*=k(k+1)\nu^3.
\end{equation} 
It is easy to check that these equations are invariant under the symmetry 
$$
g_1(k,\nu):=\left(\frac{1}{k},k\nu\right).
$$
This implies 
\begin{prop} We have an isomorphism of filtered algebras
$$
\mathcal D_{1,k,\nu}\cong {\mathcal D}_{1,\frac{1}{k},k\nu}.
$$
\end{prop}  
This proposition also follows from the results of \cite{CEE}, Sections 8,9, see also \cite{EGL}, Section 6, which establish similar symmetries for spherical Cherednik algebras of finite rank. There is also an obvious symmetry $g_2(k,\nu)=(-k-1,\nu)$. 
It is easy to see that $g_1^2=g_2^2=1$, $(g_1g_2)^3=1$, so $g_1,g_2$ generate a copy of the group $S_3$. In fact, this $S_3$-symmetry comes from permuting the parameters $q_1,q_2,q_3$ in the toroidal quantum group (\cite{Mi}), which can be degenerated to $\mathcal D_{1,k,\nu}$. 

Moreover, this group is also the Galois group of the system of equations \eqref{equati}. 
Namely, we have 
$$
s_1^*+s_2^*=(1+k+k^2)\nu^2,\ s_2^*=k(k+1)\nu^3,
$$
so 
$$
(1+k+k^2)^3=uk^2(k+1)^2,
$$
where $u:=\frac{(s_1^*+s_2^*)^3}{s_1^{*2}}$.
Dividing this by $k^3$, we get 
$$
\zeta^3-u\zeta+u=0, 
$$
where $\zeta:=k+\frac{1}{k}+1$. The group $S_3$ just mentioned is the Galois group of this cubic equation over $\Bbb C(u)$. Namely, $\Bbb C(\zeta)$ is a non-normal cubic extension of $\Bbb C(u)$, and $\Bbb C(k)$ is the corresponding splitting field (a quadratic extension of $\Bbb C(\zeta)$). 

\section{Deformed double current algebras for arbitrary $\Gamma$}

\subsection{The case of general $\Gamma$} 
In this section we will repeat the construction of Section 4.1 for the DDCA corresponding to arbitrary $\Gamma$. Here again for brevity we consider only the case of transcendental $\nu$. Since the construction is literally the same upon changing $\Rep(S_\nu)$ to $\Rep(S_\nu \ltimes \Gamma^\nu)$, we will go over it rather quickly.

First we start with a definition.
\begin{def0} \label{Htkcdef}
The object $H_{t,k,c}(\nu,\Gamma)\be \in \Rep(H_{t,k,c}(\nu,\Gamma))$ is defined to be equal to $\Ind_{S_\nu\ltimes \Gamma^\nu}^{H_{t,k,c}(\nu,\Gamma)}(\Bbbk)$. It follows that $H_{t,k,c}(\nu,\Gamma)\be = \prodF^{C,r}H_{t_n,k_n,c_n}(n,\Gamma)\be$.
\end{def0}

Note that assigning $\deg(V)=1$ gives us the filtration on $H_{t,k,c}(\nu,\Gamma)\be$ in the same fashion as in the discussion after Lemma \ref{HElemma}. The same filtration works in finite rank.

Now we can define the DDCA itself:
\begin{def0}
The DDC algebra $\mathcal D_{t,k,c,\nu}(\Gamma)$ is given by:
$$
\mathcal D_{t,k,c,\nu}(\Gamma):=\End_{\Rep(H_{t,k,c}(\nu, \Gamma))}(H_{t,k,c}(\nu,\Gamma)\be) = \Hom_{\Rep(S_\nu\ltimes \Gamma^\nu)}(\mathbb C, H_{t,k,c}(\nu,\Gamma)\be)  .
$$
\end{def0}

Similarly to Proposition \ref{DDCult}, we have:
\begin{prop}
The algebra $\mathcal D_{t,k,c,\nu}(\Gamma)$ can be constructed as the restricted ultraproduct of spherical subalgebras $\prodF^r \mathcal B_{t_n,k_n,c_n}(n,\Gamma)$ with respect to the filtrations mentioned after Definition \ref{Htkcdef}.
\end{prop}

\begin{rem}
We can also do the same thing in the Deligne categories over $\Cext$ and obtain the algebra $\widetilde{\mathcal D}_{t,k,c,\nu}(\Gamma)$ over $\Cext$.
\end{rem}

\begin{rem}
The  analogs of the results of Section 4.1.3 still hold and we can also construct the algebra $\mathcal D^{\rm ext}_{t,k,c}(\Gamma)$ over $\mathbb C$, where $\nu$ becomes a central element. 
\end{rem}

\begin{rem}
Note that we obtain the case of type A if we set $\Gamma = 1$, the trivial group. i.e., we have $\mathcal D_{t,k,\emptyset,\nu}(1) = \mathcal D_{t,k,\nu}$.
\end{rem}

\subsection{The deformed double current algebra of type B}

In this section we would like to sketch some results on the presentation of the DDCA in type B by generators and relations akin to the discussion for type A in Section 4.2. Most of the results of this section were obtained through a computer computation.

First of all note that we can obtain the DDCA of type B by taking $\Gamma = \mathbb Z/2$.
\begin{def0}
Denote $\mathcal D_{t,k,c,\nu}: = \mathcal D_{t,k,c,\nu}(\mathbb Z/2)$. Here $c$ is just a single number, since $\mathbb Z/2$ has a single non-trivial conjugacy class. Define $\widetilde{\mathcal D}_{t,k,c,\nu}$ and $\mathcal D^{\rm ext}_{t,k,c}$ in the same way.
\end{def0}

We saw that $\mathcal D^{\rm ext}_{t,k}$ was a deformation of $U(\mathfrak{po})$. It turns out that a similar statement holds for type B.

\begin{def0}
By $\po^+$ denote the Lie subaglebra of $\po$ given by the linear combinations of even degree monomials. I.e., $\po^+ = \po^{\mathbb Z/2}$, where $\mathbb Z/2$ acts on $\po$ by $p \mapsto -p$ and $q \mapsto -q$. This Lie algebra has an even grading restricted from the grading of $\po$, and this grading is also a grading by $\mathfrak{sl}_2$-modules under the adjoint action of $\po_0$.
\end{def0}

It's now easy to see, by similar arguments, that whereas the ultraproduct of type A algebras $\be H_{t,k}(n)\be$ which are isomorphic to $\FQ[x_1,\dots,x_n,y_1,\dots,y_n]^{S_n}$ as vector spaces is a deformation of $U(\po)$, the ultraproduct of type $B$ algebras $\be H_{t,k,c}(n)\be$ which are isomorphic as vector spaces to $\FQ[x_1,\dots,x_n,y_1,\dots, y_n]^{S_n \ltimes \mathbb (Z/ 2 \mathbb Z)^n}$ is a deformation of $U(\po^+)$.

Now one can also provide a presentation of $\po^+$ similar to Proposition \ref{propgen2}. To state such a result we need to give a few definitions.

\begin{def0}
Denote by $\mathfrak{b}$ the Lie subalgebra of $\po^+$ given by $\po^+_{-2} \oplus \po^+_{0}$. The Lie subalgebra $\mathfrak{n}$ is given by $\po^+_2\oplus \po^+_4 \oplus \dots$. So $\po^+ = \mathfrak{b}\oplus \mathfrak{n}$.\footnote{The algebras $\mathfrak{b}$ and $\mathfrak{n}$ should not be confused with their  analogs from Section 4.} 
\end{def0}

We will need a little more notation:
\begin{def0} \label{typeBpodef}
Fix an isomorpism of $\mathfrak{b}_0$ with $\mathfrak{sl}_2$ given by $e \to b_1 = -\frac{q^2}{2}$ and \linebreak $f \to b_3 = \frac{p^2}{2}$. Fix an isomorphism of $\mathfrak{b}_{-2}$ with $V_0$ with the highest weight vector specified as $K$.

Fix an isomorphism of $\mathfrak{n}_2$ with $V_4$ with the highest weight vector specified as $d_1 = \frac{q^4}{8}$.

Consider the free Lie algebra $L(\mathfrak{n}_2)$. Consider $\Lambda^2\mathfrak{n}_2 \subset L(\mathfrak{n}_2)_4$. As $\mathfrak{sl}_2$-modules we have $\Lambda^2\mathfrak{n}_2 \simeq V_6 \oplus V_2$. Denote the submodule of $\Lambda^2 \mathfrak{n}_2$ isomorphic to $V_0$ by $\phi_1'$ and the submodule isomorphic to $V_6$ by $\phi_2'$. Fix an isomorphism of $\phi_1'$ with $V_2$ with the highest weight vector specified as $3d_2\wedge d_3-2d_1\wedge d_4$. Fix an isomorphism of $\phi_2'$ with $V_6$ with the highest weight vector specified as $g_1 = d_2\wedge d_1$. Here $d_i = f^{i-1}d_1$.

Consider $\phi_2'\otimes \mathfrak{n}_2 \subset L(\mathfrak{n}_2)_6$. We have $\phi_2' \otimes \mathfrak{n}_2 \simeq V_{10} \oplus V_8 \oplus V_6 \oplus V_4 \oplus V_2$ as $\mathfrak{sl}_2$-modules. Denote the submodule isomorphic to $V_{10}$ by $\psi_5'$ and the submodule isomorphic to $V_4$ by $\psi_2'$. Fix an isomorphism of $\psi_2'$ with $V_4$ with the highest weight vector specified as $g_4\otimes d_1 -3g_3 \otimes d_2 +5g_2 \otimes d_3 - 5 g_1 \otimes d_4$, where $g_i = f^{i-1}g_1$.
\end{def0}

Now we can state how the presentation of $\po^+$ looks like:
\begin{prop} \label{proprelB} (see \cite{GL}, Table 3.1, relations 2.1, 3.2, 3.3) 
The Lie algebra $\po^+$ is generated by $\mathfrak{b} \oplus \mathfrak{n}_2$ with the following set of relations:
\begin{gather*}
    \mathfrak{b}_0 \simeq \mathfrak{sl}_2  , \ \mathfrak{b}_{-2} \simeq V_0 \text{ is central}  , \ \mathfrak{n}_2 \simeq V_4 \text{ as an $\mathfrak{sl}_2$-module}  , \\
    \phi_1' = 0  , \ \psi_5' = 0  , \ \psi_2'= 0  ,
\end{gather*}
where we use fixed isomorphisms from Definition \ref{typeBpodef}. 
\end{prop}

Via a computer calculation similar to Proposition \ref{propdef1} one can obtain a result about a certain class of flat filtered deformations of $U(\po^+)$. To state that result we will need one more definition:
\begin{def0}
Consider the free associative algebra $T(\mathfrak{b} \oplus \mathfrak{n}_2)$. Note that the subspace $S^2\mathfrak{b}_0 \subset T(\mathfrak{b} \oplus \mathfrak{n}_2)$ is isomorphic to $V_4 \oplus V_0$ as a $\mathfrak{sl}_2$-module. Denote the submodule isomorphic to $V_4$ as $\alpha'$. Fix an isomorphism of $\alpha'$ with $V_4$ with the highest weight vector specified by $e^2$.
\end{def0}

\begin{prop} \label{propdeftypeB}
Suppose $U$ is a flat filtered deformation of $U(\po^+)$ as an associative algebra (up to an automorphism), such that $U(\mathfrak{b})$ is still a subalgebra of  $U$, and the action of $U(\mathfrak{b})$ on $\mathfrak{n}_2$ is not deformed. Then $U$ is isomorphic to the algebra $A_{s_1,s_2,s_3}$ defined below for some values of $s_1,s_2$ and $s_3$. The algebra $A_{s_1,s_2,s_3}$ is generated by $\mathfrak{b}\oplus \mathfrak{n}_2$ with the set of relations given by the first line of Proposition \ref{proprelB} and the following relations, which substitute the last line in Proposition \ref{proprelB}:
\begin{gather*}
\phi_1' \simeq 6s_1\mathfrak{b}_0  , \ \psi_5' = 0  , \\
\psi'_2 \simeq 24(s_3\alpha' + 12s_2\mathfrak{n}_2),
\end{gather*}
where the notation used is understood in the same way as in Proposition \ref{propdef1}.
\end{prop}

Now we can state the result about the DDC-algebra of type $B$.

\begin{prop}\label{typB}
The DDC-algebra $\mathcal D^{\rm ext}_{1,k,c}$ is a flat filtered deformation of $U(\po^+)$. It is isomorphic to $A_{s_1,s_2,s_3}$ with 
$$
s_1 = 4k(k+1)K+\lambda^2-4(k^2+k+1),\ s_2 =4k(k+1)K+\lambda^2-9(k^2+k+1),\ s_3 = k(k+1),
$$
where $\lambda:= c+\frac{1}{2}$. 
\end{prop}

\begin{rem}
Notice that in the same way as $\mathbb C[x,\partial]$ is a flat filtered deformation of $\po$, Feigin's Lie algebra $\mathfrak{gl}(\lambda):=U(\mathfrak{sl}_2)/(C=\frac{\lambda^2-1}{2})$ (where $C:=ef+fe+\frac{h^2}{2}$ is the Casimir) introduced in \cite{feigin1988lie} is a flat filtered deformation of $\po^+$. More precisely we have $U(\mathfrak{gl}(\lambda)) \simeq \mathcal{D}_{1,0,\lambda-\frac{1}{2},\nu}$ (for any $\nu$, as this algebra does not depend of $\nu$); indeed, it is easy to see looking at the relations that the deformation $\mathcal{D}_{1,0,\lambda-\frac{1}{2},\nu}$ arises from the most general deformation of $\po^+$ as a (filtered) Lie algebra. These relations are given in \cite{GL}, at the beginning of Table 3.1. For more information about deformations of $\po^+$ and $\mathfrak{gl}(\lambda)$ see \cite{post1996gl}.
\end{rem}

Note that the parameters $s_1,s_2,s_3$ 
in Proposition \ref{typB} are not independent: we have 
\begin{equation}\label{rela}
s_1-s_2=5(s_3+1). 
\end{equation} 
This is, however, the most general deformation because the parameters 
$s_1,s_2,s_3$ are homogeneous of degrees $4,4,6$, hence can be rescaled 
by $s_1\mapsto z^2s_1$, $s_2\mapsto z^2s_2$, $s_3\mapsto z^3s_3$ without changing the algebra. This implies 

\begin{cor} The deformation $A_{s_1,s_2,s_3}$ is flat for all $s_1,s_2,s_3$. 
\end{cor} 

On the other hand, this means that for the algebra $\mathcal D_{1,k,\lambda-\frac{1}{2},\nu}$ we have only two essential parameters, so we cannot recover all the three parameters 
$\nu$, $\lambda$ and $k$ from $s_1,s_2,s_3$. More precisely, we can recover
$k$ (up to the symmetry $k\mapsto -k-1$) and the combination 
$\lambda^2-4k(k+1)\nu$. This gives 

\begin{cor} One has a filtered isomorphism 
$$
\mathcal D_{1,k,\lambda-\frac{1}{2},\nu}\cong \mathcal D_{1,k,\sqrt{\lambda^2-4k(k+1)\nu}-\frac{1}{2},0}.
$$ 
\end{cor} 

Thus, to study the most general DDCA of type $B$ (with $t\ne 0$), 
it suffices to consider the algebras $\mathcal D_{1,k,\lambda-\frac{1}{2},0}$. 
In this case, we have 
$$
s_1-s_2=5(k^2+k+1),\ 9s_1-4s_2=5\lambda^2,\ s_3=k(k+1). 
$$
The essential parameters (unchanged under scaling) are 
$$
u=\frac{(s_1-s_2)^3}{125s_3^2}=\frac{(k^2+k+1)^3}{k^2(k+1)^2},\ v=\frac{4s_2-9s_1}{s_2-s_1}=\frac{\lambda^2}{k^2+k+1}.
$$
It is easy to see that $u,v$ are invariant under the symmetry\footnote{The symmetry $h_1$ appears to be related to the $q\to 1$ limit of a symmetry of
``lifted'' Koornwinder polynomials \cite[Prop. 7.4]{YY}.  These are
symmetric functions that analytically continue the Koornwinder
polynomials in the dimension (introducing an additional parameter in the
process), and satisfy a duality transformation swapping $q$ and $t$, and
thus inverting $k=\log_q(t)$.}
$$
h_1(k,\lambda):=\left(\frac{1}{k},\frac{\lambda}{k}\right). 
$$

Thus, we obtain the following proposition. 

\begin{prop} We have an isomorphism of filtered algebras
$$
\mathcal{D}_{1,k,\lambda-\frac{1}{2},0}\cong 
\mathcal{D}_{1,\frac{1}{k},\frac{\lambda}{k}-\frac{1}{2},0}.
$$
\end{prop} 

We also have symmetries $h_2(k,\lambda)=(-k-1,\lambda)$ and $h_3(k,\lambda)=(k,-\lambda)$, which generate the group $S_3\times \Bbb Z/2$. This group is the Galois group of the extension $\Bbb C(k,\lambda)$ over $\Bbb C(u,v)$.
Indeed, as in type A, we have 
$$
\zeta^3-u\zeta+u=0, 
$$
where $\zeta:=k+\frac{1}{k}+1$. 
So the group $S_3$ is the Galois group of this cubic equation. Furthermore, once $k$ is found, 
we can find $\lambda^2$ from the equation 
$$
\lambda^2=(k^2+k+1)v,
$$
whose Galois group is $\Bbb Z/2$.

\begin{rem} We see that when we interpolate the spherical Cherednik algebras $\bold e H_{1,k,c}(S_n)\bold e$ of type $B$ into the DDCA $\mathcal D_{1,k,c,\nu}$, we lose one parameter (unlike the case of type $A$). Let us explain why such a loss of parameter is inevitable and to be expected a priori. 

To this end, note that for generic $k,c$ the algebra $\bold e H_{1,k,c}(n)\bold e$ is simple and therefore has no nonzero finite dimensional representations. 

On the other hand, let $\overline{\mathfrak{po}}^+:=\mathfrak{po}^+/\Bbb C$. 
We claim that any filtered deformation of $S(\overline{\mathfrak{po}}^+)$ necessarily has a 1-dimensional representation. 

Indeed, let $A$ be such a deformation. Let us show that the augmentation homomorphism $\varepsilon: S(\overline{\mathfrak{po}}^+_0)\to \Bbb C$ lifts to a 1-dimensional representation of $A$.
By definition, $A$ has generators $\bold a=(a_{ij})$ with $i+j>0$ even ($i,j\ge 0$) of filtration degree $i+j$ (namely, lifts of $p^iq^j$) and has defining relations
$$
[a_{ij},a_{kl}]=P_{ijkl}(\bold a),
$$
where $P_{ijkl}$ is a noncommutative polynomial of degree $\le i+j+k+l-2$ whose part of degree exactly $i+j+k+l-2$ is $(jk-il)a_{i+k-1,j+l-1}$. In particular, setting 
$$
-a_{20}/2=e,\ a_{11}=h,\ a_{02}/2=f, 
$$
we get 
$$
[h,e]=2e+c_1, \
[h,f]=-2f+c_2, \ [e,f]=h+c_3, 
$$
for some constants $c_i\in \Bbb C$. These constants give a 2-cocycle on $\mathfrak{sl}_2$, which must be a coboundary since $H^2(\mathfrak{sl}_2)=0$, so by shifting $e,f,h$ by constants we can make sure that $c_i=0$. Thus $A$ contains $\mathfrak{sl}_2$ and we can write all the relations $\mathfrak{sl}_2$-equivariantly. Thus we can assume that $a_{ij}$ with $i+j=2m$ span the representation $V_{2m}$. So $[a_{ij},a_{kl}]$ belongs to the representation $V_{2m}\otimes V_{2n}$ if $i+j=2m$, $k+l=2n$, $m\ne n$, and to $\Lambda^2 V_{2m}$ if $m=n$. These representations don't contain $\Bbb C$, so the polynomial $P_{ijkl}$ can be chosen without constant term. Thus we have a 1-dimensional $A$-module in which all $a_{ij}$ act by 0, as claimed. 

Note that this argument (and the statement itself) fails for $\mathfrak{po}$ (type $A$), since $\Lambda^2 V_r$ contains $\Bbb C$ for odd $r$. 

This loss of parameter is similar to the one for Deligne categories: the interpolation 
of the category ${\rm Rep}GL(m|n)$ is the Deligne category ${\rm Rep} GL_\nu$ with $\nu=m-n$. 
So the interpolation procedure forgets $m,n$ and remembers only the difference $m-n$. 
\end{rem}  

\appendix
\section{Appendix: On structure constants of the deformed double current algebra of type A}

As was promised in the proof of Lemma \ref{lemmaappA}, in this Appendix we will prove that the structure constants of $\mathcal B_{t,k}(n)$ depend polynomially on $n$, making this sequence of algebras fit  Example \ref{infalgex}. So, what we want to show is that $T_n({{\bf m}}_1) \cdot T_n({{\bf m}}_2)$ can be written as a linear combination of $T_n({{\bf m}})$ with coefficients which depend polynomially on $n$. This proof is due to Travis Schedler.

To start with, we will need the following definition.
\begin{def0}
For a function $a: [l] \to \{x,y\}$ (where $[l] = \{1,\dots, l\}$), a function $u:[l] \to [m]$ and a number $C \in \Bbbk$ call an element of $\mathcal B_{t,k}(n)$ given by:
$$
C\sum_{i_1,\dots,i_m=1}^N a(1)_{i_{u(1)}}\dots a(l)_{i_{u(l)}}\be \  
$$
an admissible sum.
\end{def0}

It is easy to see that
$$
T_{r,q,n} = \frac{r!q!}{(r+q)!} \sum_{\substack {a:[r+q] \to \{x,y\} \\  |a^{-1}(x)|=r}}\sum_{i=1}^na(1)_i\dots a(r+q)_i  .
$$
So $T_{r,q,n}$ is a sum of admissible sums with $l=r+q$ and $m=1$ with coefficients which do not depend on $n$. 

It is also easy to see that the product of admissible sums is an admissible sum (we just need to combine two pairs of functions into a single pair by concatenation). So, since $T_n({{\bf m}})$ is given by the sum of products of $T_{r,q,n}$ with coefficients which do not depend on $n$, it follows that $T_n({{\bf m}})$ is a sum of admissible sums with coefficients which do not depend on $n$.

We  are now ready to prove the following proposition.
\begin{prop}
The product of $T_n({{\bf m}}_1)$ and $T_n({{\bf m}}_2)$ can be written as a sum of $T_n({{\bf m}})$ with coefficients which depend polynomially on $n$.
\end{prop}
\begin{proof}

By the preceding discussion $T_n({{\bf m}}_1)T_n({{\bf m}}_2)$ is a sum of admissible sums with coefficients independent of $n$. So it is enough to prove that any admissible sum can be written as a sum of $T_n({{\bf m}})$ with coefficients which depend on $n$ polynomially. We will prove this by induction on $l$ (the cardinality of the source of $a$ in the definition of admissible sum, or the degree of admissible sum). Assume that the statement holds for admissible sums with $l < M$.

Suppose we are given an admissible sum $A$ with $l=M$ defined by functions \linebreak $a:[M] \to \{x,y\}$ and $u: [M] \to [k]$. If $u$ is not surjective, it follows that some of the summations are redundant and just give a coefficient in the form of the power of $n$ (this is where polynomial dependence on $n$ actually comes from). So we can reduce to the case of $u$ being surjective. For $j \in [k]$  define $r_j = |u^{-1}(j)\cap a^{-1}(x)|$ and $q_j = |u^{-1}(j)\cap a^{-1}(y)|$, i.e., this is the number of $x$'s and $y$'s in the summation corresponding to $i_j$. Define ${{\bf m}}$ in the following way: we set $m_{r,q}:=|\{j \in [k]| \ r_j =r  , q_j=q\}|$. We want to prove that there is a  number $\alpha$ which does not depend on $n$ such that $A-\alpha T({{\bf m}})$ is given by the sum of admissible sums with  $n$-independent coefficients all of which have degree $l<M$. Indeed, by the choice of ${{\bf m}}$ it follows that the highest orders of $A$ and $T_n({{\bf m}})$ are proportional up to some  factor coming from the factorials in the definition of $T_n({{\bf m}})$, so choose $\alpha \in \Bbbk$ such that ${\rm gr}^{w({{\bf m}})}(A) - \alpha \cdot {\rm gr}^{w({{\bf m}})}(T_n({{\bf m}})) = 0$.

To calculate the actual difference one would need to permute $x$'s and $y$'s in $A$ to bring it to the form of $T_n({{\bf m}})$. Obviously when we permute $x$'s and $y$'s, an admissible sum stays admissible. So we only need to see what happens with the parts arising due to commutators.

If we commute $x_{i_1}$ with $y_{i_2}$, we reduce the number of generators by $2$ (so the resulting degree is less than $M$) and insert $t-k\sum_{m \ne i_1}s_{m,i_q}$ in case $i_1= i_2$ or $ks_{i_1i_2}$ in case of $i_1 \ne i_2$. In either case we can commute group elements to the right and absorb them into $\be$. What we have afterwards is a sum of sums which only differ from admissible sums by the fact that they sometimes have the condition $i_1\ne i_2$. But since \linebreak $\sum_{i_1\ne i_2} = \sum_{i_1,i_2} - \sum_{i_1=i_2}$ this reduces to the sum of admissible sums, and we are done. 
\end{proof}

\section{Appendix: Direct calculation of generators and relations for $\mathfrak{po}$}

Here we would like to give a direct proof of Proposition \ref{1proprel}. As we mentioned in the main text, we have a surjective map $\pi:L(\mathfrak{n}_1) \to \mathfrak{n}$. Let us denote the ideal generated by $\phi_1$, $\psi_4, \psi_1$ and $\chi_1$ as $I \subset L(\mathfrak{n}_1)$. It's easy to see that this ideal is in the kernel of $\pi$. Indeed, to conclude that we only need to know that the generators of $\mathfrak{n}_1$ satisfy the four last relations of Proposition \ref{propgen3}, which is a straightforward calculation. Let us denote the quotient $L(\mathfrak{n}_1)/I$ by $\mathfrak{l}$. So we have a surjective map $\pi': \mathfrak{l} \to \mathfrak{n}$. We only need to prove that $\pi'$ is injective.

\begin{prop}
The map $\pi'$ just described is an isomorphism.
\end{prop}
\begin{proof}
Both algebras have a natural grading given by assigning $\mathfrak{n}_1$ to have degree $1$. We will prove that $\pi'$ is an isomorphism by induction. 

It will be easier for us to begin with the induction step. I.e., we will prove that if $\pi'$ is an isomorphism for all degrees up to $l-2$ (with $l \ge 6$), then it is an isomorphism for $l-1$. We will prove the base of induction (i.e., the fact that $\pi'$ is an isomorphism for degrees $2,3$ and $4$) later, using the general formulas we derived.

So, suppose we know that up to $l-2$ we have $\mathfrak{l}_{j} \simeq \mathfrak{n}_j \simeq V_{j+2}$. It means that if we want to show that a certain element of $L(\mathfrak{n}_1)$ is in $I$, we can freely commute elements with total degree $\le j$ as though they were the elements of $\po$. Indeed, this will only add to our elements something which is already contained in $I$. Let us denote the highest weight vector of $\mathfrak{l}_j$ by $v_1^{j+2}$, which corresponds to $q^{j+2}$ under the above isomorphism (i.e. $\mathfrak{l}_{j} \simeq \mathfrak{n}_j \simeq V_{j+2}$). We set $v_i^{j} = f^{i-1}v_1^{j}$.

Now we know that $\mathfrak{l}_{l-1}$ is a quotient of $\mathfrak{l}_{l-2} \otimes \mathfrak{n}_1 \simeq V_{l+3}\oplus V_{l+1} \oplus V_{l-1}\oplus V_{l-3}$, i.e., we have a surjective map $\xi_{l-1}: V_{l+3}\oplus V_{l+1} \oplus V_{l-1}\oplus V_{l-3} \to \mathfrak{l}_{l-1}$. We only need to prove that $ \xi_{l-1}(V_{l+3}\oplus V_{l-1}\oplus V_{l-3}) = 0$. 

We would like to describe the highest weight vectors of the simple $\mathfrak{sl}_2$-modules in the decomposition of $\mathfrak{l}_{l-2} \otimes \mathfrak{n}_1$ explicitly. To do so, it is enough to find the vectors of the required weight which are annihilated by the action of $e$. It is easy to see that the highest weight vector of $V_{l+3}$ is proportional to $v_1^{l}\otimes c_1$; the highest weight vector of $V_{l+1}$ to $3v_2^{l}\otimes c_1 - l v_1^l \otimes c_2$; the highest weight vector of $V_{l-1}$ to 
$$
6v_3^l\otimes c_1 - 4(l-1)v_2^{l}\otimes c_2 + l(l-1)v_1^l \otimes c_3;
$$ 
the highest weight vector of $V_{l-3}$ to 
$$
6v_4^l \otimes c_1 - 6(l-2)v_3^l\otimes c_2 + 3(l-2)(l-1)v_2^l \otimes c_3 - l(l-1)(l-2)v_1^l \otimes c_4.
$$ 
Writing the highest weight vectors in this way allows us to write the action of $\xi_{l-1}$ in a straightforward way, i.e., $y \otimes x \in \mathfrak{l}_{l-2} \otimes \mathfrak{n}_1$ is mapped into $\xi_{l-1}(y \otimes x) \to [y,x]$.

Now we need to prove that each of the highest weight vectors corresponding to $V_{l+3}$, $V_{l-1}$ and $V_{l-3}$ belongs to $I$, i.e., maps to zero under $\xi_{l-1}$. Let us start with $V_{l+3}$. Now we know that $[v_2^{l-2},[d_1,c_1]]$ is in $I$. If we transform this expression using the commutator formulas in $\po$ for elements of degree less or equal than $l-2$, we will stay in $I$ by the induction assumption. So,  in $\mathfrak{l}_{l-1}$ we have:
$$
0 = [v_2^{l-2},[d_1,c_1]] = [[v_2^{l-2},d_1],c_1] - [[v_2^{l-2},c_1],d_1] = (l-2)[v_1^l,c_1] - \frac{1}{2}(l-2)[v_1^{l-1},d_1]  , 
$$
where we have calculated $[v_2^{l-2},d_1]$ and $[v_2^{l-2},c_1]$ in $\po$ as we've discussed before the formula.
Now we also express $d_1 = [c_2,c_1]$ and get:
$$
[v_1^{l-1},d_1] = [[v_1^{l-1},c_2],c_1] - [[v_1^{l-1},c_1],c_2] = -\frac{1}{2}(l-1)[v_1^l,c_1]  .
$$
So we conclude that:
$$
0 = (l-2)(l+3)[v_1^l,c_1] = (l-2)(l+3)\xi_{l-1}(v_1^{l}\otimes c_1)  ,
$$
which is proportional to the image of the highest weight vector of $V_{l+3}$, and since $l \ge 6$, it follows that it is indeed zero.

We use a similar method for two other highest weight vectors. Starting with:
$$
0 = [v_1^{l-1},[c_1,c_4]-[c_2,c_3]]  ,
$$
we get:
$$
0=6[v_3^l,c_1]-4(l-1)[v_2^l,c_2]+ (l-1)l[v_1^l;c_3] = \xi_{l-1}(6v_3^l\otimes c_1 - 4(l-1)v_2^{l}\otimes c_2 + l(l-1)v_1^l \otimes c_3)  , 
$$
which is the highest weight vector of $V_{l-1}$. 

To deal with the highest weight vector of $V_{l-3}$ we start with:
$$
0=[v_2^{l-2},[d_4,c_1]-2[d_3,c_2] + 3[d_2,c_3]-4[d_1,c_4]].
$$
After a similar calculation we get:
\begin{align} \label{lasteq}
0 = \xi_{l-1}(12(44-16l)v_4^l \otimes c_1+12(l-2)(11l-35)v_3^l \otimes c_2 +\\ 
+12(l-2)(l-1)(13-3l)v_2^l \otimes c_3 + l(l-1)(l-2)(2l-34)v_1^l \otimes c_4)  , \nonumber
\end{align}
denote this element  of $\mathfrak{l}_{l-2} \otimes \mathfrak{n}_1$ in brackets by $\alpha_1$. We know that $\alpha_1$ is of  $\mathfrak{sl}_2$-weight $l-3$, and we know that it belongs to the kernel of $\pi'\circ\xi_{l-1}$. Hence it is the element of the submodule isomorphic to $V_{l+3}\oplus V_{l-1}\oplus V_{l-3}$. Denote by $\alpha_2$ the result of the action by $e^3$ on the highest weight vector of $V_{l+3}$ and by $\alpha_3$ the result of the action of $e$ on the highest weight vector of $V_{l+1}$. We have:
$$
\alpha_2 = v_4^l \otimes c_1 + 3v_3^l \otimes c_2 + 3v_2^l \otimes c_3 + v_1^l \otimes c_4  ,
$$
and:
$$
\alpha_3 = 6v_4^l \otimes c_1 + (10-4l)v_3^l \otimes c_2 + (l-1)(l-4)v_2^l \otimes c_3 + l(l-1)v_1^l \otimes c_4.
$$

Now if $\alpha_1$ is linearly independent of $\alpha_2$ and $\alpha_3$, then the highest weight vector of $V_{l-3}$ lies in the linear span of $\alpha_i$, and since $\xi_{l-1}(\alpha_i) = 0$ it follows that $\xi_{l-1}$ acts on the highest weight vector by zero.

But calculating the roots of the minors of the matrix given by the coordinates of $\alpha_i$, we see that the common roots are only $l=-1,-2,5$. So, since in our case $l\ge 6$, we are done, and $\pi'$ is an isomorphism in the degree $l-1$.

Now we can prove the base of induction, i.e., the degrees $2,3,4$.

Let us begin with $\mathfrak{l}_2$. We have $L(\mathfrak{n}_1) = \Lambda^2\mathfrak{n}_1 = \phi_2 \oplus \phi_1$ ($V_0\oplus V_4$ as $\mathfrak{sl}_2$-modules) and $I_2 = \phi_1$. So we see that $\mathfrak{l}_2=\phi_2 \simeq V_4$, which has the same dimension as $\mathfrak{n}_2$. So $\pi'$ must be an isomorphism in degree $2$. Note that it also follows that the minimal set of relations must contain $\phi_1$.

Now we deal with $\mathfrak{l}_3$. We have a surjective map 
$$\xi_{3}: \mathfrak{l}_2 \otimes \mathfrak{n}_1 \simeq V_7 \oplus V_5 \oplus V_3 \oplus V_1 \to \mathfrak{l}_3.$$ Here we have a part of $I_3$ generated by $I_2$, i.e., we have $\phi_1 \otimes \mathfrak{n}_1 \simeq V_3\to I_3$. But the consequence of this relation was exactly calculated by us in the general case, when we used that $0 = [v_1^{l-1},[c_1,c_4]-[c_2,c_3]]$. As was shown there, this leads to the conclusion that $V_3$ is in the kernel of $\xi_3$. However, we cannot kill anything else using only the relation $\phi_1$.
But $V_7 \oplus V_0$ are precisely $\psi_4$ and $\psi_1$, hence they lie in $I_3$  and in the kernel of $\xi_3$. So $\mathfrak{l}_3$ has the same dimension as $\mathfrak{n}_3$ and $\pi'$ is an isomorphism. Note that this also shows that the minimal set of relations must contain $\psi_1,\psi_4$.

To finish we need to consider $\mathfrak{l}_4$. As before, we have a surjective map $$\xi_{4}:\mathfrak{l}_3 \otimes \mathfrak{n}_1 \simeq V_8 \oplus V_6 \oplus V_4 \oplus V_2 \to \mathfrak{l}_4.$$ The general formulas from the induction step allow us to conclude that $\xi_4(V_8 \oplus V_4) = 0$. Now we need to deal with $V_2$. However, as we can see from the general formulas, $\alpha_1$ defined in Equation \ref{lasteq} becomes linearly dependent with $\alpha_2$ and $\alpha_3$ in degree $4$. Indeed, it turns out that $V_2$ does not belong to the ideal generated by $\psi_1,\psi_4$ and $\phi_0$.

We see that all we can generate by $\phi_0$ in degree $4$ is given by $\phi_0 \otimes \Lambda^2\mathfrak{n}_1 \simeq V_0 \oplus V_4$, \linebreak so it does not contain anything isomorphic to $V_2$. All we can generate by $\psi_4$ is \linebreak $\psi_4 \otimes \mathfrak{n}_1 \simeq V_{10} \oplus V_8 \oplus V_6 \oplus V_4$, so it does not contain anything isomorphic to $V_2$. So the only chance to kill $V_2$ is $\psi_1 \otimes \mathfrak{n}_1 = V_4 \oplus V_2 $. But using our calculation (and similar ones) it follows that this doesn't kill $V_2$ in $\mathfrak{l}_3 \otimes \mathfrak{n}_1$. 

But the relation $\chi_1$ takes care of it. So it follows both that $\mathfrak{l}_4$ is isomorphic to $\mathfrak{n}_4$ under $\pi'$ and that the minimal set of relations must contain $\chi_1$.

\end{proof}

\bibliographystyle{alpha}
\bibliography{biblio}

\end{document}